\documentclass[12pt]{amsart}
\usepackage{amsmath,amsthm,amssymb,xypic,mathtools,enumerate,verbatim,hyperref,bm}  
\usepackage[dvipsnames]{xcolor}
\usepackage[dvips]{graphicx}

\allowdisplaybreaks[1]

\numberwithin{equation}{section}

\theoremstyle{plain}
\newtheorem{theorem}{Theorem}[section]
\newtheorem{proposition}[theorem]{Proposition}         
\newtheorem{corollary}[theorem]{Corollary} 
\newtheorem{lemma}[theorem]{Lemma}

\newtheorem{openproblem}[theorem]{Open Problem}

\theoremstyle{definition}
\newtheorem{definition}[theorem]{Definition}
\newtheorem{example}[theorem]{Example}
\newtheorem{remark}[theorem]{Remark}

\newcommand{\C}{\mathbb C}

\newcommand{\M}{\mathbf M}

\DeclareMathOperator{\rank}{rank}
\DeclareMathOperator{\codim}{codim}
\DeclareMathOperator{\Hess}{Hess}
\DeclareMathOperator{\GL}{GL}
\DeclareMathOperator{\Flags}{Fl}
 
\newcommand{\Sn}{\mathfrak{S}_n}

\newcommand{\HC}{Hessenberg complement conjugated by $w$}
\newcommand{\hc}[1]{H^c_{#1}}
\newcommand{\Jac}{\textbf{J}}
\newcommand{\Jdl}[1]{\Jac_{#1}^{L}}
\newcommand{\Jdc}[1]{\Jac_{#1,#1+1}}
\newcommand{\Jdr}[1]{\Jac_{#1}^{R}}
\newcommand{\SC}[1]{X^{\circ}_{#1}}

\makeatletter

\newcommand{\Rmnum}[1]
{\expandafter\@slowromancap\romannumeral #1@}
\makeatother

\begin{document}

\title[On singularity and normality of Hessenberg varieties]{On singularity and normality of \\regular nilpotent Hessenberg varieties}

\author {Hiraku Abe}
\address{Faculty of Science, Department of Applied Mathematics, Okayama University of
Science, 1-1 Ridai-cho, Kita-ku, Okayama, 700-0005, Japan}
\email{hirakuabe@globe.ocn.ne.jp}

\author {Erik Insko}
\address{Department of Mathematics, Florida Gulf Coast University, 10501 FGCU Boulevard South, Fort Myers, FL 33965, United States of America}
\email{einsko@fgcu.edu}

\begin{abstract}
Regular nilpotent Hessenberg varieties form an important family of subvarieties of the flag variety, which are often singular and sometimes not normal varieties.
Like Schubert varieties, they contain distinguished points called permutation flags.
In this paper, we give a combinatorial characterization for a permutation flag of a regular nilpotent Hessenberg variety to be a singular point. 
We also apply this result to characterize regular nilpotent Hessenberg varieties which are normal algebraic varieties.
\end{abstract}

\maketitle

\section{Introduction}\label{sec: intro}
Hessenberg varieties are closed subvarieties of the flag variety originally defined and studied by De~Mari, Procesi, and Shayman ~\cite{DMS88,DMPS92}.   
Each Hessenberg variety $\Hess(X,h)$ is defined by two parameters; a linear operator (or $n\times n$ matrix) $X$ and a Hessenberg function $h\colon [n] \rightarrow [n]$ satisfying $h(i) \leq h(i+1)$ for all $i \in [n-1]$, where we use the notation $[n]\coloneqq\{1,2,3, \ldots, n \}$ throughout this paper.

In this manuscript, we focus on the case when $X$ is a regular nilpotent operator $N$ and $h(i) \geq i+1$ for all $i \in [n-1]$. 
The geometric properties of regular nilpotent Hessenberg varieties have been studied extensively, in for example~
\cite{AFZ20, ADGH18, AHKZ21, Bal17, I15, IT16, ITW20, IY12, Kos96, Pre13, Pre16, ST06, Tym06}. 
In 2018,  Abe, Dedieu, Galetto, Harada determined an explicit list of generators for the local defining ideals of regular nilpotent Hessenberg varieties and used those generators to prove that regular nilpotent Hessenberg varieties are local complete intersections \cite{ADGH18}(Proposition~3.7 and Corollary~3.17).   
In this paper, we analyze those local defining ideals to explicitly determine which permutation flags in regular nilpotent Hessenberg varieties are singular points. 
Since the normality of a local complete intersection can be studied by the codimension of its singular locus, we apply our result to identify Hessenberg-Schubert cells of codimension one that consist entirely of singular points and obtain a characterization of regular nilpotent Hessenberg varieties which are normal algebraic varieties.

While singular loci of Schubert varieties have been studied extensively and fruitfully \cite{BL00}, relatively little is known about singular loci of Hessenberg varieties.  
DeMari, Procesi, and Shayman noted that all regular semisimple Hessenberg varieties are nonsingular \cite{DMPS92}.  In 2019, Insko and Precup showed if $X$ is semisimple (not necessarily regular) and $h$ is the Hessenberg function with $h(i)=i+1$ for $1 \leq i <n$, then the irreducible components of $\Hess(X,h)$ are nonsingular, and the singular locus of $\Hess(X,h)$ consists of the intersections of these irreducible components \cite{IP19}.  

In regards to regular nilpotent Hessenberg varieties, little is known outside of the realm of Peterson varieties; they form an important family of regular nilpotent Hessenberg varieties named after Dale Peterson who was the first to study them and their relation to the quantum cohomology of the flag variety (cf.\ \cite{Kos96,Pet97,Rie03}).
Kostant proved that almost all Peterson varieties are not normal by demonstrating that they have singularities in codimension 1 \cite{Kos96}.   In 2012, Insko and Yong extended Kostant's work to explicitly identify the singular locus of the Peterson varieties \cite{IY12}.   Their proof identifies the singular permutation flags in the Peterson variety and proves that if a permutation flag in a Peterson-Schubert cell is singular then every point in the cell is singular.     Thus the singular locus of a Peterson variety is the union of the Peterson-Schubert cells which contain a singular permutation flag.    

In this work, we generalize parts of the work of Insko and Yong and Kostant, by explicitly describing the singular permutation flags in each regular nilpotent Hessenberg variety and describing explicitly which regular nilpotent Hessenberg varieties are normal algebraic varieties.   We also identify examples of regular nilpotent Hessenberg varieties that contain a Hessenberg-Schubert cell which is singular at its permutation flag and nonsingular at most of the points within the cell.

\subsection{Preview of our main results}\label{subsec: preview}

A Hessenberg function $h\colon [n]\rightarrow [n]$ corresponds to a configuration of shaded boxes on a square grid of size $n \times n$ which consists of boxes in the $i$-th row and the $j$-th column satisfying $i \leq h(j)$ for $i, j\in[n]$. 
We frequently denote a Hessenberg function by listing its values in a sequence, $h=(h(1),h(2),\ldots,h(n))$.
For $h=(3,3,4,5,5)$, the configuration of shaded boxes is given by the following picture.
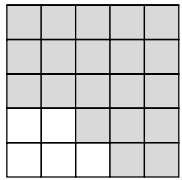
\begin{figure}[h]
\begin{center}
%WinTpicVersion4.32a
{\unitlength 0.1in%
\begin{picture}(9.0000,9.0000)(12.0000,-17.0000)%
% BOX 2 0 1 0 Black Black  
% 2 1200 800 1380 980
% 
\special{pn 0}%
\special{sh 0.150}%
\special{pa 1200 800}%
\special{pa 1380 800}%
\special{pa 1380 980}%
\special{pa 1200 980}%
\special{pa 1200 800}%
\special{ip}%
\special{pn 8}%
\special{pa 1200 800}%
\special{pa 1380 800}%
\special{pa 1380 980}%
\special{pa 1200 980}%
\special{pa 1200 800}%
\special{pa 1380 800}%
\special{fp}%
% BOX 2 0 1 0 Black Black  
% 2 1380 800 1560 980
% 
\special{pn 0}%
\special{sh 0.150}%
\special{pa 1380 800}%
\special{pa 1560 800}%
\special{pa 1560 980}%
\special{pa 1380 980}%
\special{pa 1380 800}%
\special{ip}%
\special{pn 8}%
\special{pa 1380 800}%
\special{pa 1560 800}%
\special{pa 1560 980}%
\special{pa 1380 980}%
\special{pa 1380 800}%
\special{pa 1560 800}%
\special{fp}%
% BOX 2 0 1 0 Black Black  
% 2 1560 800 1740 980
% 
\special{pn 0}%
\special{sh 0.150}%
\special{pa 1560 800}%
\special{pa 1740 800}%
\special{pa 1740 980}%
\special{pa 1560 980}%
\special{pa 1560 800}%
\special{ip}%
\special{pn 8}%
\special{pa 1560 800}%
\special{pa 1740 800}%
\special{pa 1740 980}%
\special{pa 1560 980}%
\special{pa 1560 800}%
\special{pa 1740 800}%
\special{fp}%
% BOX 2 0 1 0 Black Black  
% 2 1740 800 1920 980
% 
\special{pn 0}%
\special{sh 0.150}%
\special{pa 1740 800}%
\special{pa 1920 800}%
\special{pa 1920 980}%
\special{pa 1740 980}%
\special{pa 1740 800}%
\special{ip}%
\special{pn 8}%
\special{pa 1740 800}%
\special{pa 1920 800}%
\special{pa 1920 980}%
\special{pa 1740 980}%
\special{pa 1740 800}%
\special{pa 1920 800}%
\special{fp}%
% BOX 2 0 1 0 Black Black  
% 2 1920 800 2100 980
% 
\special{pn 0}%
\special{sh 0.150}%
\special{pa 1920 800}%
\special{pa 2100 800}%
\special{pa 2100 980}%
\special{pa 1920 980}%
\special{pa 1920 800}%
\special{ip}%
\special{pn 8}%
\special{pa 1920 800}%
\special{pa 2100 800}%
\special{pa 2100 980}%
\special{pa 1920 980}%
\special{pa 1920 800}%
\special{pa 2100 800}%
\special{fp}%
% BOX 2 0 1 0 Black Black  
% 2 1200 980 1380 1160
% 
\special{pn 0}%
\special{sh 0.150}%
\special{pa 1200 980}%
\special{pa 1380 980}%
\special{pa 1380 1160}%
\special{pa 1200 1160}%
\special{pa 1200 980}%
\special{ip}%
\special{pn 8}%
\special{pa 1200 980}%
\special{pa 1380 980}%
\special{pa 1380 1160}%
\special{pa 1200 1160}%
\special{pa 1200 980}%
\special{pa 1380 980}%
\special{fp}%
% BOX 2 0 1 0 Black Black  
% 2 1380 980 1560 1160
% 
\special{pn 0}%
\special{sh 0.150}%
\special{pa 1380 980}%
\special{pa 1560 980}%
\special{pa 1560 1160}%
\special{pa 1380 1160}%
\special{pa 1380 980}%
\special{ip}%
\special{pn 8}%
\special{pa 1380 980}%
\special{pa 1560 980}%
\special{pa 1560 1160}%
\special{pa 1380 1160}%
\special{pa 1380 980}%
\special{pa 1560 980}%
\special{fp}%
% BOX 2 0 1 0 Black Black  
% 2 1560 980 1740 1160
% 
\special{pn 0}%
\special{sh 0.150}%
\special{pa 1560 980}%
\special{pa 1740 980}%
\special{pa 1740 1160}%
\special{pa 1560 1160}%
\special{pa 1560 980}%
\special{ip}%
\special{pn 8}%
\special{pa 1560 980}%
\special{pa 1740 980}%
\special{pa 1740 1160}%
\special{pa 1560 1160}%
\special{pa 1560 980}%
\special{pa 1740 980}%
\special{fp}%
% BOX 2 0 1 0 Black Black  
% 2 1740 980 1920 1160
% 
\special{pn 0}%
\special{sh 0.150}%
\special{pa 1740 980}%
\special{pa 1920 980}%
\special{pa 1920 1160}%
\special{pa 1740 1160}%
\special{pa 1740 980}%
\special{ip}%
\special{pn 8}%
\special{pa 1740 980}%
\special{pa 1920 980}%
\special{pa 1920 1160}%
\special{pa 1740 1160}%
\special{pa 1740 980}%
\special{pa 1920 980}%
\special{fp}%
% BOX 2 0 1 0 Black Black  
% 2 1920 980 2100 1160
% 
\special{pn 0}%
\special{sh 0.150}%
\special{pa 1920 980}%
\special{pa 2100 980}%
\special{pa 2100 1160}%
\special{pa 1920 1160}%
\special{pa 1920 980}%
\special{ip}%
\special{pn 8}%
\special{pa 1920 980}%
\special{pa 2100 980}%
\special{pa 2100 1160}%
\special{pa 1920 1160}%
\special{pa 1920 980}%
\special{pa 2100 980}%
\special{fp}%
% BOX 2 0 1 0 Black Black  
% 2 1200 1160 1380 1340
% 
\special{pn 0}%
\special{sh 0.150}%
\special{pa 1200 1160}%
\special{pa 1380 1160}%
\special{pa 1380 1340}%
\special{pa 1200 1340}%
\special{pa 1200 1160}%
\special{ip}%
\special{pn 8}%
\special{pa 1200 1160}%
\special{pa 1380 1160}%
\special{pa 1380 1340}%
\special{pa 1200 1340}%
\special{pa 1200 1160}%
\special{pa 1380 1160}%
\special{fp}%
% BOX 2 0 1 0 Black Black  
% 2 1380 1160 1560 1340
% 
\special{pn 0}%
\special{sh 0.150}%
\special{pa 1380 1160}%
\special{pa 1560 1160}%
\special{pa 1560 1340}%
\special{pa 1380 1340}%
\special{pa 1380 1160}%
\special{ip}%
\special{pn 8}%
\special{pa 1380 1160}%
\special{pa 1560 1160}%
\special{pa 1560 1340}%
\special{pa 1380 1340}%
\special{pa 1380 1160}%
\special{pa 1560 1160}%
\special{fp}%
% BOX 2 0 1 0 Black Black  
% 2 1560 1160 1740 1340
% 
\special{pn 0}%
\special{sh 0.150}%
\special{pa 1560 1160}%
\special{pa 1740 1160}%
\special{pa 1740 1340}%
\special{pa 1560 1340}%
\special{pa 1560 1160}%
\special{ip}%
\special{pn 8}%
\special{pa 1560 1160}%
\special{pa 1740 1160}%
\special{pa 1740 1340}%
\special{pa 1560 1340}%
\special{pa 1560 1160}%
\special{pa 1740 1160}%
\special{fp}%
% BOX 2 0 1 0 Black Black  
% 2 1740 1160 1920 1340
% 
\special{pn 0}%
\special{sh 0.150}%
\special{pa 1740 1160}%
\special{pa 1920 1160}%
\special{pa 1920 1340}%
\special{pa 1740 1340}%
\special{pa 1740 1160}%
\special{ip}%
\special{pn 8}%
\special{pa 1740 1160}%
\special{pa 1920 1160}%
\special{pa 1920 1340}%
\special{pa 1740 1340}%
\special{pa 1740 1160}%
\special{pa 1920 1160}%
\special{fp}%
% BOX 2 0 1 0 Black Black  
% 2 1920 1160 2100 1340
% 
\special{pn 0}%
\special{sh 0.150}%
\special{pa 1920 1160}%
\special{pa 2100 1160}%
\special{pa 2100 1340}%
\special{pa 1920 1340}%
\special{pa 1920 1160}%
\special{ip}%
\special{pn 8}%
\special{pa 1920 1160}%
\special{pa 2100 1160}%
\special{pa 2100 1340}%
\special{pa 1920 1340}%
\special{pa 1920 1160}%
\special{pa 2100 1160}%
\special{fp}%
% BOX 2 0 3 0 Black White  
% 2 1200 1340 1380 1520
% 
\special{pn 8}%
\special{pa 1200 1340}%
\special{pa 1380 1340}%
\special{pa 1380 1520}%
\special{pa 1200 1520}%
\special{pa 1200 1340}%
\special{pa 1380 1340}%
\special{fp}%
% BOX 2 0 3 0 Black White  
% 2 1380 1340 1560 1520
% 
\special{pn 8}%
\special{pa 1380 1340}%
\special{pa 1560 1340}%
\special{pa 1560 1520}%
\special{pa 1380 1520}%
\special{pa 1380 1340}%
\special{pa 1560 1340}%
\special{fp}%
% BOX 2 0 1 0 Black Black  
% 2 1560 1340 1740 1520
% 
\special{pn 0}%
\special{sh 0.150}%
\special{pa 1560 1340}%
\special{pa 1740 1340}%
\special{pa 1740 1520}%
\special{pa 1560 1520}%
\special{pa 1560 1340}%
\special{ip}%
\special{pn 8}%
\special{pa 1560 1340}%
\special{pa 1740 1340}%
\special{pa 1740 1520}%
\special{pa 1560 1520}%
\special{pa 1560 1340}%
\special{pa 1740 1340}%
\special{fp}%
% BOX 2 0 1 0 Black Black  
% 2 1740 1340 1920 1520
% 
\special{pn 0}%
\special{sh 0.150}%
\special{pa 1740 1340}%
\special{pa 1920 1340}%
\special{pa 1920 1520}%
\special{pa 1740 1520}%
\special{pa 1740 1340}%
\special{ip}%
\special{pn 8}%
\special{pa 1740 1340}%
\special{pa 1920 1340}%
\special{pa 1920 1520}%
\special{pa 1740 1520}%
\special{pa 1740 1340}%
\special{pa 1920 1340}%
\special{fp}%
% BOX 2 0 1 0 Black Black  
% 2 1920 1340 2100 1520
% 
\special{pn 0}%
\special{sh 0.150}%
\special{pa 1920 1340}%
\special{pa 2100 1340}%
\special{pa 2100 1520}%
\special{pa 1920 1520}%
\special{pa 1920 1340}%
\special{ip}%
\special{pn 8}%
\special{pa 1920 1340}%
\special{pa 2100 1340}%
\special{pa 2100 1520}%
\special{pa 1920 1520}%
\special{pa 1920 1340}%
\special{pa 2100 1340}%
\special{fp}%
% BOX 2 0 3 0 Black White  
% 2 1200 1520 1380 1700
% 
\special{pn 8}%
\special{pa 1200 1520}%
\special{pa 1380 1520}%
\special{pa 1380 1700}%
\special{pa 1200 1700}%
\special{pa 1200 1520}%
\special{pa 1380 1520}%
\special{fp}%
% BOX 2 0 3 0 Black White  
% 2 1380 1520 1560 1700
% 
\special{pn 8}%
\special{pa 1380 1520}%
\special{pa 1560 1520}%
\special{pa 1560 1700}%
\special{pa 1380 1700}%
\special{pa 1380 1520}%
\special{pa 1560 1520}%
\special{fp}%
% BOX 2 0 3 0 Black White  
% 2 1560 1520 1740 1700
% 
\special{pn 8}%
\special{pa 1560 1520}%
\special{pa 1740 1520}%
\special{pa 1740 1700}%
\special{pa 1560 1700}%
\special{pa 1560 1520}%
\special{pa 1740 1520}%
\special{fp}%
% BOX 2 0 1 0 Black Black  
% 2 1740 1520 1920 1700
% 
\special{pn 0}%
\special{sh 0.150}%
\special{pa 1740 1520}%
\special{pa 1920 1520}%
\special{pa 1920 1700}%
\special{pa 1740 1700}%
\special{pa 1740 1520}%
\special{ip}%
\special{pn 8}%
\special{pa 1740 1520}%
\special{pa 1920 1520}%
\special{pa 1920 1700}%
\special{pa 1740 1700}%
\special{pa 1740 1520}%
\special{pa 1920 1520}%
\special{fp}%
% BOX 2 0 1 0 Black Black  
% 2 1920 1520 2100 1700
% 
\special{pn 0}%
\special{sh 0.150}%
\special{pa 1920 1520}%
\special{pa 2100 1520}%
\special{pa 2100 1700}%
\special{pa 1920 1700}%
\special{pa 1920 1520}%
\special{ip}%
\special{pn 8}%
\special{pa 1920 1520}%
\special{pa 2100 1520}%
\special{pa 2100 1700}%
\special{pa 1920 1700}%
\special{pa 1920 1520}%
\special{pa 2100 1520}%
\special{fp}%
\end{picture}}%
\vspace{-15pt}
\end{center}
\end{figure}\\

To state our results, we prefer to consider the collection of the \textit{unshaded} boxes of the configuration as follows. 
For a permutation $w\in\Sn$, we set
\begin{align}\label{eq: def of Hw}
 \hc{w}
 \coloneqq \{(w(i),w(j))\in[n]\times[n]\mid  i< h(j)\} .
\end{align}
We call $\hc{w}$ the \textbf{\HC}.
For instance, if $w=32154$ in one-line notation, then $\hc{w}$ is depicted as the collection of the unshaded boxes in the following picture.
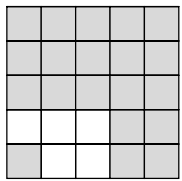
\begin{figure}[h]
\begin{center}
%WinTpicVersion4.32a
{\unitlength 0.1in%
\begin{picture}(9.0000,9.0000)(12.0000,-17.0000)%
% BOX 2 0 1 0 Black Black  
% 2 1200 800 1380 980
% 
\special{pn 0}%
\special{sh 0.150}%
\special{pa 1200 800}%
\special{pa 1380 800}%
\special{pa 1380 980}%
\special{pa 1200 980}%
\special{pa 1200 800}%
\special{ip}%
\special{pn 8}%
\special{pa 1200 800}%
\special{pa 1380 800}%
\special{pa 1380 980}%
\special{pa 1200 980}%
\special{pa 1200 800}%
\special{pa 1380 800}%
\special{fp}%
% BOX 2 0 1 0 Black Black  
% 2 1380 800 1560 980
% 
\special{pn 0}%
\special{sh 0.150}%
\special{pa 1380 800}%
\special{pa 1560 800}%
\special{pa 1560 980}%
\special{pa 1380 980}%
\special{pa 1380 800}%
\special{ip}%
\special{pn 8}%
\special{pa 1380 800}%
\special{pa 1560 800}%
\special{pa 1560 980}%
\special{pa 1380 980}%
\special{pa 1380 800}%
\special{pa 1560 800}%
\special{fp}%
% BOX 2 0 1 0 Black Black  
% 2 1560 800 1740 980
% 
\special{pn 0}%
\special{sh 0.150}%
\special{pa 1560 800}%
\special{pa 1740 800}%
\special{pa 1740 980}%
\special{pa 1560 980}%
\special{pa 1560 800}%
\special{ip}%
\special{pn 8}%
\special{pa 1560 800}%
\special{pa 1740 800}%
\special{pa 1740 980}%
\special{pa 1560 980}%
\special{pa 1560 800}%
\special{pa 1740 800}%
\special{fp}%
% BOX 2 0 1 0 Black Black  
% 2 1740 800 1920 980
% 
\special{pn 0}%
\special{sh 0.150}%
\special{pa 1740 800}%
\special{pa 1920 800}%
\special{pa 1920 980}%
\special{pa 1740 980}%
\special{pa 1740 800}%
\special{ip}%
\special{pn 8}%
\special{pa 1740 800}%
\special{pa 1920 800}%
\special{pa 1920 980}%
\special{pa 1740 980}%
\special{pa 1740 800}%
\special{pa 1920 800}%
\special{fp}%
% BOX 2 0 1 0 Black Black  
% 2 1920 800 2100 980
% 
\special{pn 0}%
\special{sh 0.150}%
\special{pa 1920 800}%
\special{pa 2100 800}%
\special{pa 2100 980}%
\special{pa 1920 980}%
\special{pa 1920 800}%
\special{ip}%
\special{pn 8}%
\special{pa 1920 800}%
\special{pa 2100 800}%
\special{pa 2100 980}%
\special{pa 1920 980}%
\special{pa 1920 800}%
\special{pa 2100 800}%
\special{fp}%
% BOX 2 0 1 0 Black Black  
% 2 1200 980 1380 1160
% 
\special{pn 0}%
\special{sh 0.150}%
\special{pa 1200 980}%
\special{pa 1380 980}%
\special{pa 1380 1160}%
\special{pa 1200 1160}%
\special{pa 1200 980}%
\special{ip}%
\special{pn 8}%
\special{pa 1200 980}%
\special{pa 1380 980}%
\special{pa 1380 1160}%
\special{pa 1200 1160}%
\special{pa 1200 980}%
\special{pa 1380 980}%
\special{fp}%
% BOX 2 0 1 0 Black Black  
% 2 1380 980 1560 1160
% 
\special{pn 0}%
\special{sh 0.150}%
\special{pa 1380 980}%
\special{pa 1560 980}%
\special{pa 1560 1160}%
\special{pa 1380 1160}%
\special{pa 1380 980}%
\special{ip}%
\special{pn 8}%
\special{pa 1380 980}%
\special{pa 1560 980}%
\special{pa 1560 1160}%
\special{pa 1380 1160}%
\special{pa 1380 980}%
\special{pa 1560 980}%
\special{fp}%
% BOX 2 0 1 0 Black Black  
% 2 1560 980 1740 1160
% 
\special{pn 0}%
\special{sh 0.150}%
\special{pa 1560 980}%
\special{pa 1740 980}%
\special{pa 1740 1160}%
\special{pa 1560 1160}%
\special{pa 1560 980}%
\special{ip}%
\special{pn 8}%
\special{pa 1560 980}%
\special{pa 1740 980}%
\special{pa 1740 1160}%
\special{pa 1560 1160}%
\special{pa 1560 980}%
\special{pa 1740 980}%
\special{fp}%
% BOX 2 0 1 0 Black Black  
% 2 1740 980 1920 1160
% 
\special{pn 0}%
\special{sh 0.150}%
\special{pa 1740 980}%
\special{pa 1920 980}%
\special{pa 1920 1160}%
\special{pa 1740 1160}%
\special{pa 1740 980}%
\special{ip}%
\special{pn 8}%
\special{pa 1740 980}%
\special{pa 1920 980}%
\special{pa 1920 1160}%
\special{pa 1740 1160}%
\special{pa 1740 980}%
\special{pa 1920 980}%
\special{fp}%
% BOX 2 0 1 0 Black Black  
% 2 1920 980 2100 1160
% 
\special{pn 0}%
\special{sh 0.150}%
\special{pa 1920 980}%
\special{pa 2100 980}%
\special{pa 2100 1160}%
\special{pa 1920 1160}%
\special{pa 1920 980}%
\special{ip}%
\special{pn 8}%
\special{pa 1920 980}%
\special{pa 2100 980}%
\special{pa 2100 1160}%
\special{pa 1920 1160}%
\special{pa 1920 980}%
\special{pa 2100 980}%
\special{fp}%
% BOX 2 0 1 0 Black Black  
% 2 1200 1160 1380 1340
% 
\special{pn 0}%
\special{sh 0.150}%
\special{pa 1200 1160}%
\special{pa 1380 1160}%
\special{pa 1380 1340}%
\special{pa 1200 1340}%
\special{pa 1200 1160}%
\special{ip}%
\special{pn 8}%
\special{pa 1200 1160}%
\special{pa 1380 1160}%
\special{pa 1380 1340}%
\special{pa 1200 1340}%
\special{pa 1200 1160}%
\special{pa 1380 1160}%
\special{fp}%
% BOX 2 0 1 0 Black Black  
% 2 1380 1160 1560 1340
% 
\special{pn 0}%
\special{sh 0.150}%
\special{pa 1380 1160}%
\special{pa 1560 1160}%
\special{pa 1560 1340}%
\special{pa 1380 1340}%
\special{pa 1380 1160}%
\special{ip}%
\special{pn 8}%
\special{pa 1380 1160}%
\special{pa 1560 1160}%
\special{pa 1560 1340}%
\special{pa 1380 1340}%
\special{pa 1380 1160}%
\special{pa 1560 1160}%
\special{fp}%
% BOX 2 0 1 0 Black Black  
% 2 1560 1160 1740 1340
% 
\special{pn 0}%
\special{sh 0.150}%
\special{pa 1560 1160}%
\special{pa 1740 1160}%
\special{pa 1740 1340}%
\special{pa 1560 1340}%
\special{pa 1560 1160}%
\special{ip}%
\special{pn 8}%
\special{pa 1560 1160}%
\special{pa 1740 1160}%
\special{pa 1740 1340}%
\special{pa 1560 1340}%
\special{pa 1560 1160}%
\special{pa 1740 1160}%
\special{fp}%
% BOX 2 0 1 0 Black Black  
% 2 1740 1160 1920 1340
% 
\special{pn 0}%
\special{sh 0.150}%
\special{pa 1740 1160}%
\special{pa 1920 1160}%
\special{pa 1920 1340}%
\special{pa 1740 1340}%
\special{pa 1740 1160}%
\special{ip}%
\special{pn 8}%
\special{pa 1740 1160}%
\special{pa 1920 1160}%
\special{pa 1920 1340}%
\special{pa 1740 1340}%
\special{pa 1740 1160}%
\special{pa 1920 1160}%
\special{fp}%
% BOX 2 0 1 0 Black Black  
% 2 1920 1160 2100 1340
% 
\special{pn 0}%
\special{sh 0.150}%
\special{pa 1920 1160}%
\special{pa 2100 1160}%
\special{pa 2100 1340}%
\special{pa 1920 1340}%
\special{pa 1920 1160}%
\special{ip}%
\special{pn 8}%
\special{pa 1920 1160}%
\special{pa 2100 1160}%
\special{pa 2100 1340}%
\special{pa 1920 1340}%
\special{pa 1920 1160}%
\special{pa 2100 1160}%
\special{fp}%
% BOX 2 0 3 0 Black White  
% 2 1200 1340 1380 1520
% 
\special{pn 8}%
\special{pa 1200 1340}%
\special{pa 1380 1340}%
\special{pa 1380 1520}%
\special{pa 1200 1520}%
\special{pa 1200 1340}%
\special{pa 1380 1340}%
\special{fp}%
% BOX 2 0 3 0 Black White  
% 2 1380 1340 1560 1520
% 
\special{pn 8}%
\special{pa 1380 1340}%
\special{pa 1560 1340}%
\special{pa 1560 1520}%
\special{pa 1380 1520}%
\special{pa 1380 1340}%
\special{pa 1560 1340}%
\special{fp}%
% BOX 2 0 3 0 Black Black  
% 2 1560 1340 1740 1520
% 
\special{pn 8}%
\special{pa 1560 1340}%
\special{pa 1740 1340}%
\special{pa 1740 1520}%
\special{pa 1560 1520}%
\special{pa 1560 1340}%
\special{pa 1740 1340}%
\special{fp}%
% BOX 2 0 1 0 Black Black  
% 2 1740 1340 1920 1520
% 
\special{pn 0}%
\special{sh 0.150}%
\special{pa 1740 1340}%
\special{pa 1920 1340}%
\special{pa 1920 1520}%
\special{pa 1740 1520}%
\special{pa 1740 1340}%
\special{ip}%
\special{pn 8}%
\special{pa 1740 1340}%
\special{pa 1920 1340}%
\special{pa 1920 1520}%
\special{pa 1740 1520}%
\special{pa 1740 1340}%
\special{pa 1920 1340}%
\special{fp}%
% BOX 2 0 1 0 Black Black  
% 2 1920 1340 2100 1520
% 
\special{pn 0}%
\special{sh 0.150}%
\special{pa 1920 1340}%
\special{pa 2100 1340}%
\special{pa 2100 1520}%
\special{pa 1920 1520}%
\special{pa 1920 1340}%
\special{ip}%
\special{pn 8}%
\special{pa 1920 1340}%
\special{pa 2100 1340}%
\special{pa 2100 1520}%
\special{pa 1920 1520}%
\special{pa 1920 1340}%
\special{pa 2100 1340}%
\special{fp}%
% BOX 2 0 1 0 Black Black  
% 2 1200 1520 1380 1700
% 
\special{pn 0}%
\special{sh 0.150}%
\special{pa 1200 1520}%
\special{pa 1380 1520}%
\special{pa 1380 1700}%
\special{pa 1200 1700}%
\special{pa 1200 1520}%
\special{ip}%
\special{pn 8}%
\special{pa 1200 1520}%
\special{pa 1380 1520}%
\special{pa 1380 1700}%
\special{pa 1200 1700}%
\special{pa 1200 1520}%
\special{pa 1380 1520}%
\special{fp}%
% BOX 2 0 3 0 Black White  
% 2 1380 1520 1560 1700
% 
\special{pn 8}%
\special{pa 1380 1520}%
\special{pa 1560 1520}%
\special{pa 1560 1700}%
\special{pa 1380 1700}%
\special{pa 1380 1520}%
\special{pa 1560 1520}%
\special{fp}%
% BOX 2 0 3 0 Black White  
% 2 1560 1520 1740 1700
% 
\special{pn 8}%
\special{pa 1560 1520}%
\special{pa 1740 1520}%
\special{pa 1740 1700}%
\special{pa 1560 1700}%
\special{pa 1560 1520}%
\special{pa 1740 1520}%
\special{fp}%
% BOX 2 0 1 0 Black Black  
% 2 1740 1520 1920 1700
% 
\special{pn 0}%
\special{sh 0.150}%
\special{pa 1740 1520}%
\special{pa 1920 1520}%
\special{pa 1920 1700}%
\special{pa 1740 1700}%
\special{pa 1740 1520}%
\special{ip}%
\special{pn 8}%
\special{pa 1740 1520}%
\special{pa 1920 1520}%
\special{pa 1920 1700}%
\special{pa 1740 1700}%
\special{pa 1740 1520}%
\special{pa 1920 1520}%
\special{fp}%
% BOX 2 0 1 0 Black Black  
% 2 1920 1520 2100 1700
% 
\special{pn 0}%
\special{sh 0.150}%
\special{pa 1920 1520}%
\special{pa 2100 1520}%
\special{pa 2100 1700}%
\special{pa 1920 1700}%
\special{pa 1920 1520}%
\special{ip}%
\special{pn 8}%
\special{pa 1920 1520}%
\special{pa 2100 1520}%
\special{pa 2100 1700}%
\special{pa 1920 1700}%
\special{pa 1920 1520}%
\special{pa 2100 1520}%
\special{fp}%
\end{picture}}%
\end{center}\vspace{-15pt}
\end{figure}\\

A \textbf{lower diagonal full-string} \textbf{(of height $d-1$)} of $[n]\times[n]$ is a set of the form
\[
\{(d,1), (d+1,2),\ldots, (n,n-d+1)\in[n]\times[n]\}
\]
for some $2\le d\leq n$. 
The set of theblack dots in the following picture are an example of the lower diagonal full-string of height 2 for $n=5$.\\
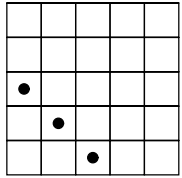
\[
%WinTpicVersion4.32a
{\unitlength 0.1in%
\begin{picture}(9.0000,9.0000)(12.0000,-17.0000)%
% BOX 2 0 3 0 Black Black  
% 2 1200 800 1380 980
% 
\special{pn 8}%
\special{pa 1200 800}%
\special{pa 1380 800}%
\special{pa 1380 980}%
\special{pa 1200 980}%
\special{pa 1200 800}%
\special{pa 1380 800}%
\special{fp}%
% BOX 2 0 3 0 Black Black  
% 2 1380 800 1560 980
% 
\special{pn 8}%
\special{pa 1380 800}%
\special{pa 1560 800}%
\special{pa 1560 980}%
\special{pa 1380 980}%
\special{pa 1380 800}%
\special{pa 1560 800}%
\special{fp}%
% BOX 2 0 3 0 Black Black  
% 2 1560 800 1740 980
% 
\special{pn 8}%
\special{pa 1560 800}%
\special{pa 1740 800}%
\special{pa 1740 980}%
\special{pa 1560 980}%
\special{pa 1560 800}%
\special{pa 1740 800}%
\special{fp}%
% BOX 2 0 3 0 Black Black  
% 2 1740 800 1920 980
% 
\special{pn 8}%
\special{pa 1740 800}%
\special{pa 1920 800}%
\special{pa 1920 980}%
\special{pa 1740 980}%
\special{pa 1740 800}%
\special{pa 1920 800}%
\special{fp}%
% BOX 2 0 3 0 Black Black  
% 2 1920 800 2100 980
% 
\special{pn 8}%
\special{pa 1920 800}%
\special{pa 2100 800}%
\special{pa 2100 980}%
\special{pa 1920 980}%
\special{pa 1920 800}%
\special{pa 2100 800}%
\special{fp}%
% BOX 2 0 3 0 Black Black  
% 2 1200 980 1380 1160
% 
\special{pn 8}%
\special{pa 1200 980}%
\special{pa 1380 980}%
\special{pa 1380 1160}%
\special{pa 1200 1160}%
\special{pa 1200 980}%
\special{pa 1380 980}%
\special{fp}%
% BOX 2 0 3 0 Black Black  
% 2 1380 980 1560 1160
% 
\special{pn 8}%
\special{pa 1380 980}%
\special{pa 1560 980}%
\special{pa 1560 1160}%
\special{pa 1380 1160}%
\special{pa 1380 980}%
\special{pa 1560 980}%
\special{fp}%
% BOX 2 0 3 0 Black Black  
% 2 1560 980 1740 1160
% 
\special{pn 8}%
\special{pa 1560 980}%
\special{pa 1740 980}%
\special{pa 1740 1160}%
\special{pa 1560 1160}%
\special{pa 1560 980}%
\special{pa 1740 980}%
\special{fp}%
% BOX 2 0 3 0 Black Black  
% 2 1740 980 1920 1160
% 
\special{pn 8}%
\special{pa 1740 980}%
\special{pa 1920 980}%
\special{pa 1920 1160}%
\special{pa 1740 1160}%
\special{pa 1740 980}%
\special{pa 1920 980}%
\special{fp}%
% BOX 2 0 3 0 Black Black  
% 2 1920 980 2100 1160
% 
\special{pn 8}%
\special{pa 1920 980}%
\special{pa 2100 980}%
\special{pa 2100 1160}%
\special{pa 1920 1160}%
\special{pa 1920 980}%
\special{pa 2100 980}%
\special{fp}%
% BOX 2 0 3 0 Black Black  
% 2 1200 1160 1380 1340
% 
\special{pn 8}%
\special{pa 1200 1160}%
\special{pa 1380 1160}%
\special{pa 1380 1340}%
\special{pa 1200 1340}%
\special{pa 1200 1160}%
\special{pa 1380 1160}%
\special{fp}%
% BOX 2 0 3 0 Black Black  
% 2 1380 1160 1560 1340
% 
\special{pn 8}%
\special{pa 1380 1160}%
\special{pa 1560 1160}%
\special{pa 1560 1340}%
\special{pa 1380 1340}%
\special{pa 1380 1160}%
\special{pa 1560 1160}%
\special{fp}%
% BOX 2 0 3 0 Black Black  
% 2 1560 1160 1740 1340
% 
\special{pn 8}%
\special{pa 1560 1160}%
\special{pa 1740 1160}%
\special{pa 1740 1340}%
\special{pa 1560 1340}%
\special{pa 1560 1160}%
\special{pa 1740 1160}%
\special{fp}%
% BOX 2 0 3 0 Black Black  
% 2 1740 1160 1920 1340
% 
\special{pn 8}%
\special{pa 1740 1160}%
\special{pa 1920 1160}%
\special{pa 1920 1340}%
\special{pa 1740 1340}%
\special{pa 1740 1160}%
\special{pa 1920 1160}%
\special{fp}%
% BOX 2 0 3 0 Black Black  
% 2 1920 1160 2100 1340
% 
\special{pn 8}%
\special{pa 1920 1160}%
\special{pa 2100 1160}%
\special{pa 2100 1340}%
\special{pa 1920 1340}%
\special{pa 1920 1160}%
\special{pa 2100 1160}%
\special{fp}%
% BOX 2 0 3 0 Black White  
% 2 1200 1340 1380 1520
% 
\special{pn 8}%
\special{pa 1200 1340}%
\special{pa 1380 1340}%
\special{pa 1380 1520}%
\special{pa 1200 1520}%
\special{pa 1200 1340}%
\special{pa 1380 1340}%
\special{fp}%
% BOX 2 0 3 0 Black White  
% 2 1380 1340 1560 1520
% 
\special{pn 8}%
\special{pa 1380 1340}%
\special{pa 1560 1340}%
\special{pa 1560 1520}%
\special{pa 1380 1520}%
\special{pa 1380 1340}%
\special{pa 1560 1340}%
\special{fp}%
% BOX 2 0 3 0 Black Black  
% 2 1560 1340 1740 1520
% 
\special{pn 8}%
\special{pa 1560 1340}%
\special{pa 1740 1340}%
\special{pa 1740 1520}%
\special{pa 1560 1520}%
\special{pa 1560 1340}%
\special{pa 1740 1340}%
\special{fp}%
% BOX 2 0 3 0 Black Black  
% 2 1740 1340 1920 1520
% 
\special{pn 8}%
\special{pa 1740 1340}%
\special{pa 1920 1340}%
\special{pa 1920 1520}%
\special{pa 1740 1520}%
\special{pa 1740 1340}%
\special{pa 1920 1340}%
\special{fp}%
% BOX 2 0 3 0 Black Black  
% 2 1920 1340 2100 1520
% 
\special{pn 8}%
\special{pa 1920 1340}%
\special{pa 2100 1340}%
\special{pa 2100 1520}%
\special{pa 1920 1520}%
\special{pa 1920 1340}%
\special{pa 2100 1340}%
\special{fp}%
% BOX 2 0 3 0 Black Black  
% 2 1200 1520 1380 1700
% 
\special{pn 8}%
\special{pa 1200 1520}%
\special{pa 1380 1520}%
\special{pa 1380 1700}%
\special{pa 1200 1700}%
\special{pa 1200 1520}%
\special{pa 1380 1520}%
\special{fp}%
% BOX 2 0 3 0 Black White  
% 2 1380 1520 1560 1700
% 
\special{pn 8}%
\special{pa 1380 1520}%
\special{pa 1560 1520}%
\special{pa 1560 1700}%
\special{pa 1380 1700}%
\special{pa 1380 1520}%
\special{pa 1560 1520}%
\special{fp}%
% BOX 2 0 3 0 Black White  
% 2 1560 1520 1740 1700
% 
\special{pn 8}%
\special{pa 1560 1520}%
\special{pa 1740 1520}%
\special{pa 1740 1700}%
\special{pa 1560 1700}%
\special{pa 1560 1520}%
\special{pa 1740 1520}%
\special{fp}%
% BOX 2 0 3 0 Black Black  
% 2 1740 1520 1920 1700
% 
\special{pn 8}%
\special{pa 1740 1520}%
\special{pa 1920 1520}%
\special{pa 1920 1700}%
\special{pa 1740 1700}%
\special{pa 1740 1520}%
\special{pa 1920 1520}%
\special{fp}%
% BOX 2 0 3 0 Black Black  
% 2 1920 1520 2100 1700
% 
\special{pn 8}%
\special{pa 1920 1520}%
\special{pa 2100 1520}%
\special{pa 2100 1700}%
\special{pa 1920 1700}%
\special{pa 1920 1520}%
\special{pa 2100 1520}%
\special{fp}%
% CIRCLE 2 0 0 0 Black Black  
% 4 1650 1610 1677 1610 1677 1610 1677 1610
% 
\special{sh 1.000}%
\special{ia 1650 1610 27 27 0.0000000 6.2831853}%
\special{pn 8}%
\special{ar 1650 1610 27 27 0.0000000 6.2831853}%
% CIRCLE 2 0 0 0 Black Black  
% 4 1290 1250 1317 1250 1317 1250 1317 1250
% 
\special{sh 1.000}%
\special{ia 1290 1250 27 27 0.0000000 6.2831853}%
\special{pn 8}%
\special{ar 1290 1250 27 27 0.0000000 6.2831853}%
% CIRCLE 2 0 0 0 Black Black  
% 4 1470 1430 1497 1430 1497 1430 1497 1430
% 
\special{sh 1.000}%
\special{ia 1470 1430 27 27 0.0000000 6.2831853}%
\special{pn 8}%
\special{ar 1470 1430 27 27 0.0000000 6.2831853}%
\end{picture}}%
\vspace{5pt}
\]
If we take $w=32154$ as above, then $\hc{w}$ contains a lower diagonal full-string of height~$3$ as the following picture shows.\vspace{10pt}
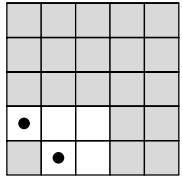
\[
%WinTpicVersion4.32a
{\unitlength 0.1in%
\begin{picture}(9.0000,9.0000)(12.0000,-17.0000)%
% BOX 2 0 1 0 Black Black  
% 2 1200 800 1380 980
% 
\special{pn 0}%
\special{sh 0.150}%
\special{pa 1200 800}%
\special{pa 1380 800}%
\special{pa 1380 980}%
\special{pa 1200 980}%
\special{pa 1200 800}%
\special{ip}%
\special{pn 8}%
\special{pa 1200 800}%
\special{pa 1380 800}%
\special{pa 1380 980}%
\special{pa 1200 980}%
\special{pa 1200 800}%
\special{pa 1380 800}%
\special{fp}%
% BOX 2 0 1 0 Black Black  
% 2 1380 800 1560 980
% 
\special{pn 0}%
\special{sh 0.150}%
\special{pa 1380 800}%
\special{pa 1560 800}%
\special{pa 1560 980}%
\special{pa 1380 980}%
\special{pa 1380 800}%
\special{ip}%
\special{pn 8}%
\special{pa 1380 800}%
\special{pa 1560 800}%
\special{pa 1560 980}%
\special{pa 1380 980}%
\special{pa 1380 800}%
\special{pa 1560 800}%
\special{fp}%
% BOX 2 0 1 0 Black Black  
% 2 1560 800 1740 980
% 
\special{pn 0}%
\special{sh 0.150}%
\special{pa 1560 800}%
\special{pa 1740 800}%
\special{pa 1740 980}%
\special{pa 1560 980}%
\special{pa 1560 800}%
\special{ip}%
\special{pn 8}%
\special{pa 1560 800}%
\special{pa 1740 800}%
\special{pa 1740 980}%
\special{pa 1560 980}%
\special{pa 1560 800}%
\special{pa 1740 800}%
\special{fp}%
% BOX 2 0 1 0 Black Black  
% 2 1740 800 1920 980
% 
\special{pn 0}%
\special{sh 0.150}%
\special{pa 1740 800}%
\special{pa 1920 800}%
\special{pa 1920 980}%
\special{pa 1740 980}%
\special{pa 1740 800}%
\special{ip}%
\special{pn 8}%
\special{pa 1740 800}%
\special{pa 1920 800}%
\special{pa 1920 980}%
\special{pa 1740 980}%
\special{pa 1740 800}%
\special{pa 1920 800}%
\special{fp}%
% BOX 2 0 1 0 Black Black  
% 2 1920 800 2100 980
% 
\special{pn 0}%
\special{sh 0.150}%
\special{pa 1920 800}%
\special{pa 2100 800}%
\special{pa 2100 980}%
\special{pa 1920 980}%
\special{pa 1920 800}%
\special{ip}%
\special{pn 8}%
\special{pa 1920 800}%
\special{pa 2100 800}%
\special{pa 2100 980}%
\special{pa 1920 980}%
\special{pa 1920 800}%
\special{pa 2100 800}%
\special{fp}%
% BOX 2 0 1 0 Black Black  
% 2 1200 980 1380 1160
% 
\special{pn 0}%
\special{sh 0.150}%
\special{pa 1200 980}%
\special{pa 1380 980}%
\special{pa 1380 1160}%
\special{pa 1200 1160}%
\special{pa 1200 980}%
\special{ip}%
\special{pn 8}%
\special{pa 1200 980}%
\special{pa 1380 980}%
\special{pa 1380 1160}%
\special{pa 1200 1160}%
\special{pa 1200 980}%
\special{pa 1380 980}%
\special{fp}%
% BOX 2 0 1 0 Black Black  
% 2 1380 980 1560 1160
% 
\special{pn 0}%
\special{sh 0.150}%
\special{pa 1380 980}%
\special{pa 1560 980}%
\special{pa 1560 1160}%
\special{pa 1380 1160}%
\special{pa 1380 980}%
\special{ip}%
\special{pn 8}%
\special{pa 1380 980}%
\special{pa 1560 980}%
\special{pa 1560 1160}%
\special{pa 1380 1160}%
\special{pa 1380 980}%
\special{pa 1560 980}%
\special{fp}%
% BOX 2 0 1 0 Black Black  
% 2 1560 980 1740 1160
% 
\special{pn 0}%
\special{sh 0.150}%
\special{pa 1560 980}%
\special{pa 1740 980}%
\special{pa 1740 1160}%
\special{pa 1560 1160}%
\special{pa 1560 980}%
\special{ip}%
\special{pn 8}%
\special{pa 1560 980}%
\special{pa 1740 980}%
\special{pa 1740 1160}%
\special{pa 1560 1160}%
\special{pa 1560 980}%
\special{pa 1740 980}%
\special{fp}%
% BOX 2 0 1 0 Black Black  
% 2 1740 980 1920 1160
% 
\special{pn 0}%
\special{sh 0.150}%
\special{pa 1740 980}%
\special{pa 1920 980}%
\special{pa 1920 1160}%
\special{pa 1740 1160}%
\special{pa 1740 980}%
\special{ip}%
\special{pn 8}%
\special{pa 1740 980}%
\special{pa 1920 980}%
\special{pa 1920 1160}%
\special{pa 1740 1160}%
\special{pa 1740 980}%
\special{pa 1920 980}%
\special{fp}%
% BOX 2 0 1 0 Black Black  
% 2 1920 980 2100 1160
% 
\special{pn 0}%
\special{sh 0.150}%
\special{pa 1920 980}%
\special{pa 2100 980}%
\special{pa 2100 1160}%
\special{pa 1920 1160}%
\special{pa 1920 980}%
\special{ip}%
\special{pn 8}%
\special{pa 1920 980}%
\special{pa 2100 980}%
\special{pa 2100 1160}%
\special{pa 1920 1160}%
\special{pa 1920 980}%
\special{pa 2100 980}%
\special{fp}%
% BOX 2 0 1 0 Black Black  
% 2 1200 1160 1380 1340
% 
\special{pn 0}%
\special{sh 0.150}%
\special{pa 1200 1160}%
\special{pa 1380 1160}%
\special{pa 1380 1340}%
\special{pa 1200 1340}%
\special{pa 1200 1160}%
\special{ip}%
\special{pn 8}%
\special{pa 1200 1160}%
\special{pa 1380 1160}%
\special{pa 1380 1340}%
\special{pa 1200 1340}%
\special{pa 1200 1160}%
\special{pa 1380 1160}%
\special{fp}%
% BOX 2 0 1 0 Black Black  
% 2 1380 1160 1560 1340
% 
\special{pn 0}%
\special{sh 0.150}%
\special{pa 1380 1160}%
\special{pa 1560 1160}%
\special{pa 1560 1340}%
\special{pa 1380 1340}%
\special{pa 1380 1160}%
\special{ip}%
\special{pn 8}%
\special{pa 1380 1160}%
\special{pa 1560 1160}%
\special{pa 1560 1340}%
\special{pa 1380 1340}%
\special{pa 1380 1160}%
\special{pa 1560 1160}%
\special{fp}%
% BOX 2 0 1 0 Black Black  
% 2 1560 1160 1740 1340
% 
\special{pn 0}%
\special{sh 0.150}%
\special{pa 1560 1160}%
\special{pa 1740 1160}%
\special{pa 1740 1340}%
\special{pa 1560 1340}%
\special{pa 1560 1160}%
\special{ip}%
\special{pn 8}%
\special{pa 1560 1160}%
\special{pa 1740 1160}%
\special{pa 1740 1340}%
\special{pa 1560 1340}%
\special{pa 1560 1160}%
\special{pa 1740 1160}%
\special{fp}%
% BOX 2 0 1 0 Black Black  
% 2 1740 1160 1920 1340
% 
\special{pn 0}%
\special{sh 0.150}%
\special{pa 1740 1160}%
\special{pa 1920 1160}%
\special{pa 1920 1340}%
\special{pa 1740 1340}%
\special{pa 1740 1160}%
\special{ip}%
\special{pn 8}%
\special{pa 1740 1160}%
\special{pa 1920 1160}%
\special{pa 1920 1340}%
\special{pa 1740 1340}%
\special{pa 1740 1160}%
\special{pa 1920 1160}%
\special{fp}%
% BOX 2 0 1 0 Black Black  
% 2 1920 1160 2100 1340
% 
\special{pn 0}%
\special{sh 0.150}%
\special{pa 1920 1160}%
\special{pa 2100 1160}%
\special{pa 2100 1340}%
\special{pa 1920 1340}%
\special{pa 1920 1160}%
\special{ip}%
\special{pn 8}%
\special{pa 1920 1160}%
\special{pa 2100 1160}%
\special{pa 2100 1340}%
\special{pa 1920 1340}%
\special{pa 1920 1160}%
\special{pa 2100 1160}%
\special{fp}%
% BOX 2 0 3 0 Black White  
% 2 1200 1340 1380 1520
% 
\special{pn 8}%
\special{pa 1200 1340}%
\special{pa 1380 1340}%
\special{pa 1380 1520}%
\special{pa 1200 1520}%
\special{pa 1200 1340}%
\special{pa 1380 1340}%
\special{fp}%
% BOX 2 0 3 0 Black White  
% 2 1380 1340 1560 1520
% 
\special{pn 8}%
\special{pa 1380 1340}%
\special{pa 1560 1340}%
\special{pa 1560 1520}%
\special{pa 1380 1520}%
\special{pa 1380 1340}%
\special{pa 1560 1340}%
\special{fp}%
% BOX 2 0 3 0 Black Black  
% 2 1560 1340 1740 1520
% 
\special{pn 8}%
\special{pa 1560 1340}%
\special{pa 1740 1340}%
\special{pa 1740 1520}%
\special{pa 1560 1520}%
\special{pa 1560 1340}%
\special{pa 1740 1340}%
\special{fp}%
% BOX 2 0 1 0 Black Black  
% 2 1740 1340 1920 1520
% 
\special{pn 0}%
\special{sh 0.150}%
\special{pa 1740 1340}%
\special{pa 1920 1340}%
\special{pa 1920 1520}%
\special{pa 1740 1520}%
\special{pa 1740 1340}%
\special{ip}%
\special{pn 8}%
\special{pa 1740 1340}%
\special{pa 1920 1340}%
\special{pa 1920 1520}%
\special{pa 1740 1520}%
\special{pa 1740 1340}%
\special{pa 1920 1340}%
\special{fp}%
% BOX 2 0 1 0 Black Black  
% 2 1920 1340 2100 1520
% 
\special{pn 0}%
\special{sh 0.150}%
\special{pa 1920 1340}%
\special{pa 2100 1340}%
\special{pa 2100 1520}%
\special{pa 1920 1520}%
\special{pa 1920 1340}%
\special{ip}%
\special{pn 8}%
\special{pa 1920 1340}%
\special{pa 2100 1340}%
\special{pa 2100 1520}%
\special{pa 1920 1520}%
\special{pa 1920 1340}%
\special{pa 2100 1340}%
\special{fp}%
% BOX 2 0 1 0 Black Black  
% 2 1200 1520 1380 1700
% 
\special{pn 0}%
\special{sh 0.150}%
\special{pa 1200 1520}%
\special{pa 1380 1520}%
\special{pa 1380 1700}%
\special{pa 1200 1700}%
\special{pa 1200 1520}%
\special{ip}%
\special{pn 8}%
\special{pa 1200 1520}%
\special{pa 1380 1520}%
\special{pa 1380 1700}%
\special{pa 1200 1700}%
\special{pa 1200 1520}%
\special{pa 1380 1520}%
\special{fp}%
% BOX 2 0 3 0 Black White  
% 2 1380 1520 1560 1700
% 
\special{pn 8}%
\special{pa 1380 1520}%
\special{pa 1560 1520}%
\special{pa 1560 1700}%
\special{pa 1380 1700}%
\special{pa 1380 1520}%
\special{pa 1560 1520}%
\special{fp}%
% BOX 2 0 3 0 Black White  
% 2 1560 1520 1740 1700
% 
\special{pn 8}%
\special{pa 1560 1520}%
\special{pa 1740 1520}%
\special{pa 1740 1700}%
\special{pa 1560 1700}%
\special{pa 1560 1520}%
\special{pa 1740 1520}%
\special{fp}%
% BOX 2 0 1 0 Black Black  
% 2 1740 1520 1920 1700
% 
\special{pn 0}%
\special{sh 0.150}%
\special{pa 1740 1520}%
\special{pa 1920 1520}%
\special{pa 1920 1700}%
\special{pa 1740 1700}%
\special{pa 1740 1520}%
\special{ip}%
\special{pn 8}%
\special{pa 1740 1520}%
\special{pa 1920 1520}%
\special{pa 1920 1700}%
\special{pa 1740 1700}%
\special{pa 1740 1520}%
\special{pa 1920 1520}%
\special{fp}%
% BOX 2 0 1 0 Black Black  
% 2 1920 1520 2100 1700
% 
\special{pn 0}%
\special{sh 0.150}%
\special{pa 1920 1520}%
\special{pa 2100 1520}%
\special{pa 2100 1700}%
\special{pa 1920 1700}%
\special{pa 1920 1520}%
\special{ip}%
\special{pn 8}%
\special{pa 1920 1520}%
\special{pa 2100 1520}%
\special{pa 2100 1700}%
\special{pa 1920 1700}%
\special{pa 1920 1520}%
\special{pa 2100 1520}%
\special{fp}%
% CIRCLE 2 0 0 0 Black Black  
% 4 1290 1430 1317 1430 1317 1430 1317 1430
% 
\special{sh 1.000}%
\special{ia 1290 1430 27 27 0.0000000 6.2831853}%
\special{pn 8}%
\special{ar 1290 1430 27 27 0.0000000 6.2831853}%
% CIRCLE 2 0 0 0 Black Black  
% 4 1470 1610 1497 1610 1497 1610 1497 1610
% 
\special{sh 1.000}%
\special{ia 1470 1610 27 27 0.0000000 6.2831853}%
\special{pn 8}%
\special{ar 1470 1610 27 27 0.0000000 6.2831853}%
\end{picture}}%
\]

We now state the main theorems of this paper.
We assume that the linear operator $N$ is the regular nilpotent matrix in Jordan canonical form and that Hessenberg functions $h\colon[n]\rightarrow [n]$ satisfy the condition
\begin{align}\label{eq: basic assumption}
 h(i)\ge i+1 \quad \text{for $1\le i<n$}.
\end{align}
 
\begin{theorem}\label{thm: main in intro}
Let $\Hess(N,h)$ be a regular nilpotent Hessenberg variety. A permutation flag $w_{\bullet}\in \Hess(N,h)$ is a singular point of $\Hess(N,h)$ if and only if $\hc{w}$ $($defined in \eqref{eq: def of Hw}$)$ contains a lower diagonal full-string.
\end{theorem}

\begin{example}\label{ex: intro}
Let $h=(3,3,4,5,5)$ as above. As we will see in Lemma~\ref{lem: permutation flags in Hess 2}, $w=32154$ in one-line notation gives a permutation flag in $\Hess(N,h)$. We have already seen above that the set $\hc{w}$ contains a lower diagonal full-string.  
Hence Theorem~\ref{thm: main in intro} says that $w_\bullet$ is a singular point of $\Hess(N,h)$.
\end{example}

\vspace{10pt}

We apply Theorem~\ref{thm: main in intro} to characterize the condition that $\Hess(N,h)$ is a normal algebraic variety.
Recall that we are assuming \eqref{eq: basic assumption}.

\begin{theorem} \label{thm:normal in intro}
A regular nilpotent Hessenberg variety $\Hess(N,h)$ is normal if and only if we have $h(i-1)> i$ or $h(i)> i+1$ for all $1 < i < n-1$.
\end{theorem}

Theorem~\ref{thm:normal in intro} gives us a simple combinatorial recipe to check a regular nilpotent Hessenberg variety for normality by scanning its Hessenberg function.
Colloquially, $\Hess(N,h)$ is not normal if there is some index $1< i < n-1 $ where the Hessenberg function looks like that of the Peterson variety in both positions $i-1$ and $i$.

\begin{example} When $n\le3$, all regular nilpotent Hessenberg varieties are normal.
We list all of the Hessenberg functions satisfying \eqref{eq: basic assumption} when $n=4$ and highlight in red the pair of indexes that offend the normality condition of Theorem~\ref{thm:normal in intro}:
$$
(\textcolor{red}{\textbf{2}},\textcolor{red}{\textbf{3}},4,4)  \quad (3,3,4,4)  \quad  (2,4,4,4) \quad (3,4,4,4) \quad (4,4,4,4).$$
Thus we see that the Peterson variety ($h=(\textcolor{red}{\textbf{2}},\textcolor{red}{\textbf{3}},4,4)$) is the only regular nilpotent Hessenberg variety that is not normal when $n=4$ under the condition \eqref{eq: basic assumption}.
Theorem~\ref{thm:normal in intro}  also shows that $\Hess(N,h)$ considered in Example~\ref{ex: intro} is not normal
because $(3,\textcolor{red}{\textbf{3}},\textcolor{red}{\textbf{4}},5,5)$ offends the normality condition at the index highlighted in red.
 
\end{example}

\section{Background on regular nilpotent Hessenberg varieties}\label{sec: background}

In this section, we recall the necessary background and establish our conventions regarding subvarieties of the flag variety.
Let $n\ge2$, and $\GL_n(\C)$ denote the general linear group and $B \subset \GL_n(\C)$ denote the Borel subgroup of upper-triangular matrices. 
The \textbf{flag variety} is the homogenous space $$\Flags(\C^n) \coloneqq  
  \{ V_{\bullet} = (\{0\} \subset V_1 \subset V_2 \subset \cdots \subset
  V_n = \C^n) \mid \dim_{\C} V_i = i\ (1\le i\le n) \}$$ consisting of all nested sequences of linear subspaces of $\C^n$. 
It is well-known that 
\begin{equation}\label{eq:flag}
 \Flags(\C^n) \cong   \GL_n(\C)/B ,
\end{equation} where this isomorphism takes a coset $gB\in \GL_n(\C)/B$, for $g \in \GL_n(\C)$, to the
flag $V_\bullet$ with $V_i$ defined by the span of the first $i$
columns of $g$.
Henceforth, we primarily use the coset definition of the flag variety and $ \GL_n(\C)/B$ to denote the flag variety.

Throughout this paper, we use the notation 
\begin{align*}
 [n] \coloneqq  \{1,2,\ldots,n\}.
\end{align*}
A \textbf{Hessenberg function} is a function $h\colon [n] \to [n]$
satisfying $h(i) \geq i$ for all $1 \leq i \leq n$ and
$h(i)\le h(i+1)$ for all $1 \leq i < n$, and we frequently denote a
Hessenberg function by listing its values in a sequence,
$h = (h(1), h(2), \ldots, h(n))$. 
As we saw in Section~\ref{subsec: preview}, we may identify a Hessenberg function $h$ and a configuration of shaded boxes in $[n]\times [n]$.

\begin{definition}\label{HessDefn}
  Let $X$ be an $n\times n$ matrix, viewed as a linear map $X\colon \C^n\rightarrow\C^n$, and
  $h\colon [n] \to [n]$ a Hessenberg function. 
  The \textbf{Hessenberg variety} associated to $X$ and
  $h$ is defined to be $$\Hess(X,h)\coloneqq \{V_\bullet\in \Flags(\C^n)\mid
  XV_i\subseteq V_{h(i)} \ (1\le i\le n)\}.$$ Equivalently, under the
  identification~\eqref{eq:flag}, we have
\begin{equation}\label{eq:def Hess conj}
\Hess(X,h) = \{ gB \in \GL_n(\C)/B \mid (g^{-1}Xg)_{i,j}=0\ (i>h(j)) \}.
\end{equation}
\end{definition}

Any Hessenberg variety $\Hess(X,h)$ is, by definition,
an algebraic subset of the flag variety $\GL_n(\C)/B$.  It is
straightforward to see that $\Hess(X,h)$ and $\Hess(gXg^{-1},h)$ are
isomorphic for all $g\in \GL_n(\C)$, so we frequently assume
without loss of generality that $X$ is in Jordan canonical form (with one Jordan block and ones along its super-diagonal) with respect to the standard basis on $\C^n$.

\begin{definition}\label{def:reg nilp}
A \textbf{regular nilpotent Hessenberg variety} $\Hess(N,h)$ is the Hessenberg variety associated to the principal nilpotent matrix $N$ in Jordan canonical form:
\begin{equation*}
N=
\begin{pmatrix}
0 & 1 & &  & \\
 & 0 &  1  & \\
 &    & \ddots & \ddots & \\
 &    &   & 0 & 1 \\
 &    &  & & 0 \\ 
\end{pmatrix}.
\end{equation*}
\end{definition}

It is known from \cite[Lemma 7.1]{AT10} that $\Hess(N,h)$ is irreducible.
It is also known from \cite{Tym04} that 
\begin{align}\label{eq: dim of Hess}
\dim_{\C} \Hess(N,h)=\sum_{j=1}^n (h(j)-j)
\end{align}
which is precisely the number of shaded boxes below the diagonal of the configuration of boxes corresponding to the Hessenberg function $h$ (see Section~\ref{subsec: preview}).
In particular, it follows that
\begin{align}\label{eq: codim of Hess}
\codim_{\C} (\Hess(N,h),\Flags(\C^n))
 =| \{(i,j)\in[n]\times[n] \mid i>h(j)\} |,
\end{align}
where the right hand side is the number of unshaded boxes in the configuration of boxes corresponding to the Hessenberg function $h$.

Let $\Sn$ denote the permutation group on $[n]$.
The Bruhat decomposition of the general linear group $$\GL_n(\C) = \coprod_{w \in \mathfrak{S}_n} BwB$$
gives rise to a Schubert cell decomposition of the flag variety
\begin{align}\label{eq: Bruhat decomposition}
 \GL_n(\C)/B = \coprod_{w \in \mathfrak{S}_n} \SC{w} ,
\end{align}
where $ \SC{w} \coloneqq  BwB/B$ is referred to as the \textbf{Schubert cell corresponding to $w\in\Sn$}.
A Schubert cell $\SC{w}$ can be represented by a set of matrices $g$ with $1$s in position $(w(j),j)$ for $1\le j\le n$, $0$s below and to the right of the $1$s, and free entries above and to the left of the $1$s.
The Zariski-closure of a Schubert cell is a \textbf{Schubert variety} $X_w$ and it consists of a union of cells  $$X_w = \coprod_{u \leq w } \SC{w},$$  
where $u \leq w$ in the Bruhat order of $\mathfrak{S}_n.$
For each $w\in\Sn$, there is the associated point $w_{\bullet}$ in $\Flags(\C^n)$ given by
$$ 
w_{\bullet} \coloneqq ( \langle e_{w(1)} \rangle_{\C} \subset \langle e_{w(1)},e_{w(2)} \rangle_{\C} \subset \cdots \subset \langle e_{w(1)},e_{w(2)},\ldots,e_{w(n)} \rangle_{\C} ),
$$
where $e_1,e_2,\ldots,e_n$ are the standard basis of $\C^n$. 
This $w_{\bullet}$ is called the permutation flag corresponding to $w\in\Sn$.
The Schubert cell $\SC{w}$ is the $B$-orbit of $w_{\bullet}$ in $\GL_n(\C)/B$ with respect to the standard $B$-action from the left. 

It has become quite commonplace to refer to the set $\Hess(N,h)\cap \SC{w}$ as a \textbf{Hessenberg-Schubert cell}.  
From \eqref{eq: Bruhat decomposition}, it is clear that we have 
\begin{align}\label{eq: def of HS cell}
\Hess(N,h) =  \coprod_{w_{\bullet} \in \Hess(N,h)} \Hess(N,h) \cap \SC{w}  .
\end{align}
Tymoczko proved in \cite{Tym06} that this decomposition forms a paving by affines. In particular, each Hessenberg-Schubert cell $\Hess(N,h) \cap \SC{w}$ is isomorphic to a complex affine space $\C^d$ for some $d\ge1$.
Precup generalized it to a much larger family of Hessenberg varieties in all Lie types \cite{Pre13}. 
The following claim tells us which permutation flags $w_{\bullet}$ appear in \eqref{eq: def of HS cell}.

\begin{lemma}\label{lem: permutation flags in Hess 2} \emph{\cite[Proposition~2.6]{Tym04}}
The permutation flag $w_{\bullet}\in Fl(\C^n)$ lies in $\Hess(N,h)$ if and only if $w^{-1}(i) \leq h(w^{-1}(i+1)) $ for all $1\le i<n$.
\end{lemma}

\bigskip

\section{Generators of the defining ideals of $\Hess(N,h)$}\label{sec: Linear terms}

\subsection{Local descriptions of $\Hess(N,h)$}\label{subsec: local description}

In the rest of this paper, we always assume that condition \eqref{eq: basic assumption} in Section~\ref{subsec: preview} holds.
It is well-known that the flag variety $\Flags(\C^n)$ can be covered by affine open subsets, each isomorphic to $\C^{n(n-1)/2}$. Let $U^-\subset \GL_n(\C)$ be the unipotent subgroup consisting of lower triangular matrices whose diagonal components are $1$.
Then we have an isomorphism $U^-\cong \C^{n(n-1)/2}$ by taking strict lower triangular components.
The map
\begin{align*}
U^- \rightarrow \GL_n(\C)/B
\quad ; \quad
g \mapsto gB
\end{align*}
is known to be an open embedding.
For each permutation $w\in\Sn$, we denote by $wU^-\subset \GL_n(\C)/B$ the $w$-translate of the  image of this open embedding:
\begin{align*}
wU^- \coloneqq\{wgB\in \GL_n(\C)/B\mid g\in U^-\}.
\end{align*}
By construction, we have $wU^-\cong \C^{n(n-1)/2}$ for each $w\in\Sn$, and it is well-known that these form an affine open cover of $\Flags(\C^n)$ :
\begin{align*}
\Flags(\C^n) = \bigcup_{w\in\mathfrak{S}_n} wU^- .
\end{align*} 
This gives us an affine open cover of $\Hess(N,h)$ :
\begin{align*}
\Hess(N,h) = \bigcup_{w\in\mathfrak{S}_n}  \Hess(N,h) \cap wU^- .
\end{align*}
We note that we may remove some pieces from this cover as follows.
Recall that we have the decomposition \eqref{eq: def of HS cell} of $\Hess(N,h)$ into the Hessenberg-Schubert cells $\Hess(N,h)\cap \SC{w}$.
Since we have $\SC{w}\subseteq wU^-$, it follows from \eqref{eq: def of HS cell} that
\begin{align}\label{eq: open cover of Hess}
\Hess(N,h) = \bigcup_{w_{\bullet}\in\Hess(N,h)} \Hess(N,h) \cap wU^- ,
\end{align}
where the union takes only the permutation flags $w_{\bullet}$ lying in $\Hess(N,h)$.

We now describe the coordinate rings of the affine varieties $wU^-$ and their subvarieties $\Hess(N,h) \cap wU^-$.
By the construction above, an element in $wU^-$ is uniquely determined by an $n\times n$ matrix
$g=(x_{i,j})$ whose entries are subject to the following relations
    \begin{equation}
      \label{eq:11}
      \begin{split}
        &x_{w(j),j} = 1, \quad (j\in [n]),\\
        &x_{w(i),j} = 0, \quad (i<j).
      \end{split}
    \end{equation}
    Thus the coordinate ring of $wU^-$, denoted by
    $\C [ \mathbf{x}_w]$, is the quotient
    of the polynomial ring $\C [x_{i,j}\mid i,j\in[n]]$ by the relations
    \eqref{eq:11}. 
More precisely, we set
\[
\C[{\bf x}_w]\coloneqq \C[x_{i,j} \mid i,j\in[n]]/I_w ,
\]
where $I_w$ is the ideal of $\C[{\bf x}_w]$ generated by 
\[
\text{$x_{w(j),j}-1 \ \ (j\in [n])$ \quad and \quad  $x_{w(i),j} \ \ (i<j)$.}
\]
Namely, $\C[{\bf x}_w]$ is the coordinate ring of $wU^-(\subset \GL_n(\C)/B)$.
For simplicity, we use the same symbol $x_{i,j}$ to denote its image in $\C[{\bf x}_w]$.
By having \eqref{eq:11} in mind, we may identify the quotient ring $\C[{\bf x}_w]$ as the polynomial ring over $\C$ with variables $x_{i,j}$ satisfying $w^{-1}(i)>j$.
This leads us to the following definition.

\begin{definition}\label{def: genuine variable}
We call $x_{i,j}\in\C[{\bf x}_w]$ a \textbf{genuine variable} if $w^{-1}(i)>j$.
\end{definition}

\begin{example}\label{ex: genuine variables}
Let $n=4$.
For $w= 3142$ in one-line notation, the permutation matrix associated to $w$ is
\[
\begin{pmatrix}
 & 1 & & \\
 & &  & 1 \\
 1 & & & \\
 & & 1 & 
\end{pmatrix},
\]
and the genuine variables of the ring $\C[{\bf x}_w]$ can be visualized in matrix form
\begin{align*}
\begin{pmatrix}
 x_{11} & 1 & & \\
 x_{21} & x_{22} & x_{23} & 1 \\
 1 & & & \\
 x_{41} & x_{42} & 1 & 
\end{pmatrix}.
\end{align*}
\end{example} 

\vspace{20pt}

For the rest of this section, we fix a  permutation $w\in\Sn$ to focus on the open affine subvariety $\Hess(N,h)\cap wU^-$ of $\Hess(N,h)$.
From 
\eqref{eq: open cover of Hess}, we know that it is enough to study $\Hess(N,h)\cap wU^-$ with $w_{\bullet}\in\Hess(N,h)$, but we do not assume this until we need it.
To describe the coordinate ring of $\Hess(N,h)\cap wU^-$, we consider the collection of variables in the ring $\C[{\bf x}_w]$ in the matrix form
\begin{align*}
M\coloneqq
\begin{pmatrix}
   x_{1,1} &  \cdots  &  x_{1,n}  \\ 
   \vdots &  \ddots  &  \vdots  \\
   x_{n,1}  &  \cdots  & x_{n,n}
\end{pmatrix}.
\end{align*}
where we have
\begin{equation}\label{eq: convention of xij}
\begin{split}
&x_{w(j),j}=1 \quad (j\in[n]), \\
&x_{w(i), j}=0 \quad (i<j).
\end{split}
\end{equation}
These conditions imply that the matrix $M$ is invertible over $\C[{\bf x}_w]$.
We set
\begin{align*}
f_{i, j}:&=
(M^{-1}N M)_{i, j} \in \C[{\bf x}_w],
\end{align*}
where $N$ is the regular nilpotent matrix given in Definition~\ref{def:reg nilp}.
We consider the ideal of $\C[{\bf x}_w]$ defined by
\begin{align}\label{eq: def of Hessenberg ideal}
J_{w,h} \coloneqq  ( f_{i, j} \mid i>h(j) ) \subseteq \C[{\bf x}_w].
\end{align}
It is clear from \eqref{eq:def Hess conj} that functions in $J_{w,h}$ vanish on $\Hess(N,h)\cap wU^-$.
The following claim says that $J_{w,h}$ is in fact the defining ideal of $\Hess(N,h)\cap wU^-$ in $wU^-$  under the assumption \eqref{eq: basic assumption}.

\begin{proposition}
$($\emph{\cite[Proposition 3.7]{ADGH18}}$)$
The quotient ring $\C[{\bf x}_w]/J_{w,h}$ is the coordinate ring of $\Hess(N,h)\cap wU^-$.
In particular, the ideal $J_{w,h}$ is radical. 
\end{proposition}

We will apply the following result to determine which points in $\Hess(N,h)$ are singular and which are nonsingular. See e.g.\ \cite[Chapter~I, Section~ 5]{Har77}.

\begin{theorem}[Jacobian Criterion]\label{theorem:Jacobian_Criterion'}
Let $Y\subseteq \C^m$ be an affine variety, and let $f_1,\ldots,f_k$ be a set of generators for the ideal of $Y$. Then a point $p\in Y$ is a singular point of $Y$ if and only if $$\rank \left(\frac{\partial f_i}{\partial x_j}(p)\right)_{1\le i\le k, \ 1\le j\le m}<m-r,$$
where $r$ is the dimension of $Y$, and $x_1,\ldots,x_m$ are the standard coordinate of $\C^m$.
\end{theorem}

To apply this criterion to the subvariety $\Hess(N,h)\cap wU^-$ of $wU^-\cong \C^{n(n-1)/2}$, we compute the generators $f_{i,j}$ in terms of genuine variables $x_{i,j}$ in $\C[{\bf x}_w]$. For this purpose, we first focus on the components of $M^{-1}$.
Let us write $$y_{i,j}\coloneqq (M^{-1})_{i,j} \qquad (i,j\in[n]).$$ Then we have
\begin{align*}
(M^{-1}M)_{i,j} 
&= \sum_{\ell=1}^n (M^{-1})_{i,w(\ell)} M_{w(\ell),j}  \\
&= \sum_{\ell=1}^n y_{i,w(\ell)} x_{w(\ell),j} \\
&= y_{i,w(j)}+\sum_{\ell=j+1}^n y_{i,w(\ell)} x_{w(\ell),j} \quad (\text{by \eqref{eq: convention of xij}}).
\end{align*}
Thus, the condition $(M^{-1}M)_{i,j}=\delta_{i,j}$ is equivalent to 
\begin{align}\label{eq: first condition on y}
y_{i,w(j)}+\sum_{\ell=j+1}^n y_{i,w(\ell)} x_{w(\ell),j}=\delta_{i,j}.
\end{align}
From this, we see that
\begin{equation}\label{eq: convention of yij}
\begin{split}
&y_{i,w(i)}=1 \quad (i\in[n]), \\
&y_{i, w(j)}=0 \quad (i<j) 
\end{split}
\end{equation}
by an inductive argument with the base case $j=n$.
Hence, we obtain from \eqref{eq: first condition on y} and \eqref{eq: convention of yij} that
\begin{align}\label{eq: second condition on y}
y_{i,w(j)}+\sum_{\ell=j+1}^{i-1} y_{i,w(\ell)} x_{w(\ell),j}+x_{w(i),j}=0
\qquad (i>j),
\end{align}
where we note that the condition $i>j$ ensures that the term $\ell=i$ does appear in the summation in the left hand side of \eqref{eq: first condition on y}.
This enables us to compute $y_{i,w(j)}$ inductively in terms of genuine variables. We will use the equality \eqref{eq: second condition on y} in the next section to give a combinatorial formula for $y_{i,w(j)}$.

By definition, we have
\begin{align}\notag
f_{i,j} 
&= (M^{-1}NM)_{i,j} \\ \notag
&= \sum_{1\leq k \leq n} \sum_{1\leq \ell \leq n} (M^{-1})_{i,w(k)} N_{w(k),w(\ell)} M_{w(\ell),j} \\ \label{eq: formula of fij}
&= \sum_{1\leq k \leq i } \ \sum_{j\leq \ell \leq n} y_{i,w(k)} \delta_{w(k)+1,w(\ell)} x_{w(\ell),j} ,
\end{align}
where we used \eqref{eq: convention of xij} and \eqref{eq: convention of yij} at the third equality. We make the following change of indexing of the variables and the generators.

\begin{definition}\label{def: zij and gij}
For $i,j\in[n]$, we set
\begin{align*}
  &z_{i,j} \coloneqq  x_{i,w^{-1}(j)}\in \C[{\bf x}_w], \\
  &g_{i,j} \coloneqq  f_{w^{-1}(i),w^{-1}(j)}\in \C[{\bf x}_w],
\end{align*}
where $w\in \Sn$ is the fixed permutation as mentioned in the beginning of this section.
\end{definition}
 
This change of indexing will greatly simplify our proofs involving the Jacobian matrices of the generators of the defining ideals $J_{w,h}\subseteq \C[{\bf x}_w]$ of $\Hess(N,h)\cap wU^-$ in $wU^-$ as we will see in the next section. In particular, we will show that they have a particularly nice upper-triangular form when evaluated at any point in a Hessenberg-Schubert cell.
Recalling that the genuine variables in $\C[{\bf x}_w]$ are defined in Definition~\ref{def: genuine variable}, the following claim is immediate by \eqref{eq: convention of xij} and \eqref{eq: def of Hessenberg ideal}.

\begin{lemma}\label{lem: genuine zij}
The genuine variables in $\C[{\bf x}_w]$ are precisely $z_{i,j}$ for $w^{-1}(i)>w^{-1}(j)$. In particular, we have
\begin{equation}\label{eq: convention of zij}
\begin{split}
&z_{j,j}=1 \quad (j\in[n]), \\
&z_{i,j}=0 \quad (w^{-1}(i)<w^{-1}(j))
\end{split}
\end{equation}
in $\C[{\bf x}_w]$. Also, we have
\begin{align*}
J_{w,h} = ( \ g_{i, j} \mid w^{-1}(i)>h(w^{-1}(j)) \ ) \quad \text{in\ } \C[{\bf x}_w].
\end{align*}
\end{lemma}

\vspace{5pt}

\begin{example}
We saw in Example~\ref{ex: genuine variables} the set of genuine variables in $\C[{\bf x}_w]$ for $w=3142$. In terms of the new variables $z_{i,j}$, the genuine variables are $z_{1,3}$, $z_{2,3}$, $z_{2,1}$, $z_{2,4}$, $z_{4,3}$, $z_{4,1}$.
For these $z_{i,j}$, one can see that the index $(i,j)$ appears in the one-line notation $w=3142$ in the reversed order since it simply means that the condition $w^{-1}(i)>w^{-1}(j)$ holds.
\end{example}

\vspace{10pt}

\subsection{Properties of the generators of $J_{w,h}$}\label{subsec: properties of gij}
In the previous subsection, we fixed a permutation $w_{\bullet} \in \Sn$ to discuss the affine variety $\Hess(N,h)\cap wU^-$ in $wU^-$ and its defining ideal $J_{w,h}\subseteq \C[{\bf x}_w]$. 
By \eqref{eq: formula of fij} and Definition~\ref{def: zij and gij}, we have
\begin{align}\label{eq: gij in z}
g_{i,j} 
&= \sum_{1\leq w^{-1}(k) \leq w^{-1}(i) } \ \sum_{w^{-1}(j)\leq w^{-1}(\ell) \leq n} y_{w^{-1}(i),k} \delta_{k+1,\ell} z_{\ell,j}. 
\end{align}
To better understand the structure of the generators $g_{i,j}$, we first study 
\begin{align*}
y_{w^{-1}(i),k} \qquad (w^{-1}(i) \geq w^{-1}(k))
\end{align*}
appearing in this formula. 
When $w^{-1}(i)= w^{-1}(k)$, we have $i=k$ and hence $y_{w^{-1}(i), k}=1$ by \eqref{eq: convention of yij}. Hence, we consider the case $w^{-1}(i) > w^{-1}(k)$ in what follows.
In this case, 
we have from \eqref{eq: second condition on y} that
\begin{align}\label{eq: second condition on y in z}
y_{w^{-1}(i),k}+\sum_{w^{-1}(\ell)=w^{-1}(k)+1}^{w^{-1}(i)-1} y_{w^{-1}(i),\ell} z_{\ell,k}\ +z_{i,k}=0\qquad (w^{-1}(i) > w^{-1}(k)).
\end{align}
The next claim describes the explicit combinatorial structure of the summands appearing in $y_{w^{-1}(i),k}$.

\begin{proposition} \label{prop:summand}
For $w^{-1}(i) > w^{-1}(k)$, we have 
\begin{align}\label{eq: formula for yij}
 y_{w^{-1}(i), k} 
= \sum_{d=1}^{D}  \sum_{\bm{a}} 
(-1)^d z_{i,a_1}z_{a_1,a_2}\cdots z_{a_{d-1},k} ,
\end{align}
where $D=w^{-1}(i)-w^{-1}(k)$ is the degree of $y_{w^{-1}(i), k}$, and the second summation runs over all sub-sequences 
$\bm{a}=(k,a_{d-1},\ldots,a_2,a_1,i)$ 
of the one-line notation of $w$ whose head and tail are $k$ and $i$, respectively: 
\begin{align*}
w = 
 \ \cdots \ \ k \ \ \cdots \ \ a_{d-1} \ \ \cdots\cdots \ \ a_2 \ \ \cdots \ \ a_1 \ \ \cdots \ \ i \ \ \cdots \ .
\end{align*}
\end{proposition}

\begin{proof}
We prove the claim by induction on $D=w^{-1}(i) - w^{-1}(k)>0$.
When $D=1$, we have $w^{-1}(i)=w^{-1}(k)+1$ which means that $i$ is the next number of $k$ in the one-line notation of $w$: 
\begin{align}\label{eq: polynomial exp of yij 10}
w =
\begin{matrix}
\ \cdots & k & \!\!\! i & \cdots\ 
\end{matrix}.
\end{align}
The condition $w^{-1}(i) = w^{-1}(k)+1$ implies that the equality \eqref{eq: second condition on y in z} is valid and that the summation for $w^{-1}(\ell)$ is vacuous in this case. Namely, we have
\begin{align*}
y_{w^{-1}(i),k}+z_{i,k}=0
\end{align*}
which implies that $y_{w^{-1}(i),k}=-z_{i,k}$. This clearly satisfies the claim, where the (unique) subsequence appearing in the summation is simply $(k,i)$ in \eqref{eq: polynomial exp of yij 10}.

We now assume by induction that the claim holds for the cases which satisfies $D=w^{-1}(i)-w^{-1}(k)<p$ for some positive integer $p$, and we prove the claim for $D=p+1$.
Since we are assuming $w^{-1}(i) > w^{-1}(k)$, we have from \eqref{eq: second condition on y in z} that
\begin{align}\label{eq: str of y ind}
y_{w^{-1}(i),k}+\sum_{w^{-1}(\ell)=w^{-1}(k)+1}^{w^{-1}(i)-1} y_{w^{-1}(i),\ell} z_{\ell,k}+z_{i,k}=0 .
\end{align}
To apply the inductive assumption to $y_{w^{-1}(i),\ell}$ in the second summand, we need to verify that
\begin{align*}
 w^{-1}(i) - w^{-1}(\ell) < D.
\end{align*}
This inequality easily follows from the range of $\ell$:
\begin{align*}
 w^{-1}(i) - w^{-1}(\ell) < \ w^{-1}(i) - w^{-1}(k) = D.
\end{align*}
Thus, by the inductive assumption, $y_{w^{-1}(i),\ell}$ appearing in \eqref{eq: str of y ind} is an alternating sum of monomials whose indexes are sub-sequences 
$(\ell,a_{d'-1},\ldots,a_2,a_1,i)$ 
of the one-line notation of $w$:
\begin{align*}
w = 
 \left( \hspace{40pt}\ \cdots \ \ \ell \ \ \cdots \ \ a_{d'-1} \ \ \cdots\cdots \ \ a_2 \ \ \cdots \ \ a_1 \ \ \cdots \ \ i \ \ \cdots \ \right)
\end{align*}
for all $1\le d'\le w^{-1}(i) - w^{-1}(\ell)$.
Also, the indexes of $z_{\ell,k}$ in the second summand of \eqref{eq: str of y ind} appear in the one-line notation of $w$ as
\begin{align*}
\hspace{-90pt}
w =
 \left( \ \cdots \ \ k \ \ \cdots \ \ \ell \ \ \cdots\cdots \hspace{100pt}\right)
\end{align*}
since $w^{-1}(k)<w^{-1}(\ell)$ in \eqref{eq: str of y ind}.
Thus, each product $y_{w^{-1}(i),\ell} z_{\ell,k}$ in the second summand of \eqref{eq: str of y ind} satisfies the desired property; it is an alternating sum of monomials whose indexes are sub-sequences $(k,\ell,a_{d'-1},\ldots,a_2,a_1,i)$ of the one-line notation of $w$ for all $1\le d'\le w^{-1}(i) - w^{-1}(\ell)$. 
Therefore, the second summand in \eqref{eq: str of y ind} in total is an alternating sum of monomials whose indexes are sub-sequences $(k,a_{d-1},\ldots,a_2,a_1,i)$ of the one-line notation of $w$ for all $2\le d\le w^{-1}(i) - w^{-1}(k)=D$. 
Combining this with the last summand $z_{i,k}$, we see that $y_{w^{-1}(i),k}$ satisfies the desired claim for the case $D=p+1$.
\end{proof}

\begin{example}\label{ex: yij}
Let $n=6$, and take the permutation $w=564321$ in one-line notation.
Take $i=2$ and $k=5$ so that $w^{-1}(i) > w^{-1}(k)$ holds.
The subsequences of the one-line notation of $w$ whose head and tail are $k=5$ and $i=2$ respectively are given by
\begin{align*}
 (5,6,4,3,2), \ (5,4,3,2), \ (5,6,3,2), \ (5,6,4,2), \ (5,3,2), \ (5,4,2), \ (5,6,2), \ (5,2).
\end{align*}
Thus, by reversing each sequence, we obtain
\begin{equation*}
\begin{split}
 y_{w^{-1}(2), 5} 
 &= z_{2,3}z_{3,4}z_{4,6}z_{6,5} - z_{2,3}z_{3,4}z_{4,5} + z_{2,3}z_{3,6}z_{6,5} \\
 &\hspace{60pt}- z_{2,4}z_{4,6}z_{6,5} + z_{2,3}z_{3,5} - z_{2,4}z_{4,5} + z_{2,6}z_{6,5} - z_{2,5}.
\end{split}
\end{equation*}
\end{example}

\vspace{10pt}
An observant reader may notice in Example~\ref{ex: yij} that within each summand of $y_{w^{-1}(i),k}$, the net difference of the indexes appearing in that summand is always $i-k$. 
For example, the net difference for the first summand of $y_{w^{-1}(2), 5}$ above is $(2-3)+(3-4)+(4-5)+(5-6)=2-5$.
This easily follows from the formula \eqref{eq: formula for yij}, and this observation for $y_{w^{-1}(i), k}$ implies the following property of the generators $g_{i,j}$.

\begin{corollary} \label{cor:net_index}
For $i,j\in[n]$, every monomial $z_{a_0,b_0}z_{a_1,b_1}\cdots z_{a_{k},b_{k}} $ appearing in the generator $g_{i,j}$ 
has the net difference of indexes of $i-j+1$, or in other words, $$ (a_0-b_0)+ (a_1-b_1) + \cdots + (a_k-b_k)   = i-j+1.$$  
\end{corollary}

\begin{proof}  
We know from \eqref{eq: gij in z} that $g_{i,j}$ is given by 
\begin{align}\notag
g_{i,j} 
&= \sum_{1\leq w^{-1}(k) \leq w^{-1}(i) } \ \sum_{w^{-1}(j)\leq w^{-1}(\ell) \leq n} y_{w^{-1}(i),k} \delta_{k+1,\ell} z_{\ell,j} \\ \label{eq: gij in z 2}
&= \sum_{ \substack{1\leq w^{-1}(\ell-1) \leq w^{-1}(i) \\ w^{-1}(j)\leq w^{-1}(\ell) \leq n}} \  y_{w^{-1}(i),\ell-1} z_{\ell,j} ,
\end{align}
where we mean $y_{w^{-1}(i),0}=0$ by convention.
We may suppose that the summation is not vacuous since otherwise $g_{i,j}$ is zero. 
If there exists $\ell(\ge2)$ such that $w^{-1}(\ell-1) = w^{-1}(i)$, then we have $\ell=i+1$ so that the corresponding summand is $1\cdot z_{\ell,j}=z_{i+1,j}$, and it is obvious that this summands satisfies the claim.
If there exists $\ell (\ge2)$ such that $w^{-1}(\ell-1) < w^{-1}(i)$, then the corresponding summand is $y_{w^{-1}(i),\ell-1} z_{\ell,j}$, and \eqref{eq: formula for yij} implies that every monomial in $y_{w^{-1}(i),\ell-1}$ has the net difference of indexes $i-(\ell-1)$ as we observed above.
Hence, the net difference of indexes of the product $y_{w^{-1}(i),\ell-1}z_{\ell,j}$ is $i-(\ell-1)+\ell-j=i-j+1$.  Therefore,
every monomial appearing in $g_{i,j}$ 
has the net difference of indexes $i-j+1$.
\end{proof}

\begin{example}\label{ex:gij} 
To illustrate Corollary~\ref{cor:net_index}, we continue with the example of $w=564321$ with $i=2$ and $j=5$. The formula \eqref{eq: gij in z 2} implies that $g_{i,j}$ has a nontrivial term $y_{w^{-1}(i),\ell-1} z_{\ell,j}$ for each number $\ell$ in the one-line notation of $w$ where $\ell$ is weakly to the right of $j=5$ and $\ell-1$ is weakly to the left of $i=2$. One can confirm that $\ell=3,4,5,6$ are exactly such numbers. Hence, 
we have
\begin{align*} g_{2,5}  & =   ( y_{w^{-1}(2), 2} )z_{3,5} + (y_{w^{-1}(2),3} ) z_{4,5}  +(y_{w^{-1}(2),4}) z_{5,5} + ( y_{w^{-1}(2),5} ) z_{6,5} \\ 
& = (1)z_{3,5} +  (-z_{2,3}) z_{4,5} + (z_{2,3}z_{3,4} - z_{2,4}) 1  \\
 &\hspace{40pt}  + (z_{2,3}z_{3,4}z_{4,6}z_{6,5} - z_{2,3}z_{3,4}z_{4,5} - z_{2,3}z_{3,6}z_{6,5} - z_{2,4}z_{4,6}z_{6,5} \\
 &\hspace{160pt} + z_{2,3}z_{3,5} + z_{2,4}z_{4,5} + z_{2,6}z_{6,5} - z_{2,5}) z_{6,5}.
\end{align*}
One can see that every monomial appearing in $g_{2,5}$ has the net difference of indexes $2-5+1=-2.$
\end{example}

\begin{example}\label{eg: 312654-2}
Let $n=6$ and $w=312654$ in one-line notation, and 
take $i=6$ and $j=3$. Similarly, we have
\begin{align*}
 g_{6,3} 
 &= (y_{w^{-1}(6),1})z_{2,3} + (y_{w^{-1}(6),2})z_{3,3} + (y_{w^{-1}(6),3})z_{4,3} \\
 &= (z_{6,2}z_{2,1}-z_{6,1})z_{2,3} + (-z_{6,2})1+(-z_{6,2}z_{2,1}z_{1,3}+z_{6,2}z_{2,3}+z_{6,1}z_{1,3}-z_{6,3})z_{4,3},
\end{align*}
and every monomial appearing in $g_{6,3} $ has the net difference of indexes $6-3+1=4.$
\end{example} 

\vspace{10pt}

\subsection{Differentials of $g_{i,j}$ on Hessenberg-Schubert cells}
Recall that $\C[{\bf x}_w]$ is the coordinate ring of the affine open subset $wU^-\subset \Flags(\C^n)$.
Since the Schubert cell $\SC{w}$ is contained in $wU^-$, 
a function $F\in \C[{\bf x}_w]$ can be restricted on $\SC{w}$ which we denote by $F|_{\SC{w}}$.
Recall also from Section~\ref{subsec: local description} that we may identify
$\C[{\bf x}_w]$ as a polynomial ring over $\C$ feely generated by the genuine variables $x_{i,j}$, i.e.,\ the ones satisfying $w^{-1}(i)>j$.
From the description of the Schubert cell $\SC{w}$ in Section~\ref{sec: background}, we have 
\begin{align}\label{eq: vanishing on Cw}
 x_{i,j}|_{\SC{w}} =0 
 \qquad \text{for all genuine variables $x_{i,j}$ satisfying $i>w(j)$}
\end{align}
as is well-known.

\begin{example}\label{ex: genuine variables 2}
Let $n=4$ and $w= 3142$ in one-line notation as in Example~\ref{ex: genuine variables}, where we visualized the genuine variables of the ring $\C[{\bf x}_w]$ in the matrix form 
\begin{align*}
\begin{pmatrix}
 x_{11} & 1 & & \\
 x_{21} & x_{22} & x_{23} & 1 \\
 1 & & & \\
 x_{41} & x_{42} & 1 & 
\end{pmatrix}.
\end{align*}
The equality \eqref{eq: vanishing on Cw} means that $x_{41}|_{\SC{w}}=x_{22}|_{\SC{w}}=x_{42}|_{\SC{w}}=0$ for the genuine variables lying below the 1s
in this matrix.
\end{example} 

\vspace{10pt}

The following claim is an immediate consequence of \eqref{eq: vanishing on Cw}, where we recall that the genuine variables $z_{i,j}=x_{i,w^{-1}(j)}$ in $\C[{\bf x}_w]$ are characterized in Lemma~\ref{lem: genuine zij}.

\begin{lemma}\label{lem: HS condition}
$z_{i,j}|_{\SC{w}} =0$ for all genuine variables $z_{i,j}\in\C[{\bf x}_w]$ 
satisfying $i>j$.
\end{lemma}

\vspace{10pt}

\begin{proposition}\label{prop: a partial result 1}
We have 
\begin{align*}
 \left.\frac{\partial}{\partial z_{p,q}}g_{i,j}\ \right|_{\SC{w}} = 0 
\end{align*}
for all genuine variables $z_{p,q}$ satisfying $p-q\le i-j$,
where $|_{\SC{w}}$ denotes the restriction after the differentiation.
\end{proposition}

\begin{proof} 
Let $\M$ be a monomial appearing in $g_{i,j}$.
If $\M$ does not contain $z_{p,q}$, then the claim is obvious.
Suppose that $\M$ contains $z_{p,q}$. Then the residual part $\M/z_{p,q}$ of $\M$ has the net difference of indexes $i-j+1-(p-q)$ by Corollary~\ref{cor:net_index}. 
Since this integer is positive by the assumption, it follows that the residual part $\M/z_{p,q}$ (and $\M$) must contain some $z_{k,l}$ with $k > l$. 
This implies that the differential $\frac{\partial}{\partial z_{p,q}} \M$ is proportional to $z_{k,l}$ (even when $(k,\ell)=(p,q)$). 
Since we have $z_{k,l}|_{\SC{w}}=0$
by Lemma~\ref{lem: HS condition}, the claim follows.
\end{proof}

\vspace{5pt}

\begin{example}  \label{ex: vanishing derivative}
Let $n=6$, and we continue with the example of $w=564321$; we computed $g_{2,5}$ in Example~\ref{ex:gij}.
Take $p=2$ and $q=6$ so that $2-6 < 2-5$. 
From the computation in Example~\ref{ex:gij}, 
we obtain that
$$\left.\frac{\partial}{\partial z_{2,6}}g_{2,5} \ \right|_{\SC{w}}= \left.z_{65}^2\ \right|_{\SC{w}} = 0, $$ where the last equality follows from Lemma~\ref{lem: HS condition}.
\end{example}

\begin{example}\label{ex: vanishing derivative 2}
Let $n=6$, and we continue with the example of $w=312654$; 
we computed $g_{6,3}$ in Example~\ref{eg: 312654-2}.
Take $p=4$ and $q=1$ so that $4-1 = 6-3$. 
From the computation in Example~\ref{eg: 312654-2}, we see that
\begin{align*}
 \left.\frac{\partial}{\partial z_{4,1}}g_{6,3} \ \right|_{\SC{w}}
 = 0.
\end{align*}
\end{example} 

\vspace{10pt}

So far, the permutation $w\in\Sn$ was fixed arbitrary.
As we mentioned in the beginning of Section~\ref{subsec: local description}, 
it suffices to study $\Hess(N,h)\cap wU^-$ for each $w$ such that $w_{\bullet} \in \Hess(N,h)$ because of \eqref{eq: open cover of Hess}. We now focus on these cases.
Recall from Lemma~\ref{lem: genuine zij} that the defining ideal $J_{w,h}\subseteq \C[{\bf x}_w]$ of $\Hess(N,h)\cap wU^-$ in $wU^-$ is generated by $g_{i,j}$ for $w^{-1}(i)>h(w^{-1}(j))$.

\begin{lemma}\label{lem: permutation flags in Hess with i > h(j)}
Suppose that $w_{\bullet} \in \Hess(N,h)$ and $w^{-1}(i)>h(w^{-1}(j))$. Then we have 
\begin{align*}
w^{-1}(j) < w^{-1}(i+1).
\end{align*}
for $1\le i<n$
\end{lemma}

\begin{proof}
If $w^{-1}(j) \ge w^{-1}(i+1)$, then we have $w^{-1}(i+1)=w^{-1}(j)-m$ for some $0\leq m<w^{-1}(j)$. 
Since we have $w_{\bullet}\in \Hess(N,h)$, this implies by Lemma~\ref{lem: permutation flags in Hess 2} that 
$$w^{-1}(i)
\leq h(w^{-1}(i+1))
= h(w^{-1}(j)-m)\leq h(w^{-1}(j))$$ 
which contradicts to the assumption $w^{-1}(i)>h(w^{-1}(j))$.
\end{proof}

\vspace{10pt}
 
\begin{proposition}\label{prop: linear term}
Suppose that $w_{\bullet} \in \Hess(N,h)$ and $w^{-1}(i) >h(w^{-1}(j) )$. The linear terms of $g_{i,j}$ are
\begin{align}\label{eq: linear term}
  - z_{i,j-1} + z_{i+1,j}
\end{align}
with the convention $z_{i,0}=z_{n+1,j}=0$.
Moreover, $z_{i,j-1}$ and $z_{i+1,j}$ in \eqref{eq: linear term} are genuine variables if $j-1\ne0$ and $i+1\ne n+1$, respectively.
\end{proposition}

\begin{proof}
Recall from \eqref{eq: gij in z 2} that we have
\begin{align}\label{eq: gij again}
g_{i,j} 
= \sum_{ \substack{1\leq w^{-1}(\ell-1) \leq w^{-1}(i) \\ w^{-1}(j)\leq w^{-1}(\ell) \leq n}} \  y_{w^{-1}(i),\ell-1} z_{\ell,j},
\end{align}
where we mean $y_{w^{-1}(i),0}=0$ by convention.
By \eqref{eq: convention of yij} and \eqref{eq: convention of zij}, the summands in \eqref{eq: gij again} which possibly contain a linear term of $g_{i,j}$ are the ones for $\ell=i+1$ or $\ell=j$, where we note that those summands might not appear in the summation because of the conditions on the running indexes $\ell$. We verify their appearance in what follows.

We first consider the summand for $\ell=i+1$.
When $i\ne n$, it appears in \eqref{eq: gij again} by Lemma~\ref{lem: permutation flags in Hess with i > h(j)}.
This summand is 
\begin{align*}
  y_{w^{-1}(i),\ell-1} z_{\ell,j} = 1\cdot z_{i+1,j} = z_{i+1,j}
\end{align*}
by \eqref{eq: convention of yij} which is one of the desired linear terms in \eqref{eq: linear term}. 
This $z_{i+1,j}$ is a genuine variable by Lemma~\ref{lem: genuine zij} and Lemma~\ref{lem: permutation flags in Hess with i > h(j)}.
When $i=n$, this summand does not appear in \eqref{eq: gij again}, and hence $g_{i,j}$ does not have the linear term corresponding to $z_{i+1,j}(=z_{n+1,j})$ in the claim of this proposition.

We next consider the summand for $\ell=j$.
It appears in \eqref{eq: gij again} since we have 
\begin{align*}
  w^{-1}(j-1)\le h(w^{-1}(j))<w^{-1}(i),
\end{align*}
where we used Lemma~\ref{lem: permutation flags in Hess 2} for the first inequality, and we used the assumption of this proposition for the second inequality. 
This summand is given by
\begin{align*}
  y_{w^{-1}(i),\ell-1} z_{\ell,j} = y_{w^{-1}(i),j-1} \cdot 1 = y_{w^{-1}(i),j-1}
\end{align*}
by \eqref{eq: convention of zij}. Since $w^{-1}(i) > w^{-1}(j-1)$ as we saw above, it follows from Proposition~\ref{prop:summand} that the linear terms of $y_{w^{-1}(i),j-1}$ is the other desired linear term $-z_{i,j-1}$ which is a genuine variable by Lemma~\ref{lem: genuine zij}.
When $j=1$, this summand does not appear in \eqref{eq: gij again} since $\ell-1=0$ in this case, and hence $g_{i,j}$ does not have the linear term corresponding to $-z_{i,j-1}(=z_{i,0})$ in the claim. This completes the proof.
\end{proof}

\vspace{10pt}

\begin{example}  \label{ex:linear term}
Let $n=6$ and $h=(3,4,4,5,6,6)$.
We continue with the example of $w=564321$.
One can verify that $w_{\bullet}\in\Hess(N,h)$ by Lemma~\ref{lem: permutation flags in Hess 2}.
We computed $g_{2,5}$ for this $w$ in Example~\ref{ex:gij}, and from that one can see that the linear terms of $g_{2,5}$ are precisely $-z_{24} +z_{35}.$
\end{example}

\begin{example} \label{ex:linear term 2}
Let $n=6$ and $h=(3,4,4,5,6,6)$.
We continue with the example of $w=312654$.
One can verify that $w_{\bullet}\in\Hess(N,h)$ by Lemma~\ref{lem: permutation flags in Hess 2}.
From the computation in Example~\ref{eg: 312654-2}, the linear term of $g_{6,3}$ is $-z_{6,2}$, where we note that $z_{7,2}(=0)$ does not appear in $g_{6,3}$ since $7=n+1$.
\end{example} 

\vspace{10pt}

Recall that $\C[{\bf x}_w]$ is the coordinate ring of the affine open subset $wU^- \subset \Flags(\C^n)$.
Since the Schubert cell $\SC{w}$ is contained in $wU^-$, we have
\begin{align*}
 \Hess(N,h)\cap \SC{w} \subseteq wU^- .
\end{align*}
This leads us to consider the following definition.

\begin{definition}\label{def: def of rest on HS}
For a function $F\in \C[{\bf x}_w]$, we denote by $F|_{\text{\rm HS cell}}$ the restriction of the function $F$ over the Hessenberg-Schubert cell $\Hess(N,h)\cap \SC{w}$.
\end{definition}

\begin{proposition}\label{prop: a partial result 2}
Suppose that $w_{\bullet} \in \Hess(N,h)$ and $w^{-1}(i) > h(w^{-1}(j))$. We have 
$$
 \left.\frac{\partial}{\partial z_{p,q}}g_{i,j}\ \right|_{\text{\rm HS cell}} = \begin{cases} - 1 & \text{ if } (p,q)=(i,j-1) \\ 1 & \text{ if } (p,q)=(i+1,j) \\ 0 &  \text{otherwise} \end{cases} 
$$
for all genuine variables $z_{p,q}$ satisfying $p-q = i-j+1 $,
where $|_{\text{\rm HS cell}}$ denotes the evaluation after differentiation.
\end{proposition}

\begin{proof} 
By Proposition~\ref{prop: linear term}, it suffices to show that 
$\frac{\partial}{\partial z_{p,q}} \M\ |_{\text{\rm HS cell}} = 0$
for each monomial $\M$ appearing in $g_{i,j}$ with $\deg \M\ge2$.

Let $\M$ be a monomial appearing in $g_{i,j}$ such that $\deg \M\ge2$.
If $\M$ does not contain $z_{p,q}$, then we clearly have $$  \frac{\partial}{\partial z_{p,q}} \M  = 0 .$$   
If $\M$ contains $z_{p,q}$, then the residual part $\M/z_{p,q}$ has the net difference of indexes $i-j+1-(p-q) = 0$ by Corollary~\ref{cor:net_index}. 
Noticing that $\deg (\M/z_{p,q})\ge1$, the residual part $\M/z_{p,q}$ (and $\M$) must contain at least one variable of the form $z_{k,l}$ with $k > l$ since $z_{k,k}=1$ for all $k$ by Lemma~\ref{lem: genuine zij}. 
This implies that the differential $\frac{\partial}{\partial z_{p,q}} \M$ is proportional to $z_{k,l}$ (even when $(k,\ell)=(p,q)$). 
By Lemma~\ref{lem: HS condition} we see that $z_{k,l} =0$ on the Hessenberg-Schubert cell, and thus we obtain the desired equality
$$ \left. \frac{\partial}{\partial z_{p,q}} \M \  \right|_{\text{\rm HS cell}} = 0 .$$   
\end{proof}

\vspace{10pt}
 
\begin{example}  \label{ex: derivative pm1}
Let $n=6$ and $h=(3,4,4,5,6,6)$.
We continue with the example of $w=564321$.
For example, $z_{2,4}$, $z_{3,5}$, $z_{1,3}$ are genuine variables in $\C[{\bf x}_w]$, and we have from Example~\ref{ex:gij} that
\begin{align*} 
&\left.\frac{\partial}{\partial z_{2,4}}g_{2,5} \ \right|_{\text{\rm HS cell}}= \left.(-1-z_{46}z_{65}^2+z_{45}z_{65})\ \right|_{\text{\rm HS cell}} = -1, \\
&\left.\frac{\partial}{\partial z_{3,5}}g_{2,5} \ \right|_{\text{\rm HS cell}}= \left.(1+z_{2,3}z_{6,5})\ \right|_{\text{\rm HS cell}} = 1, \\
&\left.\frac{\partial}{\partial z_{1,3}}g_{2,5} \ \right|_{\text{\rm HS cell}}=  0
\end{align*}
by Lemma~\ref{lem: HS condition}.
\end{example}

\begin{example}  \label{ex: derivative pm2}
Let $n=6$ and $h=(3,4,4,5,6,6)$.
We continue with the example of $w=312654$.
For example, $z_{6,2}$ is a genuine variable in $\C[{\bf x}_w]$, and we have from Example~\ref{eg: 312654-2} that
\begin{align*}
\left.\frac{\partial}{\partial z_{6,2}}g_{6,3} \ \right|_{\text{\rm HS cell}}= \left.(z_{21}z_{23}-1-z_{21}z_{13}z_{43}+z_{23}z_{43})\ \right|_{\text{\rm HS cell}} = -1
\end{align*}
by Lemma~\ref{lem: HS condition}.
\end{example}

\subsection{The \HC}\label{subsec: hc}

Continuing from Section~\ref{subsec: properties of gij}, we keep fixing a permutation $w \in \Sn$ to study the generators $g_{i,j}$ of the defining ideal $J_{w,h}\subseteq \C[{\bf x}_w]$ of $\Hess(N,h)\cap wU^-$ in $wU^-$.
In Section~\ref{subsec: preview}, we defined the \HC , denoted by $\hc{w}$, as follows:
\begin{equation}\label{eq:r_h}
\begin{split}
 \hc{w}
 &\coloneqq\{(w(i),w(j))\in[n]\times[n]\mid  i< h(j)\} \\
 &=\{(i,j)\in[n]\times[n] \mid w^{-1}(i)>h(w^{-1}(j))\}. 
\end{split}
\end{equation}
As we can see from Lemma~\ref{lem: genuine zij}, this set encodes the indexes of the generators $g_{i,j}$ appearing in the defining ideal $J_{w,h}$. In other words, the index $(i,j)$ of a generator $g_{i,j}$ corresponds to the location of $(i,j)\in \hc{w}$ in the square grid $[n]\times[n]$.

\begin{example}\label{eg: locations of gij}
Let $n=6$, and $h=(3,4,4,5,6,6)$. 
The Hessenberg function $h$ corresponds to the following configuration of boxes in the $n\times n$ grid.\\
\[%WinTpicVersion4.32a
{\unitlength 0.1in%
\begin{picture}(12.0000,12.0000)(34.0000,-16.0000)%
% BOX 2 0 1 0 Black Black  
% 2 3400 400 3600 600
% 
\special{pn 0}%
\special{sh 0.200}%
\special{pa 3400 400}%
\special{pa 3600 400}%
\special{pa 3600 600}%
\special{pa 3400 600}%
\special{pa 3400 400}%
\special{ip}%
\special{pn 8}%
\special{pa 3400 400}%
\special{pa 3600 400}%
\special{pa 3600 600}%
\special{pa 3400 600}%
\special{pa 3400 400}%
\special{pa 3600 400}%
\special{fp}%
% BOX 2 0 1 0 Black Black  
% 2 3600 400 3800 600
% 
\special{pn 0}%
\special{sh 0.200}%
\special{pa 3600 400}%
\special{pa 3800 400}%
\special{pa 3800 600}%
\special{pa 3600 600}%
\special{pa 3600 400}%
\special{ip}%
\special{pn 8}%
\special{pa 3600 400}%
\special{pa 3800 400}%
\special{pa 3800 600}%
\special{pa 3600 600}%
\special{pa 3600 400}%
\special{pa 3800 400}%
\special{fp}%
% BOX 2 0 1 0 Black Black  
% 2 3800 400 4000 600
% 
\special{pn 0}%
\special{sh 0.200}%
\special{pa 3800 400}%
\special{pa 4000 400}%
\special{pa 4000 600}%
\special{pa 3800 600}%
\special{pa 3800 400}%
\special{ip}%
\special{pn 8}%
\special{pa 3800 400}%
\special{pa 4000 400}%
\special{pa 4000 600}%
\special{pa 3800 600}%
\special{pa 3800 400}%
\special{pa 4000 400}%
\special{fp}%
% BOX 2 0 1 0 Black Black  
% 2 4000 400 4200 600
% 
\special{pn 0}%
\special{sh 0.200}%
\special{pa 4000 400}%
\special{pa 4200 400}%
\special{pa 4200 600}%
\special{pa 4000 600}%
\special{pa 4000 400}%
\special{ip}%
\special{pn 8}%
\special{pa 4000 400}%
\special{pa 4200 400}%
\special{pa 4200 600}%
\special{pa 4000 600}%
\special{pa 4000 400}%
\special{pa 4200 400}%
\special{fp}%
% BOX 2 0 1 0 Black Black  
% 2 4200 400 4400 600
% 
\special{pn 0}%
\special{sh 0.200}%
\special{pa 4200 400}%
\special{pa 4400 400}%
\special{pa 4400 600}%
\special{pa 4200 600}%
\special{pa 4200 400}%
\special{ip}%
\special{pn 8}%
\special{pa 4200 400}%
\special{pa 4400 400}%
\special{pa 4400 600}%
\special{pa 4200 600}%
\special{pa 4200 400}%
\special{pa 4400 400}%
\special{fp}%
% BOX 2 0 1 0 Black Black  
% 2 4400 400 4600 600
% 
\special{pn 0}%
\special{sh 0.200}%
\special{pa 4400 400}%
\special{pa 4600 400}%
\special{pa 4600 600}%
\special{pa 4400 600}%
\special{pa 4400 400}%
\special{ip}%
\special{pn 8}%
\special{pa 4400 400}%
\special{pa 4600 400}%
\special{pa 4600 600}%
\special{pa 4400 600}%
\special{pa 4400 400}%
\special{pa 4600 400}%
\special{fp}%
% BOX 2 0 1 0 Black Black  
% 2 3400 600 3600 800
% 
\special{pn 0}%
\special{sh 0.200}%
\special{pa 3400 600}%
\special{pa 3600 600}%
\special{pa 3600 800}%
\special{pa 3400 800}%
\special{pa 3400 600}%
\special{ip}%
\special{pn 8}%
\special{pa 3400 600}%
\special{pa 3600 600}%
\special{pa 3600 800}%
\special{pa 3400 800}%
\special{pa 3400 600}%
\special{pa 3600 600}%
\special{fp}%
% BOX 2 0 1 0 Black Black  
% 2 3600 600 3800 800
% 
\special{pn 0}%
\special{sh 0.200}%
\special{pa 3600 600}%
\special{pa 3800 600}%
\special{pa 3800 800}%
\special{pa 3600 800}%
\special{pa 3600 600}%
\special{ip}%
\special{pn 8}%
\special{pa 3600 600}%
\special{pa 3800 600}%
\special{pa 3800 800}%
\special{pa 3600 800}%
\special{pa 3600 600}%
\special{pa 3800 600}%
\special{fp}%
% BOX 2 0 1 0 Black Black  
% 2 3800 600 4000 800
% 
\special{pn 0}%
\special{sh 0.200}%
\special{pa 3800 600}%
\special{pa 4000 600}%
\special{pa 4000 800}%
\special{pa 3800 800}%
\special{pa 3800 600}%
\special{ip}%
\special{pn 8}%
\special{pa 3800 600}%
\special{pa 4000 600}%
\special{pa 4000 800}%
\special{pa 3800 800}%
\special{pa 3800 600}%
\special{pa 4000 600}%
\special{fp}%
% BOX 2 0 1 0 Black Black  
% 2 4000 600 4200 800
% 
\special{pn 0}%
\special{sh 0.200}%
\special{pa 4000 600}%
\special{pa 4200 600}%
\special{pa 4200 800}%
\special{pa 4000 800}%
\special{pa 4000 600}%
\special{ip}%
\special{pn 8}%
\special{pa 4000 600}%
\special{pa 4200 600}%
\special{pa 4200 800}%
\special{pa 4000 800}%
\special{pa 4000 600}%
\special{pa 4200 600}%
\special{fp}%
% BOX 2 0 1 0 Black Black  
% 2 4200 600 4400 800
% 
\special{pn 0}%
\special{sh 0.200}%
\special{pa 4200 600}%
\special{pa 4400 600}%
\special{pa 4400 800}%
\special{pa 4200 800}%
\special{pa 4200 600}%
\special{ip}%
\special{pn 8}%
\special{pa 4200 600}%
\special{pa 4400 600}%
\special{pa 4400 800}%
\special{pa 4200 800}%
\special{pa 4200 600}%
\special{pa 4400 600}%
\special{fp}%
% BOX 2 0 1 0 Black Black  
% 2 4400 600 4600 800
% 
\special{pn 0}%
\special{sh 0.200}%
\special{pa 4400 600}%
\special{pa 4600 600}%
\special{pa 4600 800}%
\special{pa 4400 800}%
\special{pa 4400 600}%
\special{ip}%
\special{pn 8}%
\special{pa 4400 600}%
\special{pa 4600 600}%
\special{pa 4600 800}%
\special{pa 4400 800}%
\special{pa 4400 600}%
\special{pa 4600 600}%
\special{fp}%
% BOX 2 0 1 0 Black Black  
% 2 3400 800 3600 1000
% 
\special{pn 0}%
\special{sh 0.200}%
\special{pa 3400 800}%
\special{pa 3600 800}%
\special{pa 3600 1000}%
\special{pa 3400 1000}%
\special{pa 3400 800}%
\special{ip}%
\special{pn 8}%
\special{pa 3400 800}%
\special{pa 3600 800}%
\special{pa 3600 1000}%
\special{pa 3400 1000}%
\special{pa 3400 800}%
\special{pa 3600 800}%
\special{fp}%
% BOX 2 0 1 0 Black Black  
% 2 3600 800 3800 1000
% 
\special{pn 0}%
\special{sh 0.200}%
\special{pa 3600 800}%
\special{pa 3800 800}%
\special{pa 3800 1000}%
\special{pa 3600 1000}%
\special{pa 3600 800}%
\special{ip}%
\special{pn 8}%
\special{pa 3600 800}%
\special{pa 3800 800}%
\special{pa 3800 1000}%
\special{pa 3600 1000}%
\special{pa 3600 800}%
\special{pa 3800 800}%
\special{fp}%
% BOX 2 0 1 0 Black Black  
% 2 3800 800 4000 1000
% 
\special{pn 0}%
\special{sh 0.200}%
\special{pa 3800 800}%
\special{pa 4000 800}%
\special{pa 4000 1000}%
\special{pa 3800 1000}%
\special{pa 3800 800}%
\special{ip}%
\special{pn 8}%
\special{pa 3800 800}%
\special{pa 4000 800}%
\special{pa 4000 1000}%
\special{pa 3800 1000}%
\special{pa 3800 800}%
\special{pa 4000 800}%
\special{fp}%
% BOX 2 0 1 0 Black Black  
% 2 4000 800 4200 1000
% 
\special{pn 0}%
\special{sh 0.200}%
\special{pa 4000 800}%
\special{pa 4200 800}%
\special{pa 4200 1000}%
\special{pa 4000 1000}%
\special{pa 4000 800}%
\special{ip}%
\special{pn 8}%
\special{pa 4000 800}%
\special{pa 4200 800}%
\special{pa 4200 1000}%
\special{pa 4000 1000}%
\special{pa 4000 800}%
\special{pa 4200 800}%
\special{fp}%
% BOX 2 0 1 0 Black Black  
% 2 4200 800 4400 1000
% 
\special{pn 0}%
\special{sh 0.200}%
\special{pa 4200 800}%
\special{pa 4400 800}%
\special{pa 4400 1000}%
\special{pa 4200 1000}%
\special{pa 4200 800}%
\special{ip}%
\special{pn 8}%
\special{pa 4200 800}%
\special{pa 4400 800}%
\special{pa 4400 1000}%
\special{pa 4200 1000}%
\special{pa 4200 800}%
\special{pa 4400 800}%
\special{fp}%
% BOX 2 0 1 0 Black Black  
% 2 4400 800 4600 1000
% 
\special{pn 0}%
\special{sh 0.200}%
\special{pa 4400 800}%
\special{pa 4600 800}%
\special{pa 4600 1000}%
\special{pa 4400 1000}%
\special{pa 4400 800}%
\special{ip}%
\special{pn 8}%
\special{pa 4400 800}%
\special{pa 4600 800}%
\special{pa 4600 1000}%
\special{pa 4400 1000}%
\special{pa 4400 800}%
\special{pa 4600 800}%
\special{fp}%
% BOX 2 0 3 0 Black White  
% 2 3400 1000 3600 1200
% 
\special{pn 8}%
\special{pa 3400 1000}%
\special{pa 3600 1000}%
\special{pa 3600 1200}%
\special{pa 3400 1200}%
\special{pa 3400 1000}%
\special{pa 3600 1000}%
\special{fp}%
% BOX 2 0 1 0 Black Black  
% 2 3600 1000 3800 1200
% 
\special{pn 0}%
\special{sh 0.200}%
\special{pa 3600 1000}%
\special{pa 3800 1000}%
\special{pa 3800 1200}%
\special{pa 3600 1200}%
\special{pa 3600 1000}%
\special{ip}%
\special{pn 8}%
\special{pa 3600 1000}%
\special{pa 3800 1000}%
\special{pa 3800 1200}%
\special{pa 3600 1200}%
\special{pa 3600 1000}%
\special{pa 3800 1000}%
\special{fp}%
% BOX 2 0 1 0 Black Black  
% 2 3800 1000 4000 1200
% 
\special{pn 0}%
\special{sh 0.200}%
\special{pa 3800 1000}%
\special{pa 4000 1000}%
\special{pa 4000 1200}%
\special{pa 3800 1200}%
\special{pa 3800 1000}%
\special{ip}%
\special{pn 8}%
\special{pa 3800 1000}%
\special{pa 4000 1000}%
\special{pa 4000 1200}%
\special{pa 3800 1200}%
\special{pa 3800 1000}%
\special{pa 4000 1000}%
\special{fp}%
% BOX 2 0 1 0 Black Black  
% 2 4000 1000 4200 1200
% 
\special{pn 0}%
\special{sh 0.200}%
\special{pa 4000 1000}%
\special{pa 4200 1000}%
\special{pa 4200 1200}%
\special{pa 4000 1200}%
\special{pa 4000 1000}%
\special{ip}%
\special{pn 8}%
\special{pa 4000 1000}%
\special{pa 4200 1000}%
\special{pa 4200 1200}%
\special{pa 4000 1200}%
\special{pa 4000 1000}%
\special{pa 4200 1000}%
\special{fp}%
% BOX 2 0 1 0 Black Black  
% 2 4200 1000 4400 1200
% 
\special{pn 0}%
\special{sh 0.200}%
\special{pa 4200 1000}%
\special{pa 4400 1000}%
\special{pa 4400 1200}%
\special{pa 4200 1200}%
\special{pa 4200 1000}%
\special{ip}%
\special{pn 8}%
\special{pa 4200 1000}%
\special{pa 4400 1000}%
\special{pa 4400 1200}%
\special{pa 4200 1200}%
\special{pa 4200 1000}%
\special{pa 4400 1000}%
\special{fp}%
% BOX 2 0 1 0 Black Black  
% 2 4400 1000 4600 1200
% 
\special{pn 0}%
\special{sh 0.200}%
\special{pa 4400 1000}%
\special{pa 4600 1000}%
\special{pa 4600 1200}%
\special{pa 4400 1200}%
\special{pa 4400 1000}%
\special{ip}%
\special{pn 8}%
\special{pa 4400 1000}%
\special{pa 4600 1000}%
\special{pa 4600 1200}%
\special{pa 4400 1200}%
\special{pa 4400 1000}%
\special{pa 4600 1000}%
\special{fp}%
% BOX 2 0 3 0 Black White  
% 2 3400 1200 3600 1400
% 
\special{pn 8}%
\special{pa 3400 1200}%
\special{pa 3600 1200}%
\special{pa 3600 1400}%
\special{pa 3400 1400}%
\special{pa 3400 1200}%
\special{pa 3600 1200}%
\special{fp}%
% BOX 2 0 3 0 Black White  
% 2 3600 1200 3800 1400
% 
\special{pn 8}%
\special{pa 3600 1200}%
\special{pa 3800 1200}%
\special{pa 3800 1400}%
\special{pa 3600 1400}%
\special{pa 3600 1200}%
\special{pa 3800 1200}%
\special{fp}%
% BOX 2 0 3 0 Black White  
% 2 3800 1200 4000 1400
% 
\special{pn 8}%
\special{pa 3800 1200}%
\special{pa 4000 1200}%
\special{pa 4000 1400}%
\special{pa 3800 1400}%
\special{pa 3800 1200}%
\special{pa 4000 1200}%
\special{fp}%
% BOX 2 0 1 0 Black Black  
% 2 4000 1200 4200 1400
% 
\special{pn 0}%
\special{sh 0.200}%
\special{pa 4000 1200}%
\special{pa 4200 1200}%
\special{pa 4200 1400}%
\special{pa 4000 1400}%
\special{pa 4000 1200}%
\special{ip}%
\special{pn 8}%
\special{pa 4000 1200}%
\special{pa 4200 1200}%
\special{pa 4200 1400}%
\special{pa 4000 1400}%
\special{pa 4000 1200}%
\special{pa 4200 1200}%
\special{fp}%
% BOX 2 0 1 0 Black Black  
% 2 4200 1200 4400 1400
% 
\special{pn 0}%
\special{sh 0.200}%
\special{pa 4200 1200}%
\special{pa 4400 1200}%
\special{pa 4400 1400}%
\special{pa 4200 1400}%
\special{pa 4200 1200}%
\special{ip}%
\special{pn 8}%
\special{pa 4200 1200}%
\special{pa 4400 1200}%
\special{pa 4400 1400}%
\special{pa 4200 1400}%
\special{pa 4200 1200}%
\special{pa 4400 1200}%
\special{fp}%
% BOX 2 0 1 0 Black Black  
% 2 4400 1200 4600 1400
% 
\special{pn 0}%
\special{sh 0.200}%
\special{pa 4400 1200}%
\special{pa 4600 1200}%
\special{pa 4600 1400}%
\special{pa 4400 1400}%
\special{pa 4400 1200}%
\special{ip}%
\special{pn 8}%
\special{pa 4400 1200}%
\special{pa 4600 1200}%
\special{pa 4600 1400}%
\special{pa 4400 1400}%
\special{pa 4400 1200}%
\special{pa 4600 1200}%
\special{fp}%
% BOX 2 0 3 0 Black White  
% 2 3400 1400 3600 1600
% 
\special{pn 8}%
\special{pa 3400 1400}%
\special{pa 3600 1400}%
\special{pa 3600 1600}%
\special{pa 3400 1600}%
\special{pa 3400 1400}%
\special{pa 3600 1400}%
\special{fp}%
% BOX 2 0 3 0 Black White  
% 2 3600 1400 3800 1600
% 
\special{pn 8}%
\special{pa 3600 1400}%
\special{pa 3800 1400}%
\special{pa 3800 1600}%
\special{pa 3600 1600}%
\special{pa 3600 1400}%
\special{pa 3800 1400}%
\special{fp}%
% BOX 2 0 3 0 Black White  
% 2 3800 1400 4000 1600
% 
\special{pn 8}%
\special{pa 3800 1400}%
\special{pa 4000 1400}%
\special{pa 4000 1600}%
\special{pa 3800 1600}%
\special{pa 3800 1400}%
\special{pa 4000 1400}%
\special{fp}%
% BOX 2 0 3 0 Black White  
% 2 4000 1400 4200 1600
% 
\special{pn 8}%
\special{pa 4000 1400}%
\special{pa 4200 1400}%
\special{pa 4200 1600}%
\special{pa 4000 1600}%
\special{pa 4000 1400}%
\special{pa 4200 1400}%
\special{fp}%
% BOX 2 0 1 0 Black Black  
% 2 4200 1400 4400 1600
% 
\special{pn 0}%
\special{sh 0.200}%
\special{pa 4200 1400}%
\special{pa 4400 1400}%
\special{pa 4400 1600}%
\special{pa 4200 1600}%
\special{pa 4200 1400}%
\special{ip}%
\special{pn 8}%
\special{pa 4200 1400}%
\special{pa 4400 1400}%
\special{pa 4400 1600}%
\special{pa 4200 1600}%
\special{pa 4200 1400}%
\special{pa 4400 1400}%
\special{fp}%
% BOX 2 0 1 0 Black Black  
% 2 4400 1400 4600 1600
% 
\special{pn 0}%
\special{sh 0.200}%
\special{pa 4400 1400}%
\special{pa 4600 1400}%
\special{pa 4600 1600}%
\special{pa 4400 1600}%
\special{pa 4400 1400}%
\special{ip}%
\special{pn 8}%
\special{pa 4400 1400}%
\special{pa 4600 1400}%
\special{pa 4600 1600}%
\special{pa 4400 1600}%
\special{pa 4400 1400}%
\special{pa 4600 1400}%
\special{fp}%
\end{picture}}%
\]
We now take $w=564321$ in one-line notation.
By permuting the rows and columns of the picture above by $w$, we obtain
\[
%WinTpicVersion4.32a
{\unitlength 0.1in%
\begin{picture}(12.0000,12.0000)(34.0000,-16.0000)%
% BOX 2 0 1 0 Black Black  
% 2 3400 400 3600 600
% 
\special{pn 0}%
\special{sh 0.200}%
\special{pa 3400 400}%
\special{pa 3600 400}%
\special{pa 3600 600}%
\special{pa 3400 600}%
\special{pa 3400 400}%
\special{ip}%
\special{pn 8}%
\special{pa 3400 400}%
\special{pa 3600 400}%
\special{pa 3600 600}%
\special{pa 3400 600}%
\special{pa 3400 400}%
\special{pa 3600 400}%
\special{fp}%
% BOX 2 0 1 0 Black Black  
% 2 3600 400 3800 600
% 
\special{pn 0}%
\special{sh 0.200}%
\special{pa 3600 400}%
\special{pa 3800 400}%
\special{pa 3800 600}%
\special{pa 3600 600}%
\special{pa 3600 400}%
\special{ip}%
\special{pn 8}%
\special{pa 3600 400}%
\special{pa 3800 400}%
\special{pa 3800 600}%
\special{pa 3600 600}%
\special{pa 3600 400}%
\special{pa 3800 400}%
\special{fp}%
% BOX 2 0 3 0 Black Black  
% 2 3800 400 4000 600
% 
\special{pn 8}%
\special{pa 3800 400}%
\special{pa 4000 400}%
\special{pa 4000 600}%
\special{pa 3800 600}%
\special{pa 3800 400}%
\special{pa 4000 400}%
\special{fp}%
% BOX 2 0 3 0 Black Black  
% 2 4000 400 4200 600
% 
\special{pn 8}%
\special{pa 4000 400}%
\special{pa 4200 400}%
\special{pa 4200 600}%
\special{pa 4000 600}%
\special{pa 4000 400}%
\special{pa 4200 400}%
\special{fp}%
% BOX 2 0 3 0 Black Black  
% 2 4200 400 4400 600
% 
\special{pn 8}%
\special{pa 4200 400}%
\special{pa 4400 400}%
\special{pa 4400 600}%
\special{pa 4200 600}%
\special{pa 4200 400}%
\special{pa 4400 400}%
\special{fp}%
% BOX 2 0 3 0 Black Black  
% 2 4400 400 4600 600
% 
\special{pn 8}%
\special{pa 4400 400}%
\special{pa 4600 400}%
\special{pa 4600 600}%
\special{pa 4400 600}%
\special{pa 4400 400}%
\special{pa 4600 400}%
\special{fp}%
% BOX 2 0 1 0 Black Black  
% 2 3400 600 3600 800
% 
\special{pn 0}%
\special{sh 0.200}%
\special{pa 3400 600}%
\special{pa 3600 600}%
\special{pa 3600 800}%
\special{pa 3400 800}%
\special{pa 3400 600}%
\special{ip}%
\special{pn 8}%
\special{pa 3400 600}%
\special{pa 3600 600}%
\special{pa 3600 800}%
\special{pa 3400 800}%
\special{pa 3400 600}%
\special{pa 3600 600}%
\special{fp}%
% BOX 2 0 1 0 Black Black  
% 2 3600 600 3800 800
% 
\special{pn 0}%
\special{sh 0.200}%
\special{pa 3600 600}%
\special{pa 3800 600}%
\special{pa 3800 800}%
\special{pa 3600 800}%
\special{pa 3600 600}%
\special{ip}%
\special{pn 8}%
\special{pa 3600 600}%
\special{pa 3800 600}%
\special{pa 3800 800}%
\special{pa 3600 800}%
\special{pa 3600 600}%
\special{pa 3800 600}%
\special{fp}%
% BOX 2 0 1 0 Black Black  
% 2 3800 600 4000 800
% 
\special{pn 0}%
\special{sh 0.200}%
\special{pa 3800 600}%
\special{pa 4000 600}%
\special{pa 4000 800}%
\special{pa 3800 800}%
\special{pa 3800 600}%
\special{ip}%
\special{pn 8}%
\special{pa 3800 600}%
\special{pa 4000 600}%
\special{pa 4000 800}%
\special{pa 3800 800}%
\special{pa 3800 600}%
\special{pa 4000 600}%
\special{fp}%
% BOX 2 0 3 0 Black Black  
% 2 4000 600 4200 800
% 
\special{pn 8}%
\special{pa 4000 600}%
\special{pa 4200 600}%
\special{pa 4200 800}%
\special{pa 4000 800}%
\special{pa 4000 600}%
\special{pa 4200 600}%
\special{fp}%
% BOX 2 0 3 0 Black Black  
% 2 4200 600 4400 800
% 
\special{pn 8}%
\special{pa 4200 600}%
\special{pa 4400 600}%
\special{pa 4400 800}%
\special{pa 4200 800}%
\special{pa 4200 600}%
\special{pa 4400 600}%
\special{fp}%
% BOX 2 0 3 0 Black Black  
% 2 4400 600 4600 800
% 
\special{pn 8}%
\special{pa 4400 600}%
\special{pa 4600 600}%
\special{pa 4600 800}%
\special{pa 4400 800}%
\special{pa 4400 600}%
\special{pa 4600 600}%
\special{fp}%
% BOX 2 0 1 0 Black Black  
% 2 3400 800 3600 1000
% 
\special{pn 0}%
\special{sh 0.200}%
\special{pa 3400 800}%
\special{pa 3600 800}%
\special{pa 3600 1000}%
\special{pa 3400 1000}%
\special{pa 3400 800}%
\special{ip}%
\special{pn 8}%
\special{pa 3400 800}%
\special{pa 3600 800}%
\special{pa 3600 1000}%
\special{pa 3400 1000}%
\special{pa 3400 800}%
\special{pa 3600 800}%
\special{fp}%
% BOX 2 0 1 0 Black Black  
% 2 3600 800 3800 1000
% 
\special{pn 0}%
\special{sh 0.200}%
\special{pa 3600 800}%
\special{pa 3800 800}%
\special{pa 3800 1000}%
\special{pa 3600 1000}%
\special{pa 3600 800}%
\special{ip}%
\special{pn 8}%
\special{pa 3600 800}%
\special{pa 3800 800}%
\special{pa 3800 1000}%
\special{pa 3600 1000}%
\special{pa 3600 800}%
\special{pa 3800 800}%
\special{fp}%
% BOX 2 0 1 0 Black Black  
% 2 3800 800 4000 1000
% 
\special{pn 0}%
\special{sh 0.200}%
\special{pa 3800 800}%
\special{pa 4000 800}%
\special{pa 4000 1000}%
\special{pa 3800 1000}%
\special{pa 3800 800}%
\special{ip}%
\special{pn 8}%
\special{pa 3800 800}%
\special{pa 4000 800}%
\special{pa 4000 1000}%
\special{pa 3800 1000}%
\special{pa 3800 800}%
\special{pa 4000 800}%
\special{fp}%
% BOX 2 0 1 0 Black Black  
% 2 4000 800 4200 1000
% 
\special{pn 0}%
\special{sh 0.200}%
\special{pa 4000 800}%
\special{pa 4200 800}%
\special{pa 4200 1000}%
\special{pa 4000 1000}%
\special{pa 4000 800}%
\special{ip}%
\special{pn 8}%
\special{pa 4000 800}%
\special{pa 4200 800}%
\special{pa 4200 1000}%
\special{pa 4000 1000}%
\special{pa 4000 800}%
\special{pa 4200 800}%
\special{fp}%
% BOX 2 0 3 0 Black Black  
% 2 4200 800 4400 1000
% 
\special{pn 8}%
\special{pa 4200 800}%
\special{pa 4400 800}%
\special{pa 4400 1000}%
\special{pa 4200 1000}%
\special{pa 4200 800}%
\special{pa 4400 800}%
\special{fp}%
% BOX 2 0 1 0 Black Black  
% 2 4400 800 4600 1000
% 
\special{pn 0}%
\special{sh 0.200}%
\special{pa 4400 800}%
\special{pa 4600 800}%
\special{pa 4600 1000}%
\special{pa 4400 1000}%
\special{pa 4400 800}%
\special{ip}%
\special{pn 8}%
\special{pa 4400 800}%
\special{pa 4600 800}%
\special{pa 4600 1000}%
\special{pa 4400 1000}%
\special{pa 4400 800}%
\special{pa 4600 800}%
\special{fp}%
% BOX 2 0 1 0 Black Black  
% 2 3400 1000 3600 1200
% 
\special{pn 0}%
\special{sh 0.200}%
\special{pa 3400 1000}%
\special{pa 3600 1000}%
\special{pa 3600 1200}%
\special{pa 3400 1200}%
\special{pa 3400 1000}%
\special{ip}%
\special{pn 8}%
\special{pa 3400 1000}%
\special{pa 3600 1000}%
\special{pa 3600 1200}%
\special{pa 3400 1200}%
\special{pa 3400 1000}%
\special{pa 3600 1000}%
\special{fp}%
% BOX 2 0 1 0 Black Black  
% 2 3600 1000 3800 1200
% 
\special{pn 0}%
\special{sh 0.200}%
\special{pa 3600 1000}%
\special{pa 3800 1000}%
\special{pa 3800 1200}%
\special{pa 3600 1200}%
\special{pa 3600 1000}%
\special{ip}%
\special{pn 8}%
\special{pa 3600 1000}%
\special{pa 3800 1000}%
\special{pa 3800 1200}%
\special{pa 3600 1200}%
\special{pa 3600 1000}%
\special{pa 3800 1000}%
\special{fp}%
% BOX 2 0 1 0 Black Black  
% 2 3800 1000 4000 1200
% 
\special{pn 0}%
\special{sh 0.200}%
\special{pa 3800 1000}%
\special{pa 4000 1000}%
\special{pa 4000 1200}%
\special{pa 3800 1200}%
\special{pa 3800 1000}%
\special{ip}%
\special{pn 8}%
\special{pa 3800 1000}%
\special{pa 4000 1000}%
\special{pa 4000 1200}%
\special{pa 3800 1200}%
\special{pa 3800 1000}%
\special{pa 4000 1000}%
\special{fp}%
% BOX 2 0 1 0 Black Black  
% 2 4000 1000 4200 1200
% 
\special{pn 0}%
\special{sh 0.200}%
\special{pa 4000 1000}%
\special{pa 4200 1000}%
\special{pa 4200 1200}%
\special{pa 4000 1200}%
\special{pa 4000 1000}%
\special{ip}%
\special{pn 8}%
\special{pa 4000 1000}%
\special{pa 4200 1000}%
\special{pa 4200 1200}%
\special{pa 4000 1200}%
\special{pa 4000 1000}%
\special{pa 4200 1000}%
\special{fp}%
% BOX 2 0 1 0 Black Black  
% 2 4200 1000 4400 1200
% 
\special{pn 0}%
\special{sh 0.200}%
\special{pa 4200 1000}%
\special{pa 4400 1000}%
\special{pa 4400 1200}%
\special{pa 4200 1200}%
\special{pa 4200 1000}%
\special{ip}%
\special{pn 8}%
\special{pa 4200 1000}%
\special{pa 4400 1000}%
\special{pa 4400 1200}%
\special{pa 4200 1200}%
\special{pa 4200 1000}%
\special{pa 4400 1000}%
\special{fp}%
% BOX 2 0 1 0 Black Black  
% 2 4400 1000 4600 1200
% 
\special{pn 0}%
\special{sh 0.200}%
\special{pa 4400 1000}%
\special{pa 4600 1000}%
\special{pa 4600 1200}%
\special{pa 4400 1200}%
\special{pa 4400 1000}%
\special{ip}%
\special{pn 8}%
\special{pa 4400 1000}%
\special{pa 4600 1000}%
\special{pa 4600 1200}%
\special{pa 4400 1200}%
\special{pa 4400 1000}%
\special{pa 4600 1000}%
\special{fp}%
% BOX 2 0 1 0 Black Black  
% 2 3400 1200 3600 1400
% 
\special{pn 0}%
\special{sh 0.200}%
\special{pa 3400 1200}%
\special{pa 3600 1200}%
\special{pa 3600 1400}%
\special{pa 3400 1400}%
\special{pa 3400 1200}%
\special{ip}%
\special{pn 8}%
\special{pa 3400 1200}%
\special{pa 3600 1200}%
\special{pa 3600 1400}%
\special{pa 3400 1400}%
\special{pa 3400 1200}%
\special{pa 3600 1200}%
\special{fp}%
% BOX 2 0 1 0 Black Black  
% 2 3600 1200 3800 1400
% 
\special{pn 0}%
\special{sh 0.200}%
\special{pa 3600 1200}%
\special{pa 3800 1200}%
\special{pa 3800 1400}%
\special{pa 3600 1400}%
\special{pa 3600 1200}%
\special{ip}%
\special{pn 8}%
\special{pa 3600 1200}%
\special{pa 3800 1200}%
\special{pa 3800 1400}%
\special{pa 3600 1400}%
\special{pa 3600 1200}%
\special{pa 3800 1200}%
\special{fp}%
% BOX 2 0 1 0 Black Black  
% 2 3800 1200 4000 1400
% 
\special{pn 0}%
\special{sh 0.200}%
\special{pa 3800 1200}%
\special{pa 4000 1200}%
\special{pa 4000 1400}%
\special{pa 3800 1400}%
\special{pa 3800 1200}%
\special{ip}%
\special{pn 8}%
\special{pa 3800 1200}%
\special{pa 4000 1200}%
\special{pa 4000 1400}%
\special{pa 3800 1400}%
\special{pa 3800 1200}%
\special{pa 4000 1200}%
\special{fp}%
% BOX 2 0 1 0 Black Black  
% 2 4000 1200 4200 1400
% 
\special{pn 0}%
\special{sh 0.200}%
\special{pa 4000 1200}%
\special{pa 4200 1200}%
\special{pa 4200 1400}%
\special{pa 4000 1400}%
\special{pa 4000 1200}%
\special{ip}%
\special{pn 8}%
\special{pa 4000 1200}%
\special{pa 4200 1200}%
\special{pa 4200 1400}%
\special{pa 4000 1400}%
\special{pa 4000 1200}%
\special{pa 4200 1200}%
\special{fp}%
% BOX 2 0 1 0 Black Black  
% 2 4200 1200 4400 1400
% 
\special{pn 0}%
\special{sh 0.200}%
\special{pa 4200 1200}%
\special{pa 4400 1200}%
\special{pa 4400 1400}%
\special{pa 4200 1400}%
\special{pa 4200 1200}%
\special{ip}%
\special{pn 8}%
\special{pa 4200 1200}%
\special{pa 4400 1200}%
\special{pa 4400 1400}%
\special{pa 4200 1400}%
\special{pa 4200 1200}%
\special{pa 4400 1200}%
\special{fp}%
% BOX 2 0 1 0 Black Black  
% 2 4400 1200 4600 1400
% 
\special{pn 0}%
\special{sh 0.200}%
\special{pa 4400 1200}%
\special{pa 4600 1200}%
\special{pa 4600 1400}%
\special{pa 4400 1400}%
\special{pa 4400 1200}%
\special{ip}%
\special{pn 8}%
\special{pa 4400 1200}%
\special{pa 4600 1200}%
\special{pa 4600 1400}%
\special{pa 4400 1400}%
\special{pa 4400 1200}%
\special{pa 4600 1200}%
\special{fp}%
% BOX 2 0 1 0 Black Black  
% 2 3400 1400 3600 1600
% 
\special{pn 0}%
\special{sh 0.200}%
\special{pa 3400 1400}%
\special{pa 3600 1400}%
\special{pa 3600 1600}%
\special{pa 3400 1600}%
\special{pa 3400 1400}%
\special{ip}%
\special{pn 8}%
\special{pa 3400 1400}%
\special{pa 3600 1400}%
\special{pa 3600 1600}%
\special{pa 3400 1600}%
\special{pa 3400 1400}%
\special{pa 3600 1400}%
\special{fp}%
% BOX 2 0 1 0 Black Black  
% 2 3600 1400 3800 1600
% 
\special{pn 0}%
\special{sh 0.200}%
\special{pa 3600 1400}%
\special{pa 3800 1400}%
\special{pa 3800 1600}%
\special{pa 3600 1600}%
\special{pa 3600 1400}%
\special{ip}%
\special{pn 8}%
\special{pa 3600 1400}%
\special{pa 3800 1400}%
\special{pa 3800 1600}%
\special{pa 3600 1600}%
\special{pa 3600 1400}%
\special{pa 3800 1400}%
\special{fp}%
% BOX 2 0 1 0 Black Black  
% 2 3800 1400 4000 1600
% 
\special{pn 0}%
\special{sh 0.200}%
\special{pa 3800 1400}%
\special{pa 4000 1400}%
\special{pa 4000 1600}%
\special{pa 3800 1600}%
\special{pa 3800 1400}%
\special{ip}%
\special{pn 8}%
\special{pa 3800 1400}%
\special{pa 4000 1400}%
\special{pa 4000 1600}%
\special{pa 3800 1600}%
\special{pa 3800 1400}%
\special{pa 4000 1400}%
\special{fp}%
% BOX 2 0 1 0 Black Black  
% 2 4000 1400 4200 1600
% 
\special{pn 0}%
\special{sh 0.200}%
\special{pa 4000 1400}%
\special{pa 4200 1400}%
\special{pa 4200 1600}%
\special{pa 4000 1600}%
\special{pa 4000 1400}%
\special{ip}%
\special{pn 8}%
\special{pa 4000 1400}%
\special{pa 4200 1400}%
\special{pa 4200 1600}%
\special{pa 4000 1600}%
\special{pa 4000 1400}%
\special{pa 4200 1400}%
\special{fp}%
% BOX 2 0 1 0 Black Black  
% 2 4200 1400 4400 1600
% 
\special{pn 0}%
\special{sh 0.200}%
\special{pa 4200 1400}%
\special{pa 4400 1400}%
\special{pa 4400 1600}%
\special{pa 4200 1600}%
\special{pa 4200 1400}%
\special{ip}%
\special{pn 8}%
\special{pa 4200 1400}%
\special{pa 4400 1400}%
\special{pa 4400 1600}%
\special{pa 4200 1600}%
\special{pa 4200 1400}%
\special{pa 4400 1400}%
\special{fp}%
% BOX 2 0 1 0 Black Black  
% 2 4400 1400 4600 1600
% 
\special{pn 0}%
\special{sh 0.200}%
\special{pa 4400 1400}%
\special{pa 4600 1400}%
\special{pa 4600 1600}%
\special{pa 4400 1600}%
\special{pa 4400 1400}%
\special{ip}%
\special{pn 8}%
\special{pa 4400 1400}%
\special{pa 4600 1400}%
\special{pa 4600 1600}%
\special{pa 4400 1600}%
\special{pa 4400 1400}%
\special{pa 4600 1400}%
\special{fp}%
\end{picture}}%
\]
The set of locations of white boxes in the resulting configuration is the set $\hc{w}$ :
\begin{equation}\label{eq: 312654 10}
 \hc{w}=\{ (1,3), \ (1,4), \ (1,5), \ (1,6), \ (2,4), \ (2,5), \ (2,6),\ (3,5)\}.
\end{equation}
Accordingly, the generators of the ideal $J_{w,h}$ are 
\begin{equation*}
 g_{1,3}, \ g_{1,4}, \ g_{1,5}, \ g_{1,6}, \ g_{2,4}, \ g_{2,5}, \ g_{2,6}, \ g_{3,5}.
\end{equation*}
The indexes appearing here are exactly the elements appearing in \eqref{eq: 312654 10}. Namely, the locations of white boxes describes the indexes of the generators as exhibited in the following picture.\\
\[
%WinTpicVersion4.32a
{\unitlength 0.1in%
\begin{picture}(15.9800,15.9800)(34.0200,-29.9800)%
% BOX 2 0 1 0 Black Black  
% 2 3402 1400 3668 1666
% 
\special{pn 0}%
\special{sh 0.200}%
\special{pa 3402 1400}%
\special{pa 3668 1400}%
\special{pa 3668 1666}%
\special{pa 3402 1666}%
\special{pa 3402 1400}%
\special{ip}%
\special{pn 8}%
\special{pa 3402 1400}%
\special{pa 3668 1400}%
\special{pa 3668 1666}%
\special{pa 3402 1666}%
\special{pa 3402 1400}%
\special{pa 3668 1400}%
\special{fp}%
% BOX 2 0 1 0 Black Black  
% 2 3668 1400 3934 1666
% 
\special{pn 0}%
\special{sh 0.200}%
\special{pa 3668 1400}%
\special{pa 3934 1400}%
\special{pa 3934 1666}%
\special{pa 3668 1666}%
\special{pa 3668 1400}%
\special{ip}%
\special{pn 8}%
\special{pa 3668 1400}%
\special{pa 3934 1400}%
\special{pa 3934 1666}%
\special{pa 3668 1666}%
\special{pa 3668 1400}%
\special{pa 3934 1400}%
\special{fp}%
% BOX 2 0 3 0 Black Black  
% 2 3934 1400 4201 1666
% 
\special{pn 8}%
\special{pa 3934 1400}%
\special{pa 4201 1400}%
\special{pa 4201 1666}%
\special{pa 3934 1666}%
\special{pa 3934 1400}%
\special{pa 4201 1400}%
\special{fp}%
% BOX 2 0 3 0 Black Black  
% 2 4201 1400 4467 1666
% 
\special{pn 8}%
\special{pa 4201 1400}%
\special{pa 4467 1400}%
\special{pa 4467 1666}%
\special{pa 4201 1666}%
\special{pa 4201 1400}%
\special{pa 4467 1400}%
\special{fp}%
% BOX 2 0 3 0 Black Black  
% 2 4467 1400 4733 1666
% 
\special{pn 8}%
\special{pa 4467 1400}%
\special{pa 4733 1400}%
\special{pa 4733 1666}%
\special{pa 4467 1666}%
\special{pa 4467 1400}%
\special{pa 4733 1400}%
\special{fp}%
% BOX 2 0 3 0 Black Black  
% 2 4733 1400 5000 1666
% 
\special{pn 8}%
\special{pa 4733 1400}%
\special{pa 5000 1400}%
\special{pa 5000 1666}%
\special{pa 4733 1666}%
\special{pa 4733 1400}%
\special{pa 5000 1400}%
\special{fp}%
% BOX 2 0 1 0 Black Black  
% 2 3402 1666 3668 1933
% 
\special{pn 0}%
\special{sh 0.200}%
\special{pa 3402 1666}%
\special{pa 3668 1666}%
\special{pa 3668 1933}%
\special{pa 3402 1933}%
\special{pa 3402 1666}%
\special{ip}%
\special{pn 8}%
\special{pa 3402 1666}%
\special{pa 3668 1666}%
\special{pa 3668 1933}%
\special{pa 3402 1933}%
\special{pa 3402 1666}%
\special{pa 3668 1666}%
\special{fp}%
% BOX 2 0 1 0 Black Black  
% 2 3668 1666 3934 1933
% 
\special{pn 0}%
\special{sh 0.200}%
\special{pa 3668 1666}%
\special{pa 3934 1666}%
\special{pa 3934 1933}%
\special{pa 3668 1933}%
\special{pa 3668 1666}%
\special{ip}%
\special{pn 8}%
\special{pa 3668 1666}%
\special{pa 3934 1666}%
\special{pa 3934 1933}%
\special{pa 3668 1933}%
\special{pa 3668 1666}%
\special{pa 3934 1666}%
\special{fp}%
% BOX 2 0 1 0 Black Black  
% 2 3934 1666 4201 1933
% 
\special{pn 0}%
\special{sh 0.200}%
\special{pa 3934 1666}%
\special{pa 4201 1666}%
\special{pa 4201 1933}%
\special{pa 3934 1933}%
\special{pa 3934 1666}%
\special{ip}%
\special{pn 8}%
\special{pa 3934 1666}%
\special{pa 4201 1666}%
\special{pa 4201 1933}%
\special{pa 3934 1933}%
\special{pa 3934 1666}%
\special{pa 4201 1666}%
\special{fp}%
% BOX 2 0 3 0 Black Black  
% 2 4201 1666 4467 1933
% 
\special{pn 8}%
\special{pa 4201 1666}%
\special{pa 4467 1666}%
\special{pa 4467 1933}%
\special{pa 4201 1933}%
\special{pa 4201 1666}%
\special{pa 4467 1666}%
\special{fp}%
% BOX 2 0 3 0 Black Black  
% 2 4467 1666 4733 1933
% 
\special{pn 8}%
\special{pa 4467 1666}%
\special{pa 4733 1666}%
\special{pa 4733 1933}%
\special{pa 4467 1933}%
\special{pa 4467 1666}%
\special{pa 4733 1666}%
\special{fp}%
% BOX 2 0 3 0 Black Black  
% 2 4733 1666 5000 1933
% 
\special{pn 8}%
\special{pa 4733 1666}%
\special{pa 5000 1666}%
\special{pa 5000 1933}%
\special{pa 4733 1933}%
\special{pa 4733 1666}%
\special{pa 5000 1666}%
\special{fp}%
% BOX 2 0 1 0 Black Black  
% 2 3402 1933 3668 2199
% 
\special{pn 0}%
\special{sh 0.200}%
\special{pa 3402 1933}%
\special{pa 3668 1933}%
\special{pa 3668 2199}%
\special{pa 3402 2199}%
\special{pa 3402 1933}%
\special{ip}%
\special{pn 8}%
\special{pa 3402 1933}%
\special{pa 3668 1933}%
\special{pa 3668 2199}%
\special{pa 3402 2199}%
\special{pa 3402 1933}%
\special{pa 3668 1933}%
\special{fp}%
% BOX 2 0 1 0 Black Black  
% 2 3668 1933 3934 2199
% 
\special{pn 0}%
\special{sh 0.200}%
\special{pa 3668 1933}%
\special{pa 3934 1933}%
\special{pa 3934 2199}%
\special{pa 3668 2199}%
\special{pa 3668 1933}%
\special{ip}%
\special{pn 8}%
\special{pa 3668 1933}%
\special{pa 3934 1933}%
\special{pa 3934 2199}%
\special{pa 3668 2199}%
\special{pa 3668 1933}%
\special{pa 3934 1933}%
\special{fp}%
% BOX 2 0 1 0 Black Black  
% 2 3934 1933 4201 2199
% 
\special{pn 0}%
\special{sh 0.200}%
\special{pa 3934 1933}%
\special{pa 4201 1933}%
\special{pa 4201 2199}%
\special{pa 3934 2199}%
\special{pa 3934 1933}%
\special{ip}%
\special{pn 8}%
\special{pa 3934 1933}%
\special{pa 4201 1933}%
\special{pa 4201 2199}%
\special{pa 3934 2199}%
\special{pa 3934 1933}%
\special{pa 4201 1933}%
\special{fp}%
% BOX 2 0 1 0 Black Black  
% 2 4201 1933 4467 2199
% 
\special{pn 0}%
\special{sh 0.200}%
\special{pa 4201 1933}%
\special{pa 4467 1933}%
\special{pa 4467 2199}%
\special{pa 4201 2199}%
\special{pa 4201 1933}%
\special{ip}%
\special{pn 8}%
\special{pa 4201 1933}%
\special{pa 4467 1933}%
\special{pa 4467 2199}%
\special{pa 4201 2199}%
\special{pa 4201 1933}%
\special{pa 4467 1933}%
\special{fp}%
% BOX 2 0 3 0 Black Black  
% 2 4467 1933 4733 2199
% 
\special{pn 8}%
\special{pa 4467 1933}%
\special{pa 4733 1933}%
\special{pa 4733 2199}%
\special{pa 4467 2199}%
\special{pa 4467 1933}%
\special{pa 4733 1933}%
\special{fp}%
% BOX 2 0 1 0 Black Black  
% 2 4733 1933 5000 2199
% 
\special{pn 0}%
\special{sh 0.200}%
\special{pa 4733 1933}%
\special{pa 5000 1933}%
\special{pa 5000 2199}%
\special{pa 4733 2199}%
\special{pa 4733 1933}%
\special{ip}%
\special{pn 8}%
\special{pa 4733 1933}%
\special{pa 5000 1933}%
\special{pa 5000 2199}%
\special{pa 4733 2199}%
\special{pa 4733 1933}%
\special{pa 5000 1933}%
\special{fp}%
% BOX 2 0 1 0 Black Black  
% 2 3402 2199 3668 2465
% 
\special{pn 0}%
\special{sh 0.200}%
\special{pa 3402 2199}%
\special{pa 3668 2199}%
\special{pa 3668 2465}%
\special{pa 3402 2465}%
\special{pa 3402 2199}%
\special{ip}%
\special{pn 8}%
\special{pa 3402 2199}%
\special{pa 3668 2199}%
\special{pa 3668 2465}%
\special{pa 3402 2465}%
\special{pa 3402 2199}%
\special{pa 3668 2199}%
\special{fp}%
% BOX 2 0 1 0 Black Black  
% 2 3668 2199 3934 2465
% 
\special{pn 0}%
\special{sh 0.200}%
\special{pa 3668 2199}%
\special{pa 3934 2199}%
\special{pa 3934 2465}%
\special{pa 3668 2465}%
\special{pa 3668 2199}%
\special{ip}%
\special{pn 8}%
\special{pa 3668 2199}%
\special{pa 3934 2199}%
\special{pa 3934 2465}%
\special{pa 3668 2465}%
\special{pa 3668 2199}%
\special{pa 3934 2199}%
\special{fp}%
% BOX 2 0 1 0 Black Black  
% 2 3934 2199 4201 2465
% 
\special{pn 0}%
\special{sh 0.200}%
\special{pa 3934 2199}%
\special{pa 4201 2199}%
\special{pa 4201 2465}%
\special{pa 3934 2465}%
\special{pa 3934 2199}%
\special{ip}%
\special{pn 8}%
\special{pa 3934 2199}%
\special{pa 4201 2199}%
\special{pa 4201 2465}%
\special{pa 3934 2465}%
\special{pa 3934 2199}%
\special{pa 4201 2199}%
\special{fp}%
% BOX 2 0 1 0 Black Black  
% 2 4201 2199 4467 2465
% 
\special{pn 0}%
\special{sh 0.200}%
\special{pa 4201 2199}%
\special{pa 4467 2199}%
\special{pa 4467 2465}%
\special{pa 4201 2465}%
\special{pa 4201 2199}%
\special{ip}%
\special{pn 8}%
\special{pa 4201 2199}%
\special{pa 4467 2199}%
\special{pa 4467 2465}%
\special{pa 4201 2465}%
\special{pa 4201 2199}%
\special{pa 4467 2199}%
\special{fp}%
% BOX 2 0 1 0 Black Black  
% 2 4467 2199 4733 2465
% 
\special{pn 0}%
\special{sh 0.200}%
\special{pa 4467 2199}%
\special{pa 4733 2199}%
\special{pa 4733 2465}%
\special{pa 4467 2465}%
\special{pa 4467 2199}%
\special{ip}%
\special{pn 8}%
\special{pa 4467 2199}%
\special{pa 4733 2199}%
\special{pa 4733 2465}%
\special{pa 4467 2465}%
\special{pa 4467 2199}%
\special{pa 4733 2199}%
\special{fp}%
% BOX 2 0 1 0 Black Black  
% 2 4733 2199 5000 2465
% 
\special{pn 0}%
\special{sh 0.200}%
\special{pa 4733 2199}%
\special{pa 5000 2199}%
\special{pa 5000 2465}%
\special{pa 4733 2465}%
\special{pa 4733 2199}%
\special{ip}%
\special{pn 8}%
\special{pa 4733 2199}%
\special{pa 5000 2199}%
\special{pa 5000 2465}%
\special{pa 4733 2465}%
\special{pa 4733 2199}%
\special{pa 5000 2199}%
\special{fp}%
% BOX 2 0 1 0 Black Black  
% 2 3402 2465 3668 2732
% 
\special{pn 0}%
\special{sh 0.200}%
\special{pa 3402 2465}%
\special{pa 3668 2465}%
\special{pa 3668 2732}%
\special{pa 3402 2732}%
\special{pa 3402 2465}%
\special{ip}%
\special{pn 8}%
\special{pa 3402 2465}%
\special{pa 3668 2465}%
\special{pa 3668 2732}%
\special{pa 3402 2732}%
\special{pa 3402 2465}%
\special{pa 3668 2465}%
\special{fp}%
% BOX 2 0 1 0 Black Black  
% 2 3668 2465 3934 2732
% 
\special{pn 0}%
\special{sh 0.200}%
\special{pa 3668 2465}%
\special{pa 3934 2465}%
\special{pa 3934 2732}%
\special{pa 3668 2732}%
\special{pa 3668 2465}%
\special{ip}%
\special{pn 8}%
\special{pa 3668 2465}%
\special{pa 3934 2465}%
\special{pa 3934 2732}%
\special{pa 3668 2732}%
\special{pa 3668 2465}%
\special{pa 3934 2465}%
\special{fp}%
% BOX 2 0 1 0 Black Black  
% 2 3934 2465 4201 2732
% 
\special{pn 0}%
\special{sh 0.200}%
\special{pa 3934 2465}%
\special{pa 4201 2465}%
\special{pa 4201 2732}%
\special{pa 3934 2732}%
\special{pa 3934 2465}%
\special{ip}%
\special{pn 8}%
\special{pa 3934 2465}%
\special{pa 4201 2465}%
\special{pa 4201 2732}%
\special{pa 3934 2732}%
\special{pa 3934 2465}%
\special{pa 4201 2465}%
\special{fp}%
% BOX 2 0 1 0 Black Black  
% 2 4201 2465 4467 2732
% 
\special{pn 0}%
\special{sh 0.200}%
\special{pa 4201 2465}%
\special{pa 4467 2465}%
\special{pa 4467 2732}%
\special{pa 4201 2732}%
\special{pa 4201 2465}%
\special{ip}%
\special{pn 8}%
\special{pa 4201 2465}%
\special{pa 4467 2465}%
\special{pa 4467 2732}%
\special{pa 4201 2732}%
\special{pa 4201 2465}%
\special{pa 4467 2465}%
\special{fp}%
% BOX 2 0 1 0 Black Black  
% 2 4467 2465 4733 2732
% 
\special{pn 0}%
\special{sh 0.200}%
\special{pa 4467 2465}%
\special{pa 4733 2465}%
\special{pa 4733 2732}%
\special{pa 4467 2732}%
\special{pa 4467 2465}%
\special{ip}%
\special{pn 8}%
\special{pa 4467 2465}%
\special{pa 4733 2465}%
\special{pa 4733 2732}%
\special{pa 4467 2732}%
\special{pa 4467 2465}%
\special{pa 4733 2465}%
\special{fp}%
% BOX 2 0 1 0 Black Black  
% 2 4733 2465 5000 2732
% 
\special{pn 0}%
\special{sh 0.200}%
\special{pa 4733 2465}%
\special{pa 5000 2465}%
\special{pa 5000 2732}%
\special{pa 4733 2732}%
\special{pa 4733 2465}%
\special{ip}%
\special{pn 8}%
\special{pa 4733 2465}%
\special{pa 5000 2465}%
\special{pa 5000 2732}%
\special{pa 4733 2732}%
\special{pa 4733 2465}%
\special{pa 5000 2465}%
\special{fp}%
% BOX 2 0 1 0 Black Black  
% 2 3402 2732 3668 2998
% 
\special{pn 0}%
\special{sh 0.200}%
\special{pa 3402 2732}%
\special{pa 3668 2732}%
\special{pa 3668 2998}%
\special{pa 3402 2998}%
\special{pa 3402 2732}%
\special{ip}%
\special{pn 8}%
\special{pa 3402 2732}%
\special{pa 3668 2732}%
\special{pa 3668 2998}%
\special{pa 3402 2998}%
\special{pa 3402 2732}%
\special{pa 3668 2732}%
\special{fp}%
% BOX 2 0 1 0 Black Black  
% 2 3668 2732 3934 2998
% 
\special{pn 0}%
\special{sh 0.200}%
\special{pa 3668 2732}%
\special{pa 3934 2732}%
\special{pa 3934 2998}%
\special{pa 3668 2998}%
\special{pa 3668 2732}%
\special{ip}%
\special{pn 8}%
\special{pa 3668 2732}%
\special{pa 3934 2732}%
\special{pa 3934 2998}%
\special{pa 3668 2998}%
\special{pa 3668 2732}%
\special{pa 3934 2732}%
\special{fp}%
% BOX 2 0 1 0 Black Black  
% 2 3934 2732 4201 2998
% 
\special{pn 0}%
\special{sh 0.200}%
\special{pa 3934 2732}%
\special{pa 4201 2732}%
\special{pa 4201 2998}%
\special{pa 3934 2998}%
\special{pa 3934 2732}%
\special{ip}%
\special{pn 8}%
\special{pa 3934 2732}%
\special{pa 4201 2732}%
\special{pa 4201 2998}%
\special{pa 3934 2998}%
\special{pa 3934 2732}%
\special{pa 4201 2732}%
\special{fp}%
% BOX 2 0 1 0 Black Black  
% 2 4201 2732 4467 2998
% 
\special{pn 0}%
\special{sh 0.200}%
\special{pa 4201 2732}%
\special{pa 4467 2732}%
\special{pa 4467 2998}%
\special{pa 4201 2998}%
\special{pa 4201 2732}%
\special{ip}%
\special{pn 8}%
\special{pa 4201 2732}%
\special{pa 4467 2732}%
\special{pa 4467 2998}%
\special{pa 4201 2998}%
\special{pa 4201 2732}%
\special{pa 4467 2732}%
\special{fp}%
% BOX 2 0 1 0 Black Black  
% 2 4467 2732 4733 2998
% 
\special{pn 0}%
\special{sh 0.200}%
\special{pa 4467 2732}%
\special{pa 4733 2732}%
\special{pa 4733 2998}%
\special{pa 4467 2998}%
\special{pa 4467 2732}%
\special{ip}%
\special{pn 8}%
\special{pa 4467 2732}%
\special{pa 4733 2732}%
\special{pa 4733 2998}%
\special{pa 4467 2998}%
\special{pa 4467 2732}%
\special{pa 4733 2732}%
\special{fp}%
% BOX 2 0 1 0 Black Black  
% 2 4733 2732 5000 2998
% 
\special{pn 0}%
\special{sh 0.200}%
\special{pa 4733 2732}%
\special{pa 5000 2732}%
\special{pa 5000 2998}%
\special{pa 4733 2998}%
\special{pa 4733 2732}%
\special{ip}%
\special{pn 8}%
\special{pa 4733 2732}%
\special{pa 5000 2732}%
\special{pa 5000 2998}%
\special{pa 4733 2998}%
\special{pa 4733 2732}%
\special{pa 5000 2732}%
\special{fp}%
% STR 2 0 3 0 Black White  
% 4 4220 1466 4220 1600 2 0 0 0
% $g_{1,4}$
\put(42.2000,-16.0000){\makebox(0,0)[lb]{$g_{1,4}$}}%
% STR 2 0 3 0 Black White  
% 4 3960 1466 3960 1600 2 0 0 0
% $g_{1,3}$
\put(39.6000,-16.0000){\makebox(0,0)[lb]{$g_{1,3}$}}%
% STR 2 0 3 0 Black White  
% 4 4760 1466 4760 1600 2 0 0 0
% $g_{1,6}$
\put(47.6000,-16.0000){\makebox(0,0)[lb]{$g_{1,6}$}}%
% STR 2 0 3 0 Black White  
% 4 4490 1466 4490 1600 2 0 0 0
% $g_{1,5}$
\put(44.9000,-16.0000){\makebox(0,0)[lb]{$g_{1,5}$}}%
% STR 2 0 3 0 Black White  
% 4 4220 1726 4220 1860 2 0 0 0
% $g_{2,4}$
\put(42.2000,-18.6000){\makebox(0,0)[lb]{$g_{2,4}$}}%
% STR 2 0 3 0 Black White  
% 4 4760 1726 4760 1860 2 0 0 0
% $g_{2,6}$
\put(47.6000,-18.6000){\makebox(0,0)[lb]{$g_{2,6}$}}%
% STR 2 0 3 0 Black White  
% 4 4490 1726 4490 1860 2 0 0 0
% $g_{2,5}$
\put(44.9000,-18.6000){\makebox(0,0)[lb]{$g_{2,5}$}}%
% STR 2 0 3 0 Black White  
% 4 4490 1986 4490 2120 2 0 0 0
% $g_{3,5}$
\put(44.9000,-21.2000){\makebox(0,0)[lb]{$g_{3,5}$}}%
\end{picture}}%
\]
\end{example}

\vspace{5pt}
 
\begin{example}\label{eg: locations of gij 2}
Let $n=6$, and $h=(3,4,4,5,6,6)$ as above. 
For $w=312654$, $\hc{w}$ and the generators of the ideal $J_{w,h}$ are depicted in the following picture.\\
\[
%WinTpicVersion4.32a
{\unitlength 0.1in%
\begin{picture}(15.9800,15.9800)(34.0200,-29.9800)%
% BOX 2 0 1 0 Black Black  
% 2 3402 1400 3668 1666
% 
\special{pn 0}%
\special{sh 0.200}%
\special{pa 3402 1400}%
\special{pa 3668 1400}%
\special{pa 3668 1666}%
\special{pa 3402 1666}%
\special{pa 3402 1400}%
\special{ip}%
\special{pn 8}%
\special{pa 3402 1400}%
\special{pa 3668 1400}%
\special{pa 3668 1666}%
\special{pa 3402 1666}%
\special{pa 3402 1400}%
\special{pa 3668 1400}%
\special{fp}%
% BOX 2 0 1 0 Black Black  
% 2 3668 1400 3934 1666
% 
\special{pn 0}%
\special{sh 0.200}%
\special{pa 3668 1400}%
\special{pa 3934 1400}%
\special{pa 3934 1666}%
\special{pa 3668 1666}%
\special{pa 3668 1400}%
\special{ip}%
\special{pn 8}%
\special{pa 3668 1400}%
\special{pa 3934 1400}%
\special{pa 3934 1666}%
\special{pa 3668 1666}%
\special{pa 3668 1400}%
\special{pa 3934 1400}%
\special{fp}%
% BOX 2 0 1 0 Black Black  
% 2 3934 1400 4201 1666
% 
\special{pn 0}%
\special{sh 0.200}%
\special{pa 3934 1400}%
\special{pa 4201 1400}%
\special{pa 4201 1666}%
\special{pa 3934 1666}%
\special{pa 3934 1400}%
\special{ip}%
\special{pn 8}%
\special{pa 3934 1400}%
\special{pa 4201 1400}%
\special{pa 4201 1666}%
\special{pa 3934 1666}%
\special{pa 3934 1400}%
\special{pa 4201 1400}%
\special{fp}%
% BOX 2 0 1 0 Black Black  
% 2 4201 1400 4467 1666
% 
\special{pn 0}%
\special{sh 0.200}%
\special{pa 4201 1400}%
\special{pa 4467 1400}%
\special{pa 4467 1666}%
\special{pa 4201 1666}%
\special{pa 4201 1400}%
\special{ip}%
\special{pn 8}%
\special{pa 4201 1400}%
\special{pa 4467 1400}%
\special{pa 4467 1666}%
\special{pa 4201 1666}%
\special{pa 4201 1400}%
\special{pa 4467 1400}%
\special{fp}%
% BOX 2 0 1 0 Black Black  
% 2 4467 1400 4733 1666
% 
\special{pn 0}%
\special{sh 0.200}%
\special{pa 4467 1400}%
\special{pa 4733 1400}%
\special{pa 4733 1666}%
\special{pa 4467 1666}%
\special{pa 4467 1400}%
\special{ip}%
\special{pn 8}%
\special{pa 4467 1400}%
\special{pa 4733 1400}%
\special{pa 4733 1666}%
\special{pa 4467 1666}%
\special{pa 4467 1400}%
\special{pa 4733 1400}%
\special{fp}%
% BOX 2 0 1 0 Black Black  
% 2 4733 1400 5000 1666
% 
\special{pn 0}%
\special{sh 0.200}%
\special{pa 4733 1400}%
\special{pa 5000 1400}%
\special{pa 5000 1666}%
\special{pa 4733 1666}%
\special{pa 4733 1400}%
\special{ip}%
\special{pn 8}%
\special{pa 4733 1400}%
\special{pa 5000 1400}%
\special{pa 5000 1666}%
\special{pa 4733 1666}%
\special{pa 4733 1400}%
\special{pa 5000 1400}%
\special{fp}%
% BOX 2 0 1 0 Black Black  
% 2 3402 1666 3668 1933
% 
\special{pn 0}%
\special{sh 0.200}%
\special{pa 3402 1666}%
\special{pa 3668 1666}%
\special{pa 3668 1933}%
\special{pa 3402 1933}%
\special{pa 3402 1666}%
\special{ip}%
\special{pn 8}%
\special{pa 3402 1666}%
\special{pa 3668 1666}%
\special{pa 3668 1933}%
\special{pa 3402 1933}%
\special{pa 3402 1666}%
\special{pa 3668 1666}%
\special{fp}%
% BOX 2 0 1 0 Black Black  
% 2 3668 1666 3934 1933
% 
\special{pn 0}%
\special{sh 0.200}%
\special{pa 3668 1666}%
\special{pa 3934 1666}%
\special{pa 3934 1933}%
\special{pa 3668 1933}%
\special{pa 3668 1666}%
\special{ip}%
\special{pn 8}%
\special{pa 3668 1666}%
\special{pa 3934 1666}%
\special{pa 3934 1933}%
\special{pa 3668 1933}%
\special{pa 3668 1666}%
\special{pa 3934 1666}%
\special{fp}%
% BOX 2 0 1 0 Black Black  
% 2 3934 1666 4201 1933
% 
\special{pn 0}%
\special{sh 0.200}%
\special{pa 3934 1666}%
\special{pa 4201 1666}%
\special{pa 4201 1933}%
\special{pa 3934 1933}%
\special{pa 3934 1666}%
\special{ip}%
\special{pn 8}%
\special{pa 3934 1666}%
\special{pa 4201 1666}%
\special{pa 4201 1933}%
\special{pa 3934 1933}%
\special{pa 3934 1666}%
\special{pa 4201 1666}%
\special{fp}%
% BOX 2 0 1 0 Black Black  
% 2 4201 1666 4467 1933
% 
\special{pn 0}%
\special{sh 0.200}%
\special{pa 4201 1666}%
\special{pa 4467 1666}%
\special{pa 4467 1933}%
\special{pa 4201 1933}%
\special{pa 4201 1666}%
\special{ip}%
\special{pn 8}%
\special{pa 4201 1666}%
\special{pa 4467 1666}%
\special{pa 4467 1933}%
\special{pa 4201 1933}%
\special{pa 4201 1666}%
\special{pa 4467 1666}%
\special{fp}%
% BOX 2 0 1 0 Black Black  
% 2 4467 1666 4733 1933
% 
\special{pn 0}%
\special{sh 0.200}%
\special{pa 4467 1666}%
\special{pa 4733 1666}%
\special{pa 4733 1933}%
\special{pa 4467 1933}%
\special{pa 4467 1666}%
\special{ip}%
\special{pn 8}%
\special{pa 4467 1666}%
\special{pa 4733 1666}%
\special{pa 4733 1933}%
\special{pa 4467 1933}%
\special{pa 4467 1666}%
\special{pa 4733 1666}%
\special{fp}%
% BOX 2 0 1 0 Black Black  
% 2 4733 1666 5000 1933
% 
\special{pn 0}%
\special{sh 0.200}%
\special{pa 4733 1666}%
\special{pa 5000 1666}%
\special{pa 5000 1933}%
\special{pa 4733 1933}%
\special{pa 4733 1666}%
\special{ip}%
\special{pn 8}%
\special{pa 4733 1666}%
\special{pa 5000 1666}%
\special{pa 5000 1933}%
\special{pa 4733 1933}%
\special{pa 4733 1666}%
\special{pa 5000 1666}%
\special{fp}%
% BOX 2 0 1 0 Black Black  
% 2 3402 1933 3668 2199
% 
\special{pn 0}%
\special{sh 0.200}%
\special{pa 3402 1933}%
\special{pa 3668 1933}%
\special{pa 3668 2199}%
\special{pa 3402 2199}%
\special{pa 3402 1933}%
\special{ip}%
\special{pn 8}%
\special{pa 3402 1933}%
\special{pa 3668 1933}%
\special{pa 3668 2199}%
\special{pa 3402 2199}%
\special{pa 3402 1933}%
\special{pa 3668 1933}%
\special{fp}%
% BOX 2 0 1 0 Black Black  
% 2 3668 1933 3934 2199
% 
\special{pn 0}%
\special{sh 0.200}%
\special{pa 3668 1933}%
\special{pa 3934 1933}%
\special{pa 3934 2199}%
\special{pa 3668 2199}%
\special{pa 3668 1933}%
\special{ip}%
\special{pn 8}%
\special{pa 3668 1933}%
\special{pa 3934 1933}%
\special{pa 3934 2199}%
\special{pa 3668 2199}%
\special{pa 3668 1933}%
\special{pa 3934 1933}%
\special{fp}%
% BOX 2 0 1 0 Black Black  
% 2 3934 1933 4201 2199
% 
\special{pn 0}%
\special{sh 0.200}%
\special{pa 3934 1933}%
\special{pa 4201 1933}%
\special{pa 4201 2199}%
\special{pa 3934 2199}%
\special{pa 3934 1933}%
\special{ip}%
\special{pn 8}%
\special{pa 3934 1933}%
\special{pa 4201 1933}%
\special{pa 4201 2199}%
\special{pa 3934 2199}%
\special{pa 3934 1933}%
\special{pa 4201 1933}%
\special{fp}%
% BOX 2 0 1 0 Black Black  
% 2 4201 1933 4467 2199
% 
\special{pn 0}%
\special{sh 0.200}%
\special{pa 4201 1933}%
\special{pa 4467 1933}%
\special{pa 4467 2199}%
\special{pa 4201 2199}%
\special{pa 4201 1933}%
\special{ip}%
\special{pn 8}%
\special{pa 4201 1933}%
\special{pa 4467 1933}%
\special{pa 4467 2199}%
\special{pa 4201 2199}%
\special{pa 4201 1933}%
\special{pa 4467 1933}%
\special{fp}%
% BOX 2 0 1 0 Black Black  
% 2 4467 1933 4733 2199
% 
\special{pn 0}%
\special{sh 0.200}%
\special{pa 4467 1933}%
\special{pa 4733 1933}%
\special{pa 4733 2199}%
\special{pa 4467 2199}%
\special{pa 4467 1933}%
\special{ip}%
\special{pn 8}%
\special{pa 4467 1933}%
\special{pa 4733 1933}%
\special{pa 4733 2199}%
\special{pa 4467 2199}%
\special{pa 4467 1933}%
\special{pa 4733 1933}%
\special{fp}%
% BOX 2 0 1 0 Black Black  
% 2 4733 1933 5000 2199
% 
\special{pn 0}%
\special{sh 0.200}%
\special{pa 4733 1933}%
\special{pa 5000 1933}%
\special{pa 5000 2199}%
\special{pa 4733 2199}%
\special{pa 4733 1933}%
\special{ip}%
\special{pn 8}%
\special{pa 4733 1933}%
\special{pa 5000 1933}%
\special{pa 5000 2199}%
\special{pa 4733 2199}%
\special{pa 4733 1933}%
\special{pa 5000 1933}%
\special{fp}%
% BOX 2 0 3 0 Black White  
% 2 3402 2199 3668 2465
% 
\special{pn 8}%
\special{pa 3402 2199}%
\special{pa 3668 2199}%
\special{pa 3668 2465}%
\special{pa 3402 2465}%
\special{pa 3402 2199}%
\special{pa 3668 2199}%
\special{fp}%
% BOX 2 0 3 0 Black White  
% 2 3668 2199 3934 2465
% 
\special{pn 8}%
\special{pa 3668 2199}%
\special{pa 3934 2199}%
\special{pa 3934 2465}%
\special{pa 3668 2465}%
\special{pa 3668 2199}%
\special{pa 3934 2199}%
\special{fp}%
% BOX 2 0 3 0 Black Black  
% 2 3934 2199 4201 2465
% 
\special{pn 8}%
\special{pa 3934 2199}%
\special{pa 4201 2199}%
\special{pa 4201 2465}%
\special{pa 3934 2465}%
\special{pa 3934 2199}%
\special{pa 4201 2199}%
\special{fp}%
% BOX 2 0 1 0 Black Black  
% 2 4201 2199 4467 2465
% 
\special{pn 0}%
\special{sh 0.200}%
\special{pa 4201 2199}%
\special{pa 4467 2199}%
\special{pa 4467 2465}%
\special{pa 4201 2465}%
\special{pa 4201 2199}%
\special{ip}%
\special{pn 8}%
\special{pa 4201 2199}%
\special{pa 4467 2199}%
\special{pa 4467 2465}%
\special{pa 4201 2465}%
\special{pa 4201 2199}%
\special{pa 4467 2199}%
\special{fp}%
% BOX 2 0 1 0 Black Black  
% 2 4467 2199 4733 2465
% 
\special{pn 0}%
\special{sh 0.200}%
\special{pa 4467 2199}%
\special{pa 4733 2199}%
\special{pa 4733 2465}%
\special{pa 4467 2465}%
\special{pa 4467 2199}%
\special{ip}%
\special{pn 8}%
\special{pa 4467 2199}%
\special{pa 4733 2199}%
\special{pa 4733 2465}%
\special{pa 4467 2465}%
\special{pa 4467 2199}%
\special{pa 4733 2199}%
\special{fp}%
% BOX 2 0 3 0 Black Black  
% 2 4733 2199 5000 2465
% 
\special{pn 8}%
\special{pa 4733 2199}%
\special{pa 5000 2199}%
\special{pa 5000 2465}%
\special{pa 4733 2465}%
\special{pa 4733 2199}%
\special{pa 5000 2199}%
\special{fp}%
% BOX 2 0 3 0 Black White  
% 2 3402 2465 3668 2732
% 
\special{pn 8}%
\special{pa 3402 2465}%
\special{pa 3668 2465}%
\special{pa 3668 2732}%
\special{pa 3402 2732}%
\special{pa 3402 2465}%
\special{pa 3668 2465}%
\special{fp}%
% BOX 2 0 3 0 Black White  
% 2 3668 2465 3934 2732
% 
\special{pn 8}%
\special{pa 3668 2465}%
\special{pa 3934 2465}%
\special{pa 3934 2732}%
\special{pa 3668 2732}%
\special{pa 3668 2465}%
\special{pa 3934 2465}%
\special{fp}%
% BOX 2 0 3 0 Black White  
% 2 3934 2465 4201 2732
% 
\special{pn 8}%
\special{pa 3934 2465}%
\special{pa 4201 2465}%
\special{pa 4201 2732}%
\special{pa 3934 2732}%
\special{pa 3934 2465}%
\special{pa 4201 2465}%
\special{fp}%
% BOX 2 0 1 0 Black Black  
% 2 4201 2465 4467 2732
% 
\special{pn 0}%
\special{sh 0.200}%
\special{pa 4201 2465}%
\special{pa 4467 2465}%
\special{pa 4467 2732}%
\special{pa 4201 2732}%
\special{pa 4201 2465}%
\special{ip}%
\special{pn 8}%
\special{pa 4201 2465}%
\special{pa 4467 2465}%
\special{pa 4467 2732}%
\special{pa 4201 2732}%
\special{pa 4201 2465}%
\special{pa 4467 2465}%
\special{fp}%
% BOX 2 0 1 0 Black Black  
% 2 4467 2465 4733 2732
% 
\special{pn 0}%
\special{sh 0.200}%
\special{pa 4467 2465}%
\special{pa 4733 2465}%
\special{pa 4733 2732}%
\special{pa 4467 2732}%
\special{pa 4467 2465}%
\special{ip}%
\special{pn 8}%
\special{pa 4467 2465}%
\special{pa 4733 2465}%
\special{pa 4733 2732}%
\special{pa 4467 2732}%
\special{pa 4467 2465}%
\special{pa 4733 2465}%
\special{fp}%
% BOX 2 0 1 0 Black Black  
% 2 4733 2465 5000 2732
% 
\special{pn 0}%
\special{sh 0.200}%
\special{pa 4733 2465}%
\special{pa 5000 2465}%
\special{pa 5000 2732}%
\special{pa 4733 2732}%
\special{pa 4733 2465}%
\special{ip}%
\special{pn 8}%
\special{pa 4733 2465}%
\special{pa 5000 2465}%
\special{pa 5000 2732}%
\special{pa 4733 2732}%
\special{pa 4733 2465}%
\special{pa 5000 2465}%
\special{fp}%
% BOX 2 0 1 0 Black Black  
% 2 3402 2732 3668 2998
% 
\special{pn 0}%
\special{sh 0.200}%
\special{pa 3402 2732}%
\special{pa 3668 2732}%
\special{pa 3668 2998}%
\special{pa 3402 2998}%
\special{pa 3402 2732}%
\special{ip}%
\special{pn 8}%
\special{pa 3402 2732}%
\special{pa 3668 2732}%
\special{pa 3668 2998}%
\special{pa 3402 2998}%
\special{pa 3402 2732}%
\special{pa 3668 2732}%
\special{fp}%
% BOX 2 0 1 0 Black Black  
% 2 3668 2732 3934 2998
% 
\special{pn 0}%
\special{sh 0.200}%
\special{pa 3668 2732}%
\special{pa 3934 2732}%
\special{pa 3934 2998}%
\special{pa 3668 2998}%
\special{pa 3668 2732}%
\special{ip}%
\special{pn 8}%
\special{pa 3668 2732}%
\special{pa 3934 2732}%
\special{pa 3934 2998}%
\special{pa 3668 2998}%
\special{pa 3668 2732}%
\special{pa 3934 2732}%
\special{fp}%
% BOX 2 0 3 0 Black White  
% 2 3934 2732 4201 2998
% 
\special{pn 8}%
\special{pa 3934 2732}%
\special{pa 4201 2732}%
\special{pa 4201 2998}%
\special{pa 3934 2998}%
\special{pa 3934 2732}%
\special{pa 4201 2732}%
\special{fp}%
% BOX 2 0 1 0 Black Black  
% 2 4201 2732 4467 2998
% 
\special{pn 0}%
\special{sh 0.200}%
\special{pa 4201 2732}%
\special{pa 4467 2732}%
\special{pa 4467 2998}%
\special{pa 4201 2998}%
\special{pa 4201 2732}%
\special{ip}%
\special{pn 8}%
\special{pa 4201 2732}%
\special{pa 4467 2732}%
\special{pa 4467 2998}%
\special{pa 4201 2998}%
\special{pa 4201 2732}%
\special{pa 4467 2732}%
\special{fp}%
% BOX 2 0 1 0 Black Black  
% 2 4467 2732 4733 2998
% 
\special{pn 0}%
\special{sh 0.200}%
\special{pa 4467 2732}%
\special{pa 4733 2732}%
\special{pa 4733 2998}%
\special{pa 4467 2998}%
\special{pa 4467 2732}%
\special{ip}%
\special{pn 8}%
\special{pa 4467 2732}%
\special{pa 4733 2732}%
\special{pa 4733 2998}%
\special{pa 4467 2998}%
\special{pa 4467 2732}%
\special{pa 4733 2732}%
\special{fp}%
% BOX 2 0 1 0 Black Black  
% 2 4733 2732 5000 2998
% 
\special{pn 0}%
\special{sh 0.200}%
\special{pa 4733 2732}%
\special{pa 5000 2732}%
\special{pa 5000 2998}%
\special{pa 4733 2998}%
\special{pa 4733 2732}%
\special{ip}%
\special{pn 8}%
\special{pa 4733 2732}%
\special{pa 5000 2732}%
\special{pa 5000 2998}%
\special{pa 4733 2998}%
\special{pa 4733 2732}%
\special{pa 5000 2732}%
\special{fp}%
% STR 2 0 3 0 Black White  
% 4 3420 2266 3420 2400 2 0 0 0
% $g_{4,1}$
\put(34.2000,-24.0000){\makebox(0,0)[lb]{$g_{4,1}$}}%
% STR 2 0 3 0 Black White  
% 4 3420 2526 3420 2660 2 0 0 0
% $g_{5,1}$
\put(34.2000,-26.6000){\makebox(0,0)[lb]{$g_{5,1}$}}%
% STR 2 0 3 0 Black White  
% 4 3690 2526 3690 2660 2 0 0 0
% $g_{5,2}$
\put(36.9000,-26.6000){\makebox(0,0)[lb]{$g_{5,2}$}}%
% STR 2 0 3 0 Black White  
% 4 3690 2266 3690 2400 2 0 0 0
% $g_{4,2}$
\put(36.9000,-24.0000){\makebox(0,0)[lb]{$g_{4,2}$}}%
% STR 2 0 3 0 Black White  
% 4 3960 2266 3960 2400 2 0 0 0
% $g_{4,3}$
\put(39.6000,-24.0000){\makebox(0,0)[lb]{$g_{4,3}$}}%
% STR 2 0 3 0 Black White  
% 4 3960 2526 3960 2660 2 0 0 0
% $g_{5,3}$
\put(39.6000,-26.6000){\makebox(0,0)[lb]{$g_{5,3}$}}%
% STR 2 0 3 0 Black White  
% 4 3960 2786 3960 2920 2 0 0 0
% $g_{6,3}$
\put(39.6000,-29.2000){\makebox(0,0)[lb]{$g_{6,3}$}}%
% STR 2 0 3 0 Black White  
% 4 4760 2266 4760 2400 2 0 0 0
% $g_{4,6}$
\put(47.6000,-24.0000){\makebox(0,0)[lb]{$g_{4,6}$}}%
\end{picture}}%
\]
\end{example}

\vspace{20pt}
In the next section, we will study the Jacobian matrix of the generators $g_{i,j}$ of $J_{w,h}$. 
In particular, we will see that it has a structure of a block upper triangular matrix with respect to a particular ordering of the generators $g_{i,j}$ and the variables $z_{i,j}$ defined as follows.
We first order the elements $(i,j)$ in $[n]\times[n]$ by
\begin{align*}
\begin{cases}
(i,j)<(i',j') \quad &\text{if } i-j < i'-j', \\
(i,j)<(i',j') &\text{if } i-j = i'-j' \text{ and } i < i'.
\end{cases}
\end{align*}
This leads us to order the generators $g_{i,j}$ of $J_{w,h}$ according to their indexes; we say
$g_{i,j}<g_{i',j'}$ when $(i,j)<(i',j')$.
We also order the variables $z_{i,j}$ in the same manner.
Then the linear terms of $g_{i,j}$ stated in Proposition~\ref{prop: linear term} are ordered so that the genuine variable following $z_{i,j-1}$ is precisely $z_{i+1,j}$ (see \eqref{eq: linear term}). 

\begin{example}\label{ex: order}
The generators $g_{i,j}$ appearing in Example~\ref{eg: locations of gij} are ordered as
\begin{align*}
g_{1,6} \ \ < \ \  g_{1,5} < g_{2,6}  \ \ < \ \ g_{1,4} < g_{2,5} \ \ < \ \ g_{1,3} < g_{2,4} < g_{3,5}.
\end{align*}
If we view these $g_{i,j}$ placed in the grid $[n]\times[n]$, we may express this order more pictorially; it is simply given by reading the positions of these $g_{i,j}$ in $[n]\times[n]$ from the upper-left to the lower-right along upper or lower diagonal lines.
\end{example}

\bigskip

\section{Computing Jacobian matrices} \label{sec:Jacobi}

\subsection{Jacobian criterion for $\Hess(N,h)$}\label{subsec: Jacobi}
In this section, we determine which permutation flag $w_{\bullet}$ is a singular point of $\Hess(N,h)$ by using the open affine cover \eqref{eq: open cover of Hess} given by $\Hess(N,h)\cap wU^-$ such that $w_{\bullet}\in \Hess(N,h)$.
Since each $\Hess(N,h)\cap wU^-$ is an affine subvariety of $wU^-\cong \C^{n(n-1)/2}$, we rewrite the Jacobian criterion (Theorem~\ref{theorem:Jacobian_Criterion'}) for $\Hess(N,h)\cap wU^-$ in the following.
To this end, recall from Section~\ref{subsec: hc} that the defining ideal $J_{w,h}$ of $\Hess(N,h)\cap wU^-$ is generated by $g_{i,j}$ for $(i,j)\in \hc{w}$, and observe that
\begin{align*}
\dim_{\C} wU^- - \dim_{\C}(\Hess(N,h) \cap wU^-)
= \binom{n}{2} - \sum_{j=1}^n (h(j)-j) 
= |\hc{w}|
\end{align*}
by \eqref{eq: dim of Hess} and \eqref{eq:r_h} (c.f.\ \eqref{eq: codim of Hess}). We now obtain the following from the Jacobian criterion.

\begin{proposition}[Jacobian criterion for $\Hess(N,h)\cap wU^-$]\label{theorem:Jacobian_Criterion_Hessenberg}
A point $V_{\bullet}\in \Hess(N,h)\cap wU^-$ is a singular point of $\Hess(N,h)$ if and only if 
\begin{align}\label{eq: Jacobian_Criterion_Hessenberg}
\rank \left(\frac{\partial g_{i,j} }{\partial z_{p,q}}(V_{\bullet})\right)< |\hc{w}| ,
\end{align}
where the row indexes $(i,j)$ and the column indexes $(p,q)$ of the Jacobian matrix run over all $(i,j)\in \hc{w}$ and all $(p,q)$ for which $z_{p,q}$ is a genuine variables, respectively.
\end{proposition}

\begin{remark}\label{rem: rank and the number of rows}
The number $|\hc{w}|$ in \eqref{eq: Jacobian_Criterion_Hessenberg} is exactly the number of rows of the Jacobian matrix in the left hand side.
\end{remark}

\vspace{10pt}

Recall from Lemma~\ref{lem: genuine zij} that the genuine variables $z_{p,q}$ in $\C[{\bf x}_w]$ are the ones which satisfy $w^{-1}(p)>w^{-1}(q)$.
For a set of (ordered) functions $F_1,F_2,\ldots,F_\ell\in \C[{\bf x}_w]$, we denote by 
 \begin{align*}
\text{Jac}\Big(\{F_1,F_2,\ldots,F_\ell\},\{z_{p,q}\mid w^{-1}(p)>w^{-1}(q)\}\Big)
\end{align*} 
the Jacobian matrix consisting of $\frac{\partial F_k}{\partial z_{p,q}}$ for $k=1,2,\ldots,\ell$ and $(p,q)$ with $w^{-1}(p)>w^{-1}(q)$,
where we order the genuine variables $z_{p,q}$ as in Section~\ref{subsec: hc}.
Recall from Lemma~\ref{lem: genuine zij} again that the generators $g_{i,j}$ appearing in  $J_{w,h}$ are the ones which satisfy $w^{-1}(p) >h(w^{-1}(q) )$.
This leads us to consider the following Jacobian matrix.
\begin{align}\label{eq: def of J}
\Jac\coloneqq 
\text{Jac}\Big(\{g_{i,j}\mid w^{-1}(i)>h(w^{-1}(j))\},\{z_{p,q}\mid w^{-1}(p)>w^{-1}(q)\}\Big),
\end{align} 
where the generators $g_{i,j}$ are ordered as in Section~\ref{subsec: hc} as well. This is the Jacobian matrix appearing in Proposition~\ref{theorem:Jacobian_Criterion_Hessenberg}.

In the next subsection, we will compute the rank of $\Jac$ to apply Proposition~\ref{theorem:Jacobian_Criterion_Hessenberg}.
As preparations for that, 
we take a partition of the list of the generators:
\begin{align*}
 \{g_{i,j}\mid w^{-1}(i)>h(w^{-1}(j))\}
 = \bigsqcup_{-(n-1)\le d\le n-1)} G_d ,
\end{align*} 
where we set
\begin{align*}
G_d \coloneqq \{g_{i,j}\mid w^{-1}(i)>h(w^{-1}(j)), \ i-j=d\}
\quad \text{for \ $-(n-1)\le d\le n-1$}.
\end{align*} 
For $-(n-1)\le d\le n-1$, let $\Jac_d$ be the Jacobian matrix of $G_d$, i.e., 
\begin{align*}
\Jac_d\coloneqq 
\text{Jac}\Big(\{g_{i,j}\in G_d\},\{z_{i,j}\mid w^{-1}(i)>w^{-1}(j)\}\Big).
\end{align*} 
It is clear that each $\Jac_d$ has exactly $|G_d|$ rows, and the Jacobian matrix $\Jac$ is decomposed as
\begin{align}\label{eq: decomp of J 3}
 \Jac = 
 \begin{pmatrix}
  \text{------}\ \Jac_{-(n-1)}\ \text{------} \\ \text{------}\ \Jac_{-(n-2)}\ \text{------} \\ \vdots \\ \text{--------}\ \Jac_{n-2}\ \text{--------} \\ \text{--------}\ \Jac_{n-1}\ \text{--------}
 \end{pmatrix},
\end{align} 
where we mean that $\Jac_d$ with $G_d=\emptyset$ are empty matrices.
For example, we always have $G_0=\emptyset$ since $w^{-1}(i)\not>h(w^{-1}(i))$ for all $1\le i\le n$, and hence 
$\Jac_{0}$ does not appear in this decomposition. We call $\Jac_d$ with $G_d\ne \emptyset$ a \textbf{row matrix} of $\Jac$.

Each row matrix $\Jac_d$ can be decomposed further as
\begin{align}\label{eq: decomp of Jd 2}
 \Jac_d = (\ \Jdl{d}\mid \Jdc{d}\mid \Jdr{d}\ )
 \quad \text{for \ $-(n-1)\le d\le n-1$}
\end{align} 
by defining
\begin{align*}
&\Jdl{d}\coloneqq \text{Jac}\Big(\{g_{i,j}\in G_d \},\{z_{p,q}\mid w^{-1}(p)>w^{-1}(q), p-q<d+1\}\Big), \\
&\Jdc{d}\coloneqq \text{Jac}\Big(\{g_{i,j}\in G_d \},\{z_{p,q}\mid w^{-1}(p)>w^{-1}(q), p-q=d+1\}\Big), \\
&\Jdr{d}\coloneqq \text{Jac}\Big(\{g_{i,j}\in G_d \},\{z_{p,q}\mid w^{-1}(p)>w^{-1}(q), p-q>d+1\}\Big),
\end{align*} 
where the matrices with no columns are regarded as empty matrices.
We call $\Jdc{d}$ the \textbf{central block} of $\Jac_d$.
We will use this structure of $\Jac$ to compute its rank in the next subsection.

\subsection{Jacobian matrix evaluated at permutation flags}

The objective of this subsection is to prove Theorem~\ref{thm:fixed_point_singular} below which characterizes the permutation flags that are singular points of $\Hess(N,h)$.
Recall from Section~\ref{subsec: preview} that a \textbf{lower diagonal full-string (of height $d-1$)} of $[n]\times[n]$ is the set
\[
\{(d,1), (d+1,2),\ldots, (n,n-d+1)\in[n]\times[n]\}
\]
for some $2\le d\leq n$. See for example, the black dots in the following picture gives the lower diagonal full-string of height $2$ for $n=5$.\vspace{5pt}
\[
%WinTpicVersion4.32a
{\unitlength 0.1in%
\begin{picture}(10.0000,10.0000)(12.0000,-18.0000)%
% BOX 2 0 3 0 Black Black  
% 2 1200 800 1400 1000
% 
\special{pn 8}%
\special{pa 1200 800}%
\special{pa 1400 800}%
\special{pa 1400 1000}%
\special{pa 1200 1000}%
\special{pa 1200 800}%
\special{pa 1400 800}%
\special{fp}%
% BOX 2 0 3 0 Black Black  
% 2 1400 800 1600 1000
% 
\special{pn 8}%
\special{pa 1400 800}%
\special{pa 1600 800}%
\special{pa 1600 1000}%
\special{pa 1400 1000}%
\special{pa 1400 800}%
\special{pa 1600 800}%
\special{fp}%
% BOX 2 0 3 0 Black Black  
% 2 1600 800 1800 1000
% 
\special{pn 8}%
\special{pa 1600 800}%
\special{pa 1800 800}%
\special{pa 1800 1000}%
\special{pa 1600 1000}%
\special{pa 1600 800}%
\special{pa 1800 800}%
\special{fp}%
% BOX 2 0 3 0 Black Black  
% 2 1800 800 2000 1000
% 
\special{pn 8}%
\special{pa 1800 800}%
\special{pa 2000 800}%
\special{pa 2000 1000}%
\special{pa 1800 1000}%
\special{pa 1800 800}%
\special{pa 2000 800}%
\special{fp}%
% BOX 2 0 3 0 Black Black  
% 2 2000 800 2200 1000
% 
\special{pn 8}%
\special{pa 2000 800}%
\special{pa 2200 800}%
\special{pa 2200 1000}%
\special{pa 2000 1000}%
\special{pa 2000 800}%
\special{pa 2200 800}%
\special{fp}%
% BOX 2 0 3 0 Black Black  
% 2 1200 1000 1400 1200
% 
\special{pn 8}%
\special{pa 1200 1000}%
\special{pa 1400 1000}%
\special{pa 1400 1200}%
\special{pa 1200 1200}%
\special{pa 1200 1000}%
\special{pa 1400 1000}%
\special{fp}%
% BOX 2 0 3 0 Black Black  
% 2 1400 1000 1600 1200
% 
\special{pn 8}%
\special{pa 1400 1000}%
\special{pa 1600 1000}%
\special{pa 1600 1200}%
\special{pa 1400 1200}%
\special{pa 1400 1000}%
\special{pa 1600 1000}%
\special{fp}%
% BOX 2 0 3 0 Black Black  
% 2 1600 1000 1800 1200
% 
\special{pn 8}%
\special{pa 1600 1000}%
\special{pa 1800 1000}%
\special{pa 1800 1200}%
\special{pa 1600 1200}%
\special{pa 1600 1000}%
\special{pa 1800 1000}%
\special{fp}%
% BOX 2 0 3 0 Black Black  
% 2 1800 1000 2000 1200
% 
\special{pn 8}%
\special{pa 1800 1000}%
\special{pa 2000 1000}%
\special{pa 2000 1200}%
\special{pa 1800 1200}%
\special{pa 1800 1000}%
\special{pa 2000 1000}%
\special{fp}%
% BOX 2 0 3 0 Black Black  
% 2 2000 1000 2200 1200
% 
\special{pn 8}%
\special{pa 2000 1000}%
\special{pa 2200 1000}%
\special{pa 2200 1200}%
\special{pa 2000 1200}%
\special{pa 2000 1000}%
\special{pa 2200 1000}%
\special{fp}%
% BOX 2 0 3 0 Black Black  
% 2 1200 1200 1400 1400
% 
\special{pn 8}%
\special{pa 1200 1200}%
\special{pa 1400 1200}%
\special{pa 1400 1400}%
\special{pa 1200 1400}%
\special{pa 1200 1200}%
\special{pa 1400 1200}%
\special{fp}%
% BOX 2 0 3 0 Black Black  
% 2 1400 1200 1600 1400
% 
\special{pn 8}%
\special{pa 1400 1200}%
\special{pa 1600 1200}%
\special{pa 1600 1400}%
\special{pa 1400 1400}%
\special{pa 1400 1200}%
\special{pa 1600 1200}%
\special{fp}%
% BOX 2 0 3 0 Black Black  
% 2 1600 1200 1800 1400
% 
\special{pn 8}%
\special{pa 1600 1200}%
\special{pa 1800 1200}%
\special{pa 1800 1400}%
\special{pa 1600 1400}%
\special{pa 1600 1200}%
\special{pa 1800 1200}%
\special{fp}%
% BOX 2 0 3 0 Black Black  
% 2 1800 1200 2000 1400
% 
\special{pn 8}%
\special{pa 1800 1200}%
\special{pa 2000 1200}%
\special{pa 2000 1400}%
\special{pa 1800 1400}%
\special{pa 1800 1200}%
\special{pa 2000 1200}%
\special{fp}%
% BOX 2 0 3 0 Black Black  
% 2 2000 1200 2200 1400
% 
\special{pn 8}%
\special{pa 2000 1200}%
\special{pa 2200 1200}%
\special{pa 2200 1400}%
\special{pa 2000 1400}%
\special{pa 2000 1200}%
\special{pa 2200 1200}%
\special{fp}%
% BOX 2 0 3 0 Black White  
% 2 1200 1400 1400 1600
% 
\special{pn 8}%
\special{pa 1200 1400}%
\special{pa 1400 1400}%
\special{pa 1400 1600}%
\special{pa 1200 1600}%
\special{pa 1200 1400}%
\special{pa 1400 1400}%
\special{fp}%
% BOX 2 0 3 0 Black White  
% 2 1400 1400 1600 1600
% 
\special{pn 8}%
\special{pa 1400 1400}%
\special{pa 1600 1400}%
\special{pa 1600 1600}%
\special{pa 1400 1600}%
\special{pa 1400 1400}%
\special{pa 1600 1400}%
\special{fp}%
% BOX 2 0 3 0 Black Black  
% 2 1600 1400 1800 1600
% 
\special{pn 8}%
\special{pa 1600 1400}%
\special{pa 1800 1400}%
\special{pa 1800 1600}%
\special{pa 1600 1600}%
\special{pa 1600 1400}%
\special{pa 1800 1400}%
\special{fp}%
% BOX 2 0 3 0 Black Black  
% 2 1800 1400 2000 1600
% 
\special{pn 8}%
\special{pa 1800 1400}%
\special{pa 2000 1400}%
\special{pa 2000 1600}%
\special{pa 1800 1600}%
\special{pa 1800 1400}%
\special{pa 2000 1400}%
\special{fp}%
% BOX 2 0 3 0 Black Black  
% 2 2000 1400 2200 1600
% 
\special{pn 8}%
\special{pa 2000 1400}%
\special{pa 2200 1400}%
\special{pa 2200 1600}%
\special{pa 2000 1600}%
\special{pa 2000 1400}%
\special{pa 2200 1400}%
\special{fp}%
% BOX 2 0 3 0 Black Black  
% 2 1200 1600 1400 1800
% 
\special{pn 8}%
\special{pa 1200 1600}%
\special{pa 1400 1600}%
\special{pa 1400 1800}%
\special{pa 1200 1800}%
\special{pa 1200 1600}%
\special{pa 1400 1600}%
\special{fp}%
% BOX 2 0 3 0 Black White  
% 2 1400 1600 1600 1800
% 
\special{pn 8}%
\special{pa 1400 1600}%
\special{pa 1600 1600}%
\special{pa 1600 1800}%
\special{pa 1400 1800}%
\special{pa 1400 1600}%
\special{pa 1600 1600}%
\special{fp}%
% BOX 2 0 3 0 Black White  
% 2 1600 1600 1800 1800
% 
\special{pn 8}%
\special{pa 1600 1600}%
\special{pa 1800 1600}%
\special{pa 1800 1800}%
\special{pa 1600 1800}%
\special{pa 1600 1600}%
\special{pa 1800 1600}%
\special{fp}%
% BOX 2 0 3 0 Black Black  
% 2 1800 1600 2000 1800
% 
\special{pn 8}%
\special{pa 1800 1600}%
\special{pa 2000 1600}%
\special{pa 2000 1800}%
\special{pa 1800 1800}%
\special{pa 1800 1600}%
\special{pa 2000 1600}%
\special{fp}%
% BOX 2 0 3 0 Black Black  
% 2 2000 1600 2200 1800
% 
\special{pn 8}%
\special{pa 2000 1600}%
\special{pa 2200 1600}%
\special{pa 2200 1800}%
\special{pa 2000 1800}%
\special{pa 2000 1600}%
\special{pa 2200 1600}%
\special{fp}%
% CIRCLE 2 0 0 0 Black Black  
% 4 1700 1700 1730 1700 1730 1700 1730 1700
% 
\special{sh 1.000}%
\special{ia 1700 1700 30 30 0.0000000 6.2831853}%
\special{pn 8}%
\special{ar 1700 1700 30 30 0.0000000 6.2831853}%
% CIRCLE 2 0 0 0 Black Black  
% 4 1300 1300 1330 1300 1330 1300 1330 1300
% 
\special{sh 1.000}%
\special{ia 1300 1300 30 30 0.0000000 6.2831853}%
\special{pn 8}%
\special{ar 1300 1300 30 30 0.0000000 6.2831853}%
% CIRCLE 2 0 0 0 Black Black  
% 4 1500 1500 1530 1500 1530 1500 1530 1500
% 
\special{sh 1.000}%
\special{ia 1500 1500 30 30 0.0000000 6.2831853}%
\special{pn 8}%
\special{ar 1500 1500 30 30 0.0000000 6.2831853}%
\end{picture}}%
\vspace{5pt}
\] 
Recall also that $\hc{w}\subseteq[n]\times[n]$ is the \HC\ defined in Section~\ref{subsec: hc}.
As we stated in the beginning of Section~\ref{subsec: local description}, 
we are assuming $h(i)\geq i+1$ for all $1\leq i<n$ in this paper.

\begin{theorem} \label{thm:fixed_point_singular}
A permutation flag $w_{\bullet}\in \Hess(N,h)$ is a singular point of $\Hess(N,h)$ if and only if the Hessenberg complement $\hc{w}$ contains a lower diagonal full-string.
\end{theorem}

\begin{proof}
To begin with, we note that we have $w_{\bullet}\in \Hess(N,h)\cap wU^-$ which means that we can apply the results obtained in Section~\ref{sec: Linear terms}.
By the Jacobian criterion (Proposition~\ref{theorem:Jacobian_Criterion_Hessenberg}) and the definition \eqref{eq: def of J} of the Jacobian matrix $\Jac$, it suffices to show that the matrix $\Jac$ evaluated at the permutation flag $w_{\bullet}$ fails to have its proper full rank $|\hc{w}|$ (the number of rows of $\Jac$) if and only if $\hc{w}$ contains a lower diagonal full-string.

For a function $F\in \C[{\bf x}_w]$, we denote by $F|_{w_{\bullet}}\in\C$ the evaluation of $F$ at $w_{\bullet}$.
This evaluation is simply given by $z_{i,j}=0$ for all genuine variables $z_{i,j}$.
Hence, to compute the evaluated Jacobian matrix $\Jac|_{w_{\bullet}}$, it suffices to focus on the linear terms of the generators $g_{i,j}$.
Recall that we have the decomposition of $\Jac$ described in \eqref{eq: decomp of J 3} and \eqref{eq: decomp of Jd 2}.
For a row matrix $\Jac_d$ of $\Jac$, 
Proposition~\ref{prop: linear term}
 implies that the (possibly empty) blocks $\Jdl{d}$ and $\Jdr{d}$ in  \eqref{eq: decomp of Jd 2} are the zero matrices  when evaluated at $w_{\bullet}$ :
\begin{align}\label{eq: decomp of J 2}
\Jac_d|_{w_{\bullet}}=(\ \mathbf{0} \mid \Jdc{d}|_{w_{\bullet}} \mid \mathbf{0} \ ).
\end{align} 
In other words, the components of a row matrix $\Jac_d|_{w_{\bullet}}$ are all zero except for the central block $\Jdc{d}|_{w_{\bullet}}$. Moreover, for two distinct row matrices $\Jac_d|_{w_{\bullet}}$ and $\Jac_{d'}|_{w_{\bullet}}$ with $d< d'$, the indexes of columns of $\Jdc{d}|_{w_{\bullet}}$ are less than the indexes of columns of $\Jdc{d'}|_{w_{\bullet}}$ since we have $z_{p,q}<z_{p',q'}$ for all $p-q=d+1$ and $p'-q'=d'+1$.
Namely, $\Jdc{d}|_{w_{\bullet}}$ is located in the upper left to $\Jdc{d'}|_{w_{\bullet}}$ in the entire matrix $\Jac|_{w_{\bullet}}$.
This and \eqref{eq: decomp of J 2} together imply that the matrix $\Jac|_{w_{\bullet}}$ fails to have the full rank $|\hc{w}|$ (the number of rows of $\Jac|_{w_{\bullet}}$) if and only there exists a row matrix $\Jac_d|_{w_{\bullet}}$ such that its central block $\Jdc{d}|_{w_{\bullet}}$ fails to have its proper full rank $|G_d|$ (the number of rows of $\Jdc{d}|_{w_{\bullet}}$).

In what follows, we focus on the central block $\Jdc{d}|_{w_{\bullet}}$ of each row matrix $\Jac_d$, and we show that $\Jdc{d}|_{w_{\bullet}}$ fails to have its proper full rank $|G_d|$ if and only if $d>0$ and $\hc{w}$ contains a lower diagonal full-string of height $d$. 
To see this, we think of each generator $g_{i,j}$ in $G_d$ to lie on the $(i,j)$-th position of $\hc{w}\subseteq[n]\times[n]$ (c.f.\ Example~\ref{eg: locations of gij}); if $d<0$ then they are lined on the $|d|$-th upper diagonal line in $[n]\times[n]$, and if $d>0$ then they are lined on the $d$-th lower diagonal line in $[n]\times[n]$.
With this picture in mind, we take a partition of $G_d$ into maximal strings of generators, that is, we write
\begin{align}\label{eq: decomp of generators}
G_d=G_d^{(1)}\sqcup\cdots \sqcup G_d^{(r)},
\end{align} 
where each $G_d^{(k)}$ $(1\le k\le r)$ consists of a maximal consecutive string of generators $g_{a,b},g_{a+1,b+1},\ldots,g_{a+\ell,b+\ell}$ in $G_d$ for some length $\ell+1$ ($\ell\ge0$) which lie as
\begin{align}\label{eq: string of generators}
\begin{matrix}
g_{a,b} &  &     &  \\
    & g_{a+1,b+1} &   &  \\
    &     & \ddots &  \\
    &     &     & g_{a+\ell,b+\ell} \\ 
\end{matrix}
\end{align} 
on the $|d|$-th upper or lower diagonal line in $\hc{w}\subseteq [n]\times[n]$.
To determine each $G_d^{(k)}$ in \eqref{eq: decomp of generators} uniquely, we require that the indexes of generators in $G_d^{(k)}$ are less than the indexes of generators in $G_d^{(k')}$ when $k<k'$.
Since the elements of $G_d$ stands for the rows of the central block $\Jdc{d}|_{w_{\bullet}}$, the partition \eqref{eq: decomp of generators} induces a decomposition of $\Jdc{d}|_{w_{\bullet}}$ as follows.
Let $\Jdc{d}^{(k)}|_{w_{\bullet}}\ (1\le k\le r)$ be the Jacobian matrix for a maximal string $G_d^{(k)}$ in \eqref{eq: decomp of generators} with respect to the genuine variables $z_{p,q}$ with $p-q=d+1$.
Then the central block $\Jdc{d}|_{w_{\bullet}}$ is decomposed further as
\begin{align}\label{eq: dec of Jd}
 \Jdc{d}|_{w_{\bullet}} = 
 \begin{pmatrix}
  \text{------}\ \Jdc{d}^{(1)}|_{w_{\bullet}}\ \text{------} \\ \text{------}\ \Jdc{d}^{(2)}|_{w_{\bullet}}\ \text{------} \\ \vdots \\ \text{--------}\ \Jdc{d}^{(r)}|_{w_{\bullet}}\ \text{--------} 
 \end{pmatrix}.
\end{align} 
To compute each $\Jdc{d}^{(k)}|_{w_{\bullet}}\ (1\le k\le r)$ in \eqref{eq: dec of Jd}, 
take a component $G_d^{(k)}$ consisting of a maximal string \eqref{eq: string of generators}.
For $g_{i,j}\in G_d^{(k)}$, its linear terms are 
\begin{align}\label{eq: linear of gij}
 - z_{i,j-1} + z_{i+1,j}
\end{align}
by Proposition~\ref{prop: linear term}.
These variables are not necessarily genuine variables depending on $i\in[a,a+\ell]$ and $j\in[b,b+\ell]$; if $i=n$ then $z_{i+1,j}(=0)$ is not a genuine variable, and if $j=1$ then $z_{i,j-1}(=0)$ is not a genuine variable.

We first consider the case $d<0$.
In this case, the string \eqref{eq: string of generators} lines on the $|d|$-th upper diagonal line in $\hc{w}\subseteq [n]\times[n]$ which means that we have $b>1$ and $a+\ell<n$ in \eqref{eq: string of generators}.
Thus, for all $g_{i,j}\in G_d^{(k)}$, the linear terms in \eqref{eq: linear of gij} are genuine variables, and it follows that each row-block $\Jdc{d}^{(k)}|_{w_{\bullet}}$ in \eqref{eq: dec of Jd} takes of the form
\begin{align}\label{eq: further decomp}
\Jdc{d}^{(k)}|_{w_{\bullet}} = (\ \bm{0}\mid A_{d,k}\mid \bm{0}\ )
\qquad (1\le k\le r),
\end{align} 
where $A_{d,k}$ is the $|G_d^{(k)}|\times (|G_d^{(k)}|+1)$-matrix given by
\begin{align*}
A_{d,k}=
\begin{pmatrix}
  -1 & 1  &     &   & \\
  & -1 & 1  & & \\
   &     & \ddots & \ddots & \\
   &     &            & -1 & 1 
\end{pmatrix}
\qquad \text{(with all zeros on the empty spots).}
\end{align*}
We note that the zero matrices in \eqref{eq: further decomp} are possibly empty depending on 
the set of genuine variables $z_{p,q}$ with $p-q=d+1$ which stands for the columns of the row-block $\Jdc{d}^{(k)}|_{w_{\bullet}}$.
For two distinct row-blocks $\Jdc{d}^{(k)}|_{w_{\bullet}}$ and $\Jdc{d}^{(k')}|_{w_{\bullet}}$ with $k<k'$ in \eqref{eq: dec of Jd}, 
$A_{d,k}$ is located in the upper-left to $A_{d,k'}$ in $\Jdc{d}|_{w_{\bullet}}$.
This follows since 
the linear terms of generators in $G_d^{(k)}$ are less than the linear terms of generators in $G_d^{(k')}$ because of the definition of the partition \eqref{eq: decomp of generators}.
This together with $\rank A_{d,k}=|G_d^{(k)}|$ for all $k$ imply that $\Jdc{d}|_{w_{\bullet}}$ has its proper full rank $|G_d|$ for the case $d<0$.

Recalling that $G_0=\emptyset$ as mentioned in Section~\ref{subsec: Jacobi}, we next consider the case $d>0$.
In this case, the elements of $G_{d}$ are lined on the $d$-th lower diagonal line in $\hc{w}\subseteq [n]\times[n]$.
We first consider the case that $G_d$ is not the full-string of the maximum length (i.e.\ $G_d$ is not the string starting from the first column and ending at the bottom row), and we show that the central block $\Jdc{d}|_{w_{\bullet}}$ has its proper full rank $|G_d|$ in that case.
To this end, take a component $G_d^{(k)}$ of $G_d$ in \eqref{eq: decomp of generators} which consists of the string \eqref{eq: string of generators}.
Then $G_d^{(k)}$ cannot be the full-string of the maximum length as $G_d$ is not.
Hence, we have either of the following three conditions for the string \eqref{eq: string of generators};
\begin{itemize}
 \item[(i)] $b> 1$ and $a+\ell< n$,
 \item[(ii)] $b= 1$ and $a+\ell< n$,
 \item[(iii)] $b> 1$ and $a+\ell= n$
\end{itemize}
since if $b=1$ and $a+\ell= n$ at the same time then $G_d^{(k)}$ must be the full-string of the maximum length.
For the case (i), each $g_{i,j}\in G_d^{(k)}$ satisfies $j>1$ and $i<n$
so that both variables in the linear terms \eqref{eq: linear of gij} of $g_{i,j}$ are genuine variables (see Proposition~\ref{prop: linear term}).
Hence, 
the row-block $\Jdc{d}^{(k)}|_{w_{\bullet}}$ in \eqref{eq: dec of Jd} takes of the form 
\begin{align*}
\Jdc{d}^{(k)}|_{w_{\bullet}} = (\ \bm{0}\mid A_{d,k}\mid \bm{0}\ )
\end{align*} 
as above. 
For the case (ii), we have $G_d^{(k)}=\{g_{a,1},g_{a+1,2},\ldots, g_{a+\ell,1+\ell}\}$ with $a+\ell<n$. 
Except for the first element $g_{a,1}$, the linear terms \eqref{eq: linear of gij} of each $g_{i,j}\in G_d^{(k)}$ are genuine variables. The linear terms of $g_{a,1}$ is $-z_{a,0} + z_{a+1,1}$, where $z_{a,0}(=0)$ is not a genuine variable, and $z_{a+1,1}$ is a genuine variable since its index satisfies $a+1\le a+1+\ell\le n$. 
In particular, the former variable does not represent a column of the central block $\Jdc{d}|_{w_{\bullet}}$.
Hence, we have
\begin{align*}
\Jdc{d}^{(k)}|_{w_{\bullet}} = (\ \bm{0}\mid B_{d,k}\mid \bm{0}\ ),
\end{align*} 
where  $B_{d,k}$ is the $|G_d^{(k)}|\times |G_d^{(k)}|$-matrix given by
\begin{align*}
B_{d,k}=
\begin{pmatrix}
  1  &     &    & & \\
   -1 & 1  & & & \\
       & -1 & 1 & \\
       &     & \ddots & \ddots & \\
       &     &            & -1 & 1 
\end{pmatrix}.
\end{align*}
For the case (iii), we have $G_d^{(k)}=\{g_{a,b},g_{a+1,b+1},\ldots, g_{n,b+\ell}\}$ with $b>1$. 
Except for the last element $g_{n,b+\ell}$, the linear terms \eqref{eq: linear of gij} of each $g_{i,j}\in G_d^{(k)}$ are genuine variables. The linear terms of $g_{n,b+\ell}$ is $- z_{n,b+\ell-1} + z_{n+1,b+\ell}$, where $z_{n,b+\ell-1}$ is a genuine variable since its index satisfies $b+\ell-1\ge b-1\ge1$, and $z_{n+1,b+\ell}(=0)$ is not.
Similarly, we have
\begin{align*}
\Jdc{d}^{(k)}|_{w_{\bullet}} = (\ \bm{0}\mid C_{d,k}\mid \bm{0}\ ),
\end{align*} 
where  $C_{d,k}$ is the $|G_d^{(k)}|\times |G_d^{(k)}|$-matrix given by
\begin{align*}
C_{d,k}=
\begin{pmatrix}
 -1 & 1  & & & \\
     & -1 & 1 & \\
     &     & \ddots & \ddots & \\
     &     &            & -1 & 1 \\
     &     &            &     & -1  \\
\end{pmatrix}.
\end{align*}
Combining cases (i), (ii), and (iii), we obtain the following conclusion; 
when $G_d$ is not the full-string of the maximal length, each row-block $\Jdc{d}^{(k)}|_{w_{\bullet}}$ in \eqref{eq: dec of Jd} has its proper full rank $|G_d^{(k)}|$. 
Moreover, by the same argument as that in the case $d<0$, one can see that the non-zero columns of $\Jdc{d}^{(k)}|_{w_{\bullet}}$ is located in the upper-left to the non-zero columns of $\Jdc{d}^{(k')}|_{w_{\bullet}}$ if $k<k'$.
Thus, $\Jdc{d}|_{w_{\bullet}}$ has its proper full rank $|G_d|$ in this case.

Finally, we consider the case for which $G_d$ is the full-string of the maximum length.
This is precisely the case that the Hessenberg complement $\hc{w}$ contains a lower diagonal full-string of height $d$.  If $|G_d|=1$, the unique element $g_{n,1} \in G_d$ has no linear terms by Proposition~\ref{prop: linear term}, and $ \Jdc{d}|_{w_{\bullet}}$ consists of a single row full of zeros. Thus $\Jdc{d}|_{w_{\bullet}}$ fails to have its proper full rank $|G_d|=1$. 
If $|G_d|>1$, the arguments for the cases (ii) and (iii) above imply that we have
\begin{align}\label{eq: further decomp 4}
\Jdc{d}|_{w_{\bullet}}= (\ \bm{0}\mid D_{d}\mid \bm{0}\ ),
\end{align} 
where  $D_{d}$ is the $(|G_d|-1)\times |G_d|$-matrix given by
\begin{align*}
D_{d}=
\begin{pmatrix}
  1  & & & \\
    -1 & 1 & \\
        & \ddots & \ddots & \\
        &            & -1 & 1  \\
        &            &     & -1 \\
\end{pmatrix}.
\end{align*}
Hence, $\Jdc{d}|_{w_{\bullet}}$ fails to have its proper full rank $|G_d|$ for the case $|G_d|>1$ as well.
Therefore, we conclude that $\Jdc{d}|_{w_{\bullet}}$ fails to have its proper full rank $|G_d|$ when  the Hessenberg complement $\hc{w}$ contains a lower diagonal full-string of height $d$, and this is the only such case. 
This completes the proof.
\end{proof}

\begin{example}
Let $n=6$, and $h=(3,4,4,5,6,6)$. 
Take $w = 564321$ in one-line notation.
We saw in Example~\ref{ex:linear term} that $w_{\bullet}\in \Hess(N,h)$.
From Example~\ref{eg: locations of gij}, it is clear that the Hessenberg complement $\hc{w}$ does not contain a lower-left full string.
Hence, Theorem~\ref{thm:fixed_point_singular} claims that the permutation flag $w_{\bullet}$ is a nonsingular point of $\Hess(N,h)$.
In fact, as we have seen in Example~\ref{eg: locations of gij}, the functions
$g_{1,6}$, $g_{1,5}$, $g_{2,6}$, $g_{1,4}$, $g_{2,5}$, $g_{1,3}$, $g_{2,4}$, $g_{3,5}$ (ordered as in Example~\ref{ex: order}) generate the ideal $J_{w,h}$. A direct computation shows that Jacobian matrix $\Jac|_{w_{\bullet}}$ is given by
\[
%WinTpicVersion4.32a
{\unitlength 0.1in%
\begin{picture}(51.3500,21.2700)(12.0000,-26.6900)%
% STR 2 0 3 0 Black White  
% 4 1200 1080 1200 1162 2 0 0 0
% $g_{1,5}$
\put(12.0000,-11.6200){\makebox(0,0)[lb]{$g_{1,5}$}}%
% STR 2 0 3 0 Black White  
% 4 1200 836 1200 916 2 0 0 0
% $g_{1,6}$
\put(12.0000,-9.1600){\makebox(0,0)[lb]{$g_{1,6}$}}%
% STR 2 0 3 0 Black White  
% 4 1200 2302 1200 2384 2 0 0 0
% $g_{2,4}$
\put(12.0000,-23.8400){\makebox(0,0)[lb]{$g_{2,4}$}}%
% STR 2 0 3 0 Black White  
% 4 1200 1568 1200 1650 2 0 0 0
% $g_{1,4}$
\put(12.0000,-16.5000){\makebox(0,0)[lb]{$g_{1,4}$}}%
% STR 2 0 3 0 Black White  
% 4 1200 1814 1200 1894 2 0 0 0
% $g_{2,5}$
\put(12.0000,-18.9400){\makebox(0,0)[lb]{$g_{2,5}$}}%
% STR 2 0 3 0 Black White  
% 4 1200 1324 1200 1406 2 0 0 0
% $g_{2,6}$
\put(12.0000,-14.0600){\makebox(0,0)[lb]{$g_{2,6}$}}%
% STR 2 0 3 0 Black White  
% 4 1200 2058 1200 2140 2 0 0 0
% $g_{1,3}$
\put(12.0000,-21.4000){\makebox(0,0)[lb]{$g_{1,3}$}}%
% STR 2 0 3 0 Black White  
% 4 1200 2546 1200 2628 2 0 0 0
% $g_{3,5}$
\put(12.0000,-26.2800){\makebox(0,0)[lb]{$g_{3,5}$}}%
% STR 2 0 3 0 Black White  
% 4 4786 590 4786 672 2 0 0 0
% $z_{1,2}$
\put(47.8600,-6.7200){\makebox(0,0)[lb]{$z_{1,2}$}}%
% STR 2 0 3 0 Black White  
% 4 3865 1560 3865 1642 2 0 0 0
% $1$
\put(38.6500,-16.4200){\makebox(0,0)[lb]{$1$}}%
% STR 2 0 3 0 Black White  
% 4 5112 590 5112 672 2 0 0 0
% $z_{2,3}$
\put(51.1200,-6.7200){\makebox(0,0)[lb]{$z_{2,3}$}}%
% STR 2 0 3 0 Black White  
% 4 1526 590 1526 672 2 0 0 0
% $z_{1,6}$
\put(15.2600,-6.7200){\makebox(0,0)[lb]{$z_{1,6}$}}%
% STR 2 0 3 0 Black White  
% 4 3808 590 3808 672 2 0 0 0
% $z_{2,4}$
\put(38.0800,-6.7200){\makebox(0,0)[lb]{$z_{2,4}$}}%
% STR 2 0 3 0 Black White  
% 4 3156 590 3156 672 2 0 0 0
% $z_{3,6}$
\put(31.5600,-6.7200){\makebox(0,0)[lb]{$z_{3,6}$}}%
% STR 2 0 3 0 Black White  
% 4 5438 590 5438 672 2 0 0 0
% $z_{3,4}$
\put(54.3800,-6.7200){\makebox(0,0)[lb]{$z_{3,4}$}}%
% STR 2 0 3 0 Black White  
% 4 1852 590 1852 672 2 0 0 0
% $z_{1,5}$
\put(18.5200,-6.7200){\makebox(0,0)[lb]{$z_{1,5}$}}%
% STR 2 0 3 0 Black White  
% 4 4134 590 4134 672 2 0 0 0
% $z_{3,5}$
\put(41.3400,-6.7200){\makebox(0,0)[lb]{$z_{3,5}$}}%
% STR 2 0 3 0 Black White  
% 4 5764 590 5764 672 2 0 0 0
% $z_{4,5}$
\put(57.6400,-6.7200){\makebox(0,0)[lb]{$z_{4,5}$}}%
% STR 2 0 3 0 Black White  
% 4 2178 590 2178 672 2 0 0 0
% $z_{2,6}$
\put(21.7800,-6.7200){\makebox(0,0)[lb]{$z_{2,6}$}}%
% STR 2 0 3 0 Black White  
% 4 6090 590 6090 672 2 0 0 0
% $z_{6,5}$
\put(60.9000,-6.7200){\makebox(0,0)[lb]{$z_{6,5}$}}%
% STR 2 0 3 0 Black White  
% 4 2504 590 2504 672 2 0 0 0
% $z_{1,4}$
\put(25.0400,-6.7200){\makebox(0,0)[lb]{$z_{1,4}$}}%
% STR 2 0 3 0 Black White  
% 4 4460 590 4460 672 2 0 0 0
% $z_{4,6}$
\put(44.6000,-6.7200){\makebox(0,0)[lb]{$z_{4,6}$}}%
% STR 2 0 3 0 Black White  
% 4 3482 590 3482 672 2 0 0 0
% $z_{1,3}$
\put(34.8200,-6.7200){\makebox(0,0)[lb]{$z_{1,3}$}}%
% STR 2 0 3 0 Black White  
% 4 2830 590 2830 672 2 0 0 0
% $z_{2,5}$
\put(28.3000,-6.7200){\makebox(0,0)[lb]{$z_{2,5}$}}%
% STR 2 0 3 0 Black White  
% 4 2830 1324 2830 1406 2 0 0 0
% $-1$
\put(28.3000,-14.0600){\makebox(0,0)[lb]{$-1$}}%
% STR 2 0 3 0 Black White  
% 4 1852 836 1852 916 2 0 0 0
% $-1$
\put(18.5200,-9.1600){\makebox(0,0)[lb]{$-1$}}%
% LINE 2 0 3 0 Black White  
% 2 1200 974 6335 974
% 
\special{pn 8}%
\special{pa 1200 974}%
\special{pa 6335 974}%
\special{fp}%
% LINE 2 0 3 0 Black White  
% 2 1200 1462 6335 1462
% 
\special{pn 8}%
\special{pa 1200 1462}%
\special{pa 6335 1462}%
\special{fp}%
% LINE 2 2 3 0 Black White  
% 2 1200 1218 6335 1218
% 
\special{pn 8}%
\special{pa 1200 1218}%
\special{pa 6335 1218}%
\special{dt 0.045}%
% LINE 2 0 3 0 Black White  
% 2 1200 1952 6335 1952
% 
\special{pn 8}%
\special{pa 1200 1952}%
\special{pa 6335 1952}%
\special{fp}%
% LINE 2 2 3 0 Black White  
% 2 1200 2441 6335 2441
% 
\special{pn 8}%
\special{pa 1200 2441}%
\special{pa 6335 2441}%
\special{dt 0.045}%
% LINE 2 2 3 0 Black White  
% 2 1200 1708 6335 1708
% 
\special{pn 8}%
\special{pa 1200 1708}%
\special{pa 6335 1708}%
\special{dt 0.045}%
% LINE 2 2 3 0 Black White  
% 2 1200 2196 6335 2196
% 
\special{pn 8}%
\special{pa 1200 2196}%
\special{pa 6335 2196}%
\special{dt 0.045}%
% LINE 2 0 3 0 Black White  
% 2 1795 550 1795 2669
% 
\special{pn 8}%
\special{pa 1795 550}%
\special{pa 1795 2669}%
\special{fp}%
% LINE 2 2 3 0 Black White  
% 2 2121 550 2121 2669
% 
\special{pn 8}%
\special{pa 2121 550}%
\special{pa 2121 2669}%
\special{dt 0.045}%
% LINE 2 0 3 0 Black White  
% 2 2447 550 2447 2669
% 
\special{pn 8}%
\special{pa 2447 550}%
\special{pa 2447 2669}%
\special{fp}%
% LINE 2 2 3 0 Black White  
% 2 2773 550 2773 2669
% 
\special{pn 8}%
\special{pa 2773 550}%
\special{pa 2773 2669}%
\special{dt 0.045}%
% LINE 2 2 3 0 Black White  
% 2 3099 550 3099 2669
% 
\special{pn 8}%
\special{pa 3099 550}%
\special{pa 3099 2669}%
\special{dt 0.045}%
% LINE 2 0 3 0 Black White  
% 2 1200 730 6335 730
% 
\special{pn 8}%
\special{pa 1200 730}%
\special{pa 6335 730}%
\special{fp}%
% LINE 2 0 3 0 Black White  
% 2 3425 550 3425 2669
% 
\special{pn 8}%
\special{pa 3425 550}%
\special{pa 3425 2669}%
\special{fp}%
% LINE 2 2 3 0 Black White  
% 2 3751 550 3751 2669
% 
\special{pn 8}%
\special{pa 3751 550}%
\special{pa 3751 2669}%
\special{dt 0.045}%
% LINE 2 2 3 0 Black White  
% 2 4077 550 4077 2669
% 
\special{pn 8}%
\special{pa 4077 550}%
\special{pa 4077 2669}%
\special{dt 0.045}%
% LINE 2 2 3 0 Black White  
% 2 4403 550 4403 2669
% 
\special{pn 8}%
\special{pa 4403 550}%
\special{pa 4403 2669}%
\special{dt 0.045}%
% LINE 2 0 3 0 Black White  
% 2 4729 550 4729 2669
% 
\special{pn 8}%
\special{pa 4729 550}%
\special{pa 4729 2669}%
\special{fp}%
% LINE 2 2 3 0 Black White  
% 2 5055 550 5055 2669
% 
\special{pn 8}%
\special{pa 5055 550}%
\special{pa 5055 2669}%
\special{dt 0.045}%
% LINE 2 2 3 0 Black White  
% 2 5381 550 5381 2669
% 
\special{pn 8}%
\special{pa 5381 550}%
\special{pa 5381 2669}%
\special{dt 0.045}%
% LINE 2 2 3 0 Black White  
% 2 5707 550 5707 2669
% 
\special{pn 8}%
\special{pa 5707 550}%
\special{pa 5707 2669}%
\special{dt 0.045}%
% LINE 2 0 3 0 Black White  
% 2 6033 550 6033 2669
% 
\special{pn 8}%
\special{pa 6033 550}%
\special{pa 6033 2669}%
\special{fp}%
% LINE 2 0 3 0 Black White  
% 2 1469 550 1469 2669
% 
\special{pn 8}%
\special{pa 1469 550}%
\special{pa 1469 2669}%
\special{fp}%
% STR 2 0 3 0 Black White  
% 4 2504 1080 2504 1162 2 0 0 0
% $-1$
\put(25.0400,-11.6200){\makebox(0,0)[lb]{$-1$}}%
% STR 2 0 3 0 Black White  
% 4 2887 1072 2887 1153 2 0 0 0
% $1$
\put(28.8700,-11.5300){\makebox(0,0)[lb]{$1$}}%
% STR 2 0 3 0 Black White  
% 4 2235 828 2235 908 2 0 0 0
% $1$
\put(22.3500,-9.0800){\makebox(0,0)[lb]{$1$}}%
% STR 2 0 3 0 Black White  
% 4 3213 1316 3213 1398 2 0 0 0
% $1$
\put(32.1300,-13.9800){\makebox(0,0)[lb]{$1$}}%
% STR 2 0 3 0 Black White  
% 4 3482 1568 3482 1650 2 0 0 0
% $-1$
\put(34.8200,-16.5000){\makebox(0,0)[lb]{$-1$}}%
% STR 2 0 3 0 Black White  
% 4 4191 1805 4191 1886 2 0 0 0
% $1$
\put(41.9100,-18.8600){\makebox(0,0)[lb]{$1$}}%
% STR 2 0 3 0 Black White  
% 4 3808 1814 3808 1894 2 0 0 0
% $-1$
\put(38.0800,-18.9400){\makebox(0,0)[lb]{$-1$}}%
% STR 2 0 3 0 Black White  
% 4 4191 1805 4191 1886 2 0 0 0
% $1$
\put(41.9100,-18.8600){\makebox(0,0)[lb]{$1$}}%
% STR 2 0 3 0 Black White  
% 4 3808 1814 3808 1894 2 0 0 0
% $-1$
\put(38.0800,-18.9400){\makebox(0,0)[lb]{$-1$}}%
% STR 2 0 3 0 Black White  
% 4 5169 2050 5169 2131 2 0 0 0
% $1$
\put(51.6900,-21.3100){\makebox(0,0)[lb]{$1$}}%
% STR 2 0 3 0 Black White  
% 4 4786 2058 4786 2140 2 0 0 0
% $-1$
\put(47.8600,-21.4000){\makebox(0,0)[lb]{$-1$}}%
% STR 2 0 3 0 Black White  
% 4 5169 2050 5169 2131 2 0 0 0
% $1$
\put(51.6900,-21.3100){\makebox(0,0)[lb]{$1$}}%
% STR 2 0 3 0 Black White  
% 4 4786 2058 4786 2140 2 0 0 0
% $-1$
\put(47.8600,-21.4000){\makebox(0,0)[lb]{$-1$}}%
% STR 2 0 3 0 Black White  
% 4 5495 2294 5495 2376 2 0 0 0
% $1$
\put(54.9500,-23.7600){\makebox(0,0)[lb]{$1$}}%
% STR 2 0 3 0 Black White  
% 4 5112 2302 5112 2384 2 0 0 0
% $-1$
\put(51.1200,-23.8400){\makebox(0,0)[lb]{$-1$}}%
% STR 2 0 3 0 Black White  
% 4 5495 2294 5495 2376 2 0 0 0
% $1$
\put(54.9500,-23.7600){\makebox(0,0)[lb]{$1$}}%
% STR 2 0 3 0 Black White  
% 4 5112 2302 5112 2384 2 0 0 0
% $-1$
\put(51.1200,-23.8400){\makebox(0,0)[lb]{$-1$}}%
% STR 2 0 3 0 Black White  
% 4 5821 2538 5821 2620 2 0 0 0
% $1$
\put(58.2100,-26.2000){\makebox(0,0)[lb]{$1$}}%
% STR 2 0 3 0 Black White  
% 4 5438 2546 5438 2628 2 0 0 0
% $-1$
\put(54.3800,-26.2800){\makebox(0,0)[lb]{$-1$}}%
% STR 2 0 3 0 Black White  
% 4 5821 2538 5821 2620 2 0 0 0
% $1$
\put(58.2100,-26.2000){\makebox(0,0)[lb]{$1$}}%
% STR 2 0 3 0 Black White  
% 4 5438 2546 5438 2628 2 0 0 0
% $-1$
\put(54.3800,-26.2800){\makebox(0,0)[lb]{$-1$}}%
\end{picture}}%
\]
where we omit the zeros on the empty spots.
The solid lines represents the decomposition of the rows and columns which have the same difference of indexes. Namely, the matrix $\Jac|_{w_{\bullet}}$ is decomposed into the row matrices $\Jac_{d}|_{w_{\bullet}}$ for $d=-5,-4,-3,-2$, and the central blocks $\Jdc{d}|_{w_{\bullet}}$ are precisely the non-zero matrices surrounded by solid lines. 
One sees that the rank of the Jacobian matrix $\Jac|_{w_{\bullet}}$ is $8$ as expected.
\end{example}
\begin{example}
Let $n=6$, and $h=(3,4,4,5,6,6)$. 
Take $w=312654$ in one-line notation.
From Example~\ref{eg: locations of gij 2}, the Hessenberg complement $\hc{w}$ contains the lower-left full string of height $3$.
Hence, Theorem~\ref{thm:fixed_point_singular} claims that the permutation flag $w_{\bullet}$ is a singular point of $\Hess(N,h)$.
In fact, as we have seen in Example~\ref{eg: locations of gij 2}, the functions
$g_{4,6}$, $g_{4,3}$, $g_{4,2}$, $g_{5,3}$, $g_{4,1}$, $g_{5,2}$, $g_{6,3}$, $g_{5,1}$ generate the ideal $J_{w,h}$, where the positions of $g_{4,1}$, $g_{5,2}$, $g_{6,3}\in G_{3}$ consist a lower diagonal full-string of height 3 in $\hc{w}$.
Again, a direct computation shows that the Jacobian matrix $\Jac|_{w_{\bullet}}$ is the following.
\[
%WinTpicVersion4.32a
{\unitlength 0.1in%
\begin{picture}(51.3500,21.2700)(12.0000,-26.6900)%
% STR 2 0 3 0 Black White  
% 4 1200 1080 1200 1162 2 0 0 0
% $g_{4,3}$
\put(12.0000,-11.6200){\makebox(0,0)[lb]{$g_{4,3}$}}%
% STR 2 0 3 0 Black White  
% 4 1200 836 1200 916 2 0 0 0
% $g_{4,6}$
\put(12.0000,-9.1600){\makebox(0,0)[lb]{$g_{4,6}$}}%
% STR 2 0 3 0 Black White  
% 4 1200 2302 1200 2384 2 0 0 0
% $g_{6,3}$
\put(12.0000,-23.8400){\makebox(0,0)[lb]{$g_{6,3}$}}%
% STR 2 0 3 0 Black White  
% 4 1200 1568 1200 1650 2 0 0 0
% $g_{5,3}$
\put(12.0000,-16.5000){\makebox(0,0)[lb]{$g_{5,3}$}}%
% STR 2 0 3 0 Black White  
% 4 1200 1814 1200 1894 2 0 0 0
% $g_{4,1}$
\put(12.0000,-18.9400){\makebox(0,0)[lb]{$g_{4,1}$}}%
% STR 2 0 3 0 Black White  
% 4 1200 1324 1200 1406 2 0 0 0
% $g_{4,2}$
\put(12.0000,-14.0600){\makebox(0,0)[lb]{$g_{4,2}$}}%
% STR 2 0 3 0 Black White  
% 4 1200 2058 1200 2140 2 0 0 0
% $g_{5,2}$
\put(12.0000,-21.4000){\makebox(0,0)[lb]{$g_{5,2}$}}%
% STR 2 0 3 0 Black White  
% 4 1200 2546 1200 2628 2 0 0 0
% $g_{5,1}$
\put(12.0000,-26.2800){\makebox(0,0)[lb]{$g_{5,1}$}}%
% STR 2 0 3 0 Black White  
% 4 4134 590 4134 672 2 0 0 0
% $z_{5,3}$
\put(41.3400,-6.7200){\makebox(0,0)[lb]{$z_{5,3}$}}%
% STR 2 0 3 0 Black White  
% 4 5169 1560 5169 1642 2 0 0 0
% $1$
\put(51.6900,-16.4200){\makebox(0,0)[lb]{$1$}}%
% STR 2 0 3 0 Black White  
% 4 4460 590 4460 672 2 0 0 0
% $z_{4,1}$
\put(44.6000,-6.7200){\makebox(0,0)[lb]{$z_{4,1}$}}%
% STR 2 0 3 0 Black White  
% 4 3156 590 3156 672 2 0 0 0
% $z_{2,1}$
\put(31.5600,-6.7200){\makebox(0,0)[lb]{$z_{2,1}$}}%
% STR 2 0 3 0 Black White  
% 4 2504 590 2504 672 2 0 0 0
% $z_{4,5}$
\put(25.0400,-6.7200){\makebox(0,0)[lb]{$z_{4,5}$}}%
% STR 2 0 3 0 Black White  
% 4 4786 590 4786 672 2 0 0 0
% $z_{5,2}$
\put(47.8600,-6.7200){\makebox(0,0)[lb]{$z_{5,2}$}}%
% STR 2 0 3 0 Black White  
% 4 3482 590 3482 672 2 0 0 0
% $z_{4,3}$
\put(34.8200,-6.7200){\makebox(0,0)[lb]{$z_{4,3}$}}%
% STR 2 0 3 0 Black White  
% 4 5112 590 5112 672 2 0 0 0
% $z_{6,3}$
\put(51.1200,-6.7200){\makebox(0,0)[lb]{$z_{6,3}$}}%
% STR 2 0 3 0 Black White  
% 4 1526 590 1526 672 2 0 0 0
% $z_{1,3}$
\put(15.2600,-6.7200){\makebox(0,0)[lb]{$z_{1,3}$}}%
% STR 2 0 3 0 Black White  
% 4 5764 590 5764 672 2 0 0 0
% $z_{6,2}$
\put(57.6400,-6.7200){\makebox(0,0)[lb]{$z_{6,2}$}}%
% STR 2 0 3 0 Black White  
% 4 1852 590 1852 672 2 0 0 0
% $z_{4,6}$
\put(18.5200,-6.7200){\makebox(0,0)[lb]{$z_{4,6}$}}%
% STR 2 0 3 0 Black White  
% 4 3808 590 3808 672 2 0 0 0
% $z_{4,2}$
\put(38.0800,-6.7200){\makebox(0,0)[lb]{$z_{4,2}$}}%
% STR 2 0 3 0 Black White  
% 4 2830 590 2830 672 2 0 0 0
% $z_{5,6}$
\put(28.3000,-6.7200){\makebox(0,0)[lb]{$z_{5,6}$}}%
% STR 2 0 3 0 Black White  
% 4 2178 590 2178 672 2 0 0 0
% $z_{2,3}$
\put(21.7800,-6.7200){\makebox(0,0)[lb]{$z_{2,3}$}}%
% STR 2 0 3 0 Black White  
% 4 4460 1324 4460 1406 2 0 0 0
% $-1$
\put(44.6000,-14.0600){\makebox(0,0)[lb]{$-1$}}%
% STR 2 0 3 0 Black White  
% 4 2504 836 2504 916 2 0 0 0
% $-1$
\put(25.0400,-9.1600){\makebox(0,0)[lb]{$-1$}}%
% LINE 2 0 3 0 Black White  
% 2 1200 974 6335 974
% 
\special{pn 8}%
\special{pa 1200 974}%
\special{pa 6335 974}%
\special{fp}%
% LINE 2 2 3 0 Black White  
% 2 1200 1462 6335 1462
% 
\special{pn 8}%
\special{pa 1200 1462}%
\special{pa 6335 1462}%
\special{dt 0.045}%
% LINE 2 0 3 0 Black White  
% 2 1200 1218 6335 1218
% 
\special{pn 8}%
\special{pa 1200 1218}%
\special{pa 6335 1218}%
\special{fp}%
% LINE 2 2 3 0 Black White  
% 2 1200 1952 6335 1952
% 
\special{pn 8}%
\special{pa 1200 1952}%
\special{pa 6335 1952}%
\special{dt 0.045}%
% LINE 2 0 3 0 Black White  
% 2 1200 2441 6335 2441
% 
\special{pn 8}%
\special{pa 1200 2441}%
\special{pa 6335 2441}%
\special{fp}%
% LINE 2 0 3 0 Black White  
% 2 1200 1708 6335 1708
% 
\special{pn 8}%
\special{pa 1200 1708}%
\special{pa 6335 1708}%
\special{fp}%
% LINE 2 2 3 0 Black White  
% 2 1200 2196 6335 2196
% 
\special{pn 8}%
\special{pa 1200 2196}%
\special{pa 6335 2196}%
\special{dt 0.045}%
% LINE 2 2 3 0 Black White  
% 2 1795 550 1795 2669
% 
\special{pn 8}%
\special{pa 1795 550}%
\special{pa 1795 2669}%
\special{dt 0.045}%
% LINE 2 0 3 0 Black White  
% 2 2121 550 2121 2669
% 
\special{pn 8}%
\special{pa 2121 550}%
\special{pa 2121 2669}%
\special{fp}%
% LINE 2 2 3 0 Black White  
% 2 2447 550 2447 2669
% 
\special{pn 8}%
\special{pa 2447 550}%
\special{pa 2447 2669}%
\special{dt 0.045}%
% LINE 2 2 3 0 Black White  
% 2 2773 550 2773 2669
% 
\special{pn 8}%
\special{pa 2773 550}%
\special{pa 2773 2669}%
\special{dt 0.045}%
% LINE 2 0 3 0 Black White  
% 2 3099 550 3099 2669
% 
\special{pn 8}%
\special{pa 3099 550}%
\special{pa 3099 2669}%
\special{fp}%
% LINE 2 0 3 0 Black White  
% 2 1200 730 6335 730
% 
\special{pn 8}%
\special{pa 1200 730}%
\special{pa 6335 730}%
\special{fp}%
% LINE 2 2 3 0 Black White  
% 2 3425 550 3425 2669
% 
\special{pn 8}%
\special{pa 3425 550}%
\special{pa 3425 2669}%
\special{dt 0.045}%
% LINE 2 0 3 0 Black White  
% 2 3751 550 3751 2669
% 
\special{pn 8}%
\special{pa 3751 550}%
\special{pa 3751 2669}%
\special{fp}%
% LINE 2 2 3 0 Black White  
% 2 4077 550 4077 2669
% 
\special{pn 8}%
\special{pa 4077 550}%
\special{pa 4077 2669}%
\special{dt 0.045}%
% LINE 2 0 3 0 Black White  
% 2 4403 550 4403 2669
% 
\special{pn 8}%
\special{pa 4403 550}%
\special{pa 4403 2669}%
\special{fp}%
% LINE 2 2 3 0 Black White  
% 2 4729 550 4729 2669
% 
\special{pn 8}%
\special{pa 4729 550}%
\special{pa 4729 2669}%
\special{dt 0.045}%
% LINE 2 2 3 0 Black White  
% 2 5055 550 5055 2669
% 
\special{pn 8}%
\special{pa 5055 550}%
\special{pa 5055 2669}%
\special{dt 0.045}%
% LINE 2 0 3 0 Black White  
% 2 5381 550 5381 2669
% 
\special{pn 8}%
\special{pa 5381 550}%
\special{pa 5381 2669}%
\special{fp}%
% LINE 2 2 3 0 Black White  
% 2 5707 550 5707 2669
% 
\special{pn 8}%
\special{pa 5707 550}%
\special{pa 5707 2669}%
\special{dt 0.045}%
% LINE 2 0 3 0 Black White  
% 2 6033 550 6033 2669
% 
\special{pn 8}%
\special{pa 6033 550}%
\special{pa 6033 2669}%
\special{fp}%
% LINE 2 0 3 0 Black White  
% 2 1469 550 1469 2669
% 
\special{pn 8}%
\special{pa 1469 550}%
\special{pa 1469 2669}%
\special{fp}%
% STR 2 0 3 0 Black White  
% 4 3808 1080 3808 1162 2 0 0 0
% $-1$
\put(38.0800,-11.6200){\makebox(0,0)[lb]{$-1$}}%
% STR 2 0 3 0 Black White  
% 4 4191 1072 4191 1153 2 0 0 0
% $1$
\put(41.9100,-11.5300){\makebox(0,0)[lb]{$1$}}%
% STR 2 0 3 0 Black White  
% 4 4843 1316 4843 1398 2 0 0 0
% $1$
\put(48.4300,-13.9800){\makebox(0,0)[lb]{$1$}}%
% STR 2 0 3 0 Black White  
% 4 4786 1568 4786 1650 2 0 0 0
% $-1$
\put(47.8600,-16.5000){\makebox(0,0)[lb]{$-1$}}%
% STR 2 0 3 0 Black White  
% 4 5495 1805 5495 1886 2 0 0 0
% $1$
\put(54.9500,-18.8600){\makebox(0,0)[lb]{$1$}}%
% STR 2 0 3 0 Black White  
% 4 5495 1805 5495 1886 2 0 0 0
% $1$
\put(54.9500,-18.8600){\makebox(0,0)[lb]{$1$}}%
% STR 2 0 3 0 Black White  
% 4 5821 2050 5821 2131 2 0 0 0
% $1$
\put(58.2100,-21.3100){\makebox(0,0)[lb]{$1$}}%
% STR 2 0 3 0 Black White  
% 4 5771 2294 5771 2376 2 0 0 0
% $-1$
\put(57.7100,-23.7600){\makebox(0,0)[lb]{$-1$}}%
% STR 2 0 3 0 Black White  
% 4 6147 2538 6147 2620 2 0 0 0
% $1$
\put(61.4700,-26.2000){\makebox(0,0)[lb]{$1$}}%
% STR 2 0 3 0 Black White  
% 4 6147 2538 6147 2620 2 0 0 0
% $1$
\put(61.4700,-26.2000){\makebox(0,0)[lb]{$1$}}%
% STR 2 0 3 0 Black White  
% 4 6090 590 6090 672 2 0 0 0
% $z_{6,1}$
\put(60.9000,-6.7200){\makebox(0,0)[lb]{$z_{6,1}$}}%
% STR 2 0 3 0 Black White  
% 4 5438 590 5438 672 2 0 0 0
% $z_{5,1}$
\put(54.3800,-6.7200){\makebox(0,0)[lb]{$z_{5,1}$}}%
% STR 2 0 3 0 Black White  
% 4 5438 2058 5438 2140 2 0 0 0
% $-1$
\put(54.3800,-21.4000){\makebox(0,0)[lb]{$-1$}}%
% STR 2 0 3 0 Black White  
% 4 5438 2058 5438 2140 2 0 0 0
% $-1$
\put(54.3800,-21.4000){\makebox(0,0)[lb]{$-1$}}%
% STR 2 0 3 0 Black White  
% 4 2887 828 2887 908 2 0 0 0
% $1$
\put(28.8700,-9.0800){\makebox(0,0)[lb]{$1$}}%
\end{picture}}%
\]
One can see that the rank of the Jacobian matrix $\Jac|_{w_{\bullet}}$ is $7(<8)$ because of the row matrix $\Jac_{3}|_{w_{\bullet}}$ concerning $g_{4,1}, g_{5,2}, g_{6,3}$.
\end{example}

\begin{example}\label{ex: singularity_list}
Let $h=(3,3,4,5,5)$. In Section~\ref{subsec: preview} we demonstrated that for $w=32154$ the set $\hc{w}$ contains a lower diagonal full-string.  
By drawing the diagrams $\hc{w}$ for each permutation flag in $\Hess(N,h)$ (or asking Sage to compute the sets $\hc{w}$), one can confirm that the permutation flags $w_\bullet$ in $\Hess(N,h)$  for which $\hc{w}$ contains a lower diagonal full-string are the ones whose one-line notation are listed below. 

\begin{center}
\begin{tabular}{|c|c|c|c|c|c|c|c|} \hline
$12345$ & 
$12354$ & 
$12435$ & 
$13245$ & 
$13254$ & 
$14325$ & 
$21345$ & 
$21354$ \\ \hline
$21435$ &
$23145$ &
$23154$ &
$31245$ &
$31254$ &
$32145$ & 
$32154$ &
$41325$   \\ \hline
\end{tabular}
\end{center}
 Theorem \ref{thm:fixed_point_singular} tells us that these are precisely the singular permutation flags in $\Hess(N,h)$. 
\end{example}

\vspace{10pt}

\subsection{Jacobian matrix evaluated on Hessenberg-Schubert cells}
In this subsection, we give sufficient conditions for the rank of the Jacobian matrix of $\Jac$ to be constant on each Hessenberg-Schubert cell. These criterions will be used in Section~\ref{sec: normality} to determine which $\Hess(N,h)$ is a normal algebraic variety.
We also demonstrate that there exists a Hessenberg-Schubert cell whose permutation flag is singular, but its generic point is not. 

\begin{proposition} \label{prop: sufficient condition on cell}
Let $w_{\bullet}\in\Hess(N,h)$ be a permutation flag.
Then the following hold.
\begin{itemize}
 \item[$(1)$] If $w_{\bullet}$ is a nonsingular point of $\Hess(N,h)$, then every point $V_\bullet\in\Hess(N,h) \cap \SC{w}$ is a nonsingular point of $\Hess(N,h)$.
 \item[$(2)$] If $w_\bullet$ is a singular point of $\Hess(N,h)$, and $H_w^c$ contains a lower diagonal full string of height $n-1$ or $n-2$, then every point $V_\bullet\in\Hess(N,h) \cap \SC{w}$ is a singular point of $\Hess(N,h)$.
\end{itemize}
\end{proposition}

\begin{proof}
Since we have $\Hess(N,h) \cap \SC{w}\subseteq \Hess(N,h) \cap wU^-$, we can apply the Jacobian criterion (Proposition~\ref{theorem:Jacobian_Criterion_Hessenberg}) to prove the claim of this proposition as well.
To this end, recall that $\Jac$ is the Jacobian matrix defined in \eqref{eq: def of J}.
We denote by $\Jac |_{\text{HS cell}}$ its restriction on the Hessenberg-Schubert cell $\Hess(N,h)\cap \SC{w}$ as in Definition~\ref{def: def of rest on HS}.
We decompose the Jacobian matrix $\Jac |_{\text{HS cell}}$ as in \eqref{eq: decomp of J 3} and \eqref{eq: decomp of Jd 2}.
For a row matrix $\Jac_d|_{\text{HS cell}}$ of $\Jac|_{\text{HS cell}}$, 
Proposition~\ref{prop: a partial result 1} implies that $$\Jdl{d} |_{\text{HS cell}} = {\mathbf{0}}.$$ 
Namely, we obtain
\begin{align}\label{eq: restricting on cell}
\Jac_d|_{\text{HS cell}}=(\mathbf{0} \mid \Jdc{d}|_{\text{HS cell}} \mid \Jdr{d}|_{\text{HS cell}} ).
\end{align} 
for each row matrix $\Jac_d|_{\text{HS cell}}$ of $\Jac|_{\text{HS cell}}$  in \eqref{eq: decomp of J 3}. By the same argument as that below \eqref{eq: decomp of J 2}, this means that the Jacobian matrix $\Jac|_{\text{HS cell}}$ takes a block upper triangular form.
Moreover, applying Proposition~\ref{prop: a partial result 2} to each generator $g_{i,j}$ in $G_d$, it follows that the central block in \eqref{eq: restricting on cell} satisfies
$$\Jdc{d} |_{\text{HS cell}} = \Jdc{d} |_{w_{\bullet}}, $$
by which we mean that the central block $\Jdc{d} |_{\text{HS cell}}$ in a row matrix \eqref{eq: restricting on cell} has the same form when restricted to any point in the Hessenberg-Schubert cell $\Hess(N,h)\cap \SC{w}$.  

We now prove the claim (1).
By the assumption, $w_{\bullet}\in\Hess(N,h)$ is a nonsingular point of $\Hess(N,h)$. 
Hence the proof of Theorem~\ref{thm:fixed_point_singular} shows that each $\Jdc{d} |_{w_{\bullet}}$ has its proper full rank.
Thus, the last equality and the block upper triangularity of $\Jac|_{\text{HS cell}}$ imply that
$$\rank \left ( \Jac|_{\text{HS cell}} \right ) = \rank \left ( \Jac|_{w_{\bullet}} \right)$$
because the pivots of  the column-echelon form of $\Jac|_{\text{HS cell}}$ and the pivots of the column-echelon form of $\Jac|_{w_{\bullet}}$ are the same. This means that the Jacobian matrix $\Jac$ has its proper full rank on every point of the Hessenberg-Schubert cell $\Hess(N,h) \cap \SC{w}$.  

We next prove the claim (2).
If $H_w^c$ contains a lower diagonal full string of height $d=n-1$ (i.e. $ (n,1) \in H_w^c$), then then $ \Jdr{d}|_{\text{HS cell}}$ in \eqref{eq: restricting on cell} is the empty matrix because there are no variables $z_{p,q}$ satisfying $p-q > n$.
The same clam holds if $H_w^c$ contains a lower diagonal full string of height $d=n-2$ (i.e. $(n-1,1), (n,2)\in H_w^c$) since there are no variables $z_{p,q}$ satisfying $p-q > n-1$. Namely, in both of these cases $d=n-1,n-2$, \eqref{eq: restricting on cell} can be written as $$ \Jac_d|_{\text{HS cell}}=(\mathbf{0} \mid \Jdc{d}|_{\text{HS cell}} ) = (\mathbf{0} \mid \Jdc{d}|_{w_\bullet} ) .$$
Moreover, since $w_{\bullet}$ is a singular point of $\Hess(N,h)$ by the assumption, the matrix $\Jdc{d}|_{w_\bullet}=(\ \bm{0}\mid D_{d}\mid \bm{0}\ )$ does not have its proper full rank by the proof of Theorem~\ref{thm:fixed_point_singular}. This means by Remark~\ref{rem: rank and the number of rows} that the Jacobian matrix $\Jac$ does not have its proper full rank on every point of the cell $\Hess(N,h) \cap \SC{w}$.  
\end{proof}

\vspace{10pt}

\begin{remark}
It is well-known that regular nilpotent Hessenberg varieties admit a $\C^{\times}$-action (e.g.\ \cite[Section~5]{HT17}, \cite[Section~2]{AHHM17}) which satisfies the following two properties: (i) it preserves each Hessenberg-Schubert cell $\Hess(N,h) \cap \SC{w}$, (ii) $w_{\bullet}$ is the unique permutation flag in the cell $\Hess(N,h) \cap \SC{w}$.
Proposition~\ref{prop: sufficient condition on cell}~(1) can also be explained from this $\C^{\times}$-action since non-singularity is an open condition.
\end{remark}

\begin{remark}
We will see that Proposition~\ref{prop: sufficient condition on cell}~(2) applies to every cell in the Peterson variety in the next section (see Proposition~\ref{proposition:Peterson Strings}).  
\end{remark}

\vspace{10pt}

We end this section by presenting an example of a Hessenberg-Schubert cell $\Hess(N,h)\cap \SC{w}$ such that the permutation flag $w_{\bullet}$ is a singular point of $\Hess(N,h)$ and the rest of the points in the cell are nonsingular points of $\Hess(N,h)$.
This means that the ``singular point" analogue of Proposition~\ref{prop: sufficient condition on cell} (1) does not hold in general.

\begin{example} \label{ex:cellnonsingular}
Consider the example $h=(3,4,5,6,6,6)$ and $w=321654$ in one-line notation.  
The Hessenberg complement $\hc{w}$ contains the lower diagonal full-string of height $3$ consisting of $\{ (4,1), (5,2), (6,3) \}$.
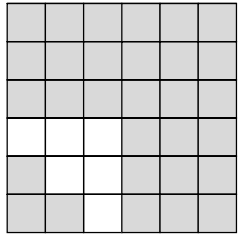
\[
%WinTpicVersion4.32a
{\unitlength 0.1in%
\begin{picture}(12.0000,12.0000)(34.0000,-16.0000)%
% BOX 2 0 1 0 Black Black  
% 2 3400 400 3600 600
% 
\special{pn 0}%
\special{sh 0.150}%
\special{pa 3400 400}%
\special{pa 3600 400}%
\special{pa 3600 600}%
\special{pa 3400 600}%
\special{pa 3400 400}%
\special{ip}%
\special{pn 8}%
\special{pa 3400 400}%
\special{pa 3600 400}%
\special{pa 3600 600}%
\special{pa 3400 600}%
\special{pa 3400 400}%
\special{pa 3600 400}%
\special{fp}%
% BOX 2 0 1 0 Black Black  
% 2 3600 400 3800 600
% 
\special{pn 0}%
\special{sh 0.150}%
\special{pa 3600 400}%
\special{pa 3800 400}%
\special{pa 3800 600}%
\special{pa 3600 600}%
\special{pa 3600 400}%
\special{ip}%
\special{pn 8}%
\special{pa 3600 400}%
\special{pa 3800 400}%
\special{pa 3800 600}%
\special{pa 3600 600}%
\special{pa 3600 400}%
\special{pa 3800 400}%
\special{fp}%
% BOX 2 0 1 0 Black Black  
% 2 3800 400 4000 600
% 
\special{pn 0}%
\special{sh 0.150}%
\special{pa 3800 400}%
\special{pa 4000 400}%
\special{pa 4000 600}%
\special{pa 3800 600}%
\special{pa 3800 400}%
\special{ip}%
\special{pn 8}%
\special{pa 3800 400}%
\special{pa 4000 400}%
\special{pa 4000 600}%
\special{pa 3800 600}%
\special{pa 3800 400}%
\special{pa 4000 400}%
\special{fp}%
% BOX 2 0 1 0 Black Black  
% 2 4000 400 4200 600
% 
\special{pn 0}%
\special{sh 0.150}%
\special{pa 4000 400}%
\special{pa 4200 400}%
\special{pa 4200 600}%
\special{pa 4000 600}%
\special{pa 4000 400}%
\special{ip}%
\special{pn 8}%
\special{pa 4000 400}%
\special{pa 4200 400}%
\special{pa 4200 600}%
\special{pa 4000 600}%
\special{pa 4000 400}%
\special{pa 4200 400}%
\special{fp}%
% BOX 2 0 1 0 Black Black  
% 2 4200 400 4400 600
% 
\special{pn 0}%
\special{sh 0.150}%
\special{pa 4200 400}%
\special{pa 4400 400}%
\special{pa 4400 600}%
\special{pa 4200 600}%
\special{pa 4200 400}%
\special{ip}%
\special{pn 8}%
\special{pa 4200 400}%
\special{pa 4400 400}%
\special{pa 4400 600}%
\special{pa 4200 600}%
\special{pa 4200 400}%
\special{pa 4400 400}%
\special{fp}%
% BOX 2 0 1 0 Black Black  
% 2 4400 400 4600 600
% 
\special{pn 0}%
\special{sh 0.150}%
\special{pa 4400 400}%
\special{pa 4600 400}%
\special{pa 4600 600}%
\special{pa 4400 600}%
\special{pa 4400 400}%
\special{ip}%
\special{pn 8}%
\special{pa 4400 400}%
\special{pa 4600 400}%
\special{pa 4600 600}%
\special{pa 4400 600}%
\special{pa 4400 400}%
\special{pa 4600 400}%
\special{fp}%
% BOX 2 0 1 0 Black Black  
% 2 3400 600 3600 800
% 
\special{pn 0}%
\special{sh 0.150}%
\special{pa 3400 600}%
\special{pa 3600 600}%
\special{pa 3600 800}%
\special{pa 3400 800}%
\special{pa 3400 600}%
\special{ip}%
\special{pn 8}%
\special{pa 3400 600}%
\special{pa 3600 600}%
\special{pa 3600 800}%
\special{pa 3400 800}%
\special{pa 3400 600}%
\special{pa 3600 600}%
\special{fp}%
% BOX 2 0 1 0 Black Black  
% 2 3600 600 3800 800
% 
\special{pn 0}%
\special{sh 0.150}%
\special{pa 3600 600}%
\special{pa 3800 600}%
\special{pa 3800 800}%
\special{pa 3600 800}%
\special{pa 3600 600}%
\special{ip}%
\special{pn 8}%
\special{pa 3600 600}%
\special{pa 3800 600}%
\special{pa 3800 800}%
\special{pa 3600 800}%
\special{pa 3600 600}%
\special{pa 3800 600}%
\special{fp}%
% BOX 2 0 1 0 Black Black  
% 2 3800 600 4000 800
% 
\special{pn 0}%
\special{sh 0.150}%
\special{pa 3800 600}%
\special{pa 4000 600}%
\special{pa 4000 800}%
\special{pa 3800 800}%
\special{pa 3800 600}%
\special{ip}%
\special{pn 8}%
\special{pa 3800 600}%
\special{pa 4000 600}%
\special{pa 4000 800}%
\special{pa 3800 800}%
\special{pa 3800 600}%
\special{pa 4000 600}%
\special{fp}%
% BOX 2 0 1 0 Black Black  
% 2 4000 600 4200 800
% 
\special{pn 0}%
\special{sh 0.150}%
\special{pa 4000 600}%
\special{pa 4200 600}%
\special{pa 4200 800}%
\special{pa 4000 800}%
\special{pa 4000 600}%
\special{ip}%
\special{pn 8}%
\special{pa 4000 600}%
\special{pa 4200 600}%
\special{pa 4200 800}%
\special{pa 4000 800}%
\special{pa 4000 600}%
\special{pa 4200 600}%
\special{fp}%
% BOX 2 0 1 0 Black Black  
% 2 4200 600 4400 800
% 
\special{pn 0}%
\special{sh 0.150}%
\special{pa 4200 600}%
\special{pa 4400 600}%
\special{pa 4400 800}%
\special{pa 4200 800}%
\special{pa 4200 600}%
\special{ip}%
\special{pn 8}%
\special{pa 4200 600}%
\special{pa 4400 600}%
\special{pa 4400 800}%
\special{pa 4200 800}%
\special{pa 4200 600}%
\special{pa 4400 600}%
\special{fp}%
% BOX 2 0 1 0 Black Black  
% 2 4400 600 4600 800
% 
\special{pn 0}%
\special{sh 0.150}%
\special{pa 4400 600}%
\special{pa 4600 600}%
\special{pa 4600 800}%
\special{pa 4400 800}%
\special{pa 4400 600}%
\special{ip}%
\special{pn 8}%
\special{pa 4400 600}%
\special{pa 4600 600}%
\special{pa 4600 800}%
\special{pa 4400 800}%
\special{pa 4400 600}%
\special{pa 4600 600}%
\special{fp}%
% BOX 2 0 1 0 Black Black  
% 2 3400 800 3600 1000
% 
\special{pn 0}%
\special{sh 0.150}%
\special{pa 3400 800}%
\special{pa 3600 800}%
\special{pa 3600 1000}%
\special{pa 3400 1000}%
\special{pa 3400 800}%
\special{ip}%
\special{pn 8}%
\special{pa 3400 800}%
\special{pa 3600 800}%
\special{pa 3600 1000}%
\special{pa 3400 1000}%
\special{pa 3400 800}%
\special{pa 3600 800}%
\special{fp}%
% BOX 2 0 1 0 Black Black  
% 2 3600 800 3800 1000
% 
\special{pn 0}%
\special{sh 0.150}%
\special{pa 3600 800}%
\special{pa 3800 800}%
\special{pa 3800 1000}%
\special{pa 3600 1000}%
\special{pa 3600 800}%
\special{ip}%
\special{pn 8}%
\special{pa 3600 800}%
\special{pa 3800 800}%
\special{pa 3800 1000}%
\special{pa 3600 1000}%
\special{pa 3600 800}%
\special{pa 3800 800}%
\special{fp}%
% BOX 2 0 1 0 Black Black  
% 2 3800 800 4000 1000
% 
\special{pn 0}%
\special{sh 0.150}%
\special{pa 3800 800}%
\special{pa 4000 800}%
\special{pa 4000 1000}%
\special{pa 3800 1000}%
\special{pa 3800 800}%
\special{ip}%
\special{pn 8}%
\special{pa 3800 800}%
\special{pa 4000 800}%
\special{pa 4000 1000}%
\special{pa 3800 1000}%
\special{pa 3800 800}%
\special{pa 4000 800}%
\special{fp}%
% BOX 2 0 1 0 Black Black  
% 2 4000 800 4200 1000
% 
\special{pn 0}%
\special{sh 0.150}%
\special{pa 4000 800}%
\special{pa 4200 800}%
\special{pa 4200 1000}%
\special{pa 4000 1000}%
\special{pa 4000 800}%
\special{ip}%
\special{pn 8}%
\special{pa 4000 800}%
\special{pa 4200 800}%
\special{pa 4200 1000}%
\special{pa 4000 1000}%
\special{pa 4000 800}%
\special{pa 4200 800}%
\special{fp}%
% BOX 2 0 1 0 Black Black  
% 2 4200 800 4400 1000
% 
\special{pn 0}%
\special{sh 0.150}%
\special{pa 4200 800}%
\special{pa 4400 800}%
\special{pa 4400 1000}%
\special{pa 4200 1000}%
\special{pa 4200 800}%
\special{ip}%
\special{pn 8}%
\special{pa 4200 800}%
\special{pa 4400 800}%
\special{pa 4400 1000}%
\special{pa 4200 1000}%
\special{pa 4200 800}%
\special{pa 4400 800}%
\special{fp}%
% BOX 2 0 1 0 Black Black  
% 2 4400 800 4600 1000
% 
\special{pn 0}%
\special{sh 0.150}%
\special{pa 4400 800}%
\special{pa 4600 800}%
\special{pa 4600 1000}%
\special{pa 4400 1000}%
\special{pa 4400 800}%
\special{ip}%
\special{pn 8}%
\special{pa 4400 800}%
\special{pa 4600 800}%
\special{pa 4600 1000}%
\special{pa 4400 1000}%
\special{pa 4400 800}%
\special{pa 4600 800}%
\special{fp}%
% BOX 2 0 3 0 Black White  
% 2 3400 1000 3600 1200
% 
\special{pn 8}%
\special{pa 3400 1000}%
\special{pa 3600 1000}%
\special{pa 3600 1200}%
\special{pa 3400 1200}%
\special{pa 3400 1000}%
\special{pa 3600 1000}%
\special{fp}%
% BOX 2 0 3 0 Black White  
% 2 3600 1000 3800 1200
% 
\special{pn 8}%
\special{pa 3600 1000}%
\special{pa 3800 1000}%
\special{pa 3800 1200}%
\special{pa 3600 1200}%
\special{pa 3600 1000}%
\special{pa 3800 1000}%
\special{fp}%
% BOX 2 0 3 0 Black Black  
% 2 3800 1000 4000 1200
% 
\special{pn 8}%
\special{pa 3800 1000}%
\special{pa 4000 1000}%
\special{pa 4000 1200}%
\special{pa 3800 1200}%
\special{pa 3800 1000}%
\special{pa 4000 1000}%
\special{fp}%
% BOX 2 0 1 0 Black Black  
% 2 4000 1000 4200 1200
% 
\special{pn 0}%
\special{sh 0.150}%
\special{pa 4000 1000}%
\special{pa 4200 1000}%
\special{pa 4200 1200}%
\special{pa 4000 1200}%
\special{pa 4000 1000}%
\special{ip}%
\special{pn 8}%
\special{pa 4000 1000}%
\special{pa 4200 1000}%
\special{pa 4200 1200}%
\special{pa 4000 1200}%
\special{pa 4000 1000}%
\special{pa 4200 1000}%
\special{fp}%
% BOX 2 0 1 0 Black Black  
% 2 4200 1000 4400 1200
% 
\special{pn 0}%
\special{sh 0.150}%
\special{pa 4200 1000}%
\special{pa 4400 1000}%
\special{pa 4400 1200}%
\special{pa 4200 1200}%
\special{pa 4200 1000}%
\special{ip}%
\special{pn 8}%
\special{pa 4200 1000}%
\special{pa 4400 1000}%
\special{pa 4400 1200}%
\special{pa 4200 1200}%
\special{pa 4200 1000}%
\special{pa 4400 1000}%
\special{fp}%
% BOX 2 0 1 0 Black Black  
% 2 4400 1000 4600 1200
% 
\special{pn 0}%
\special{sh 0.150}%
\special{pa 4400 1000}%
\special{pa 4600 1000}%
\special{pa 4600 1200}%
\special{pa 4400 1200}%
\special{pa 4400 1000}%
\special{ip}%
\special{pn 8}%
\special{pa 4400 1000}%
\special{pa 4600 1000}%
\special{pa 4600 1200}%
\special{pa 4400 1200}%
\special{pa 4400 1000}%
\special{pa 4600 1000}%
\special{fp}%
% BOX 2 0 1 0 Black Black  
% 2 3400 1200 3600 1400
% 
\special{pn 0}%
\special{sh 0.150}%
\special{pa 3400 1200}%
\special{pa 3600 1200}%
\special{pa 3600 1400}%
\special{pa 3400 1400}%
\special{pa 3400 1200}%
\special{ip}%
\special{pn 8}%
\special{pa 3400 1200}%
\special{pa 3600 1200}%
\special{pa 3600 1400}%
\special{pa 3400 1400}%
\special{pa 3400 1200}%
\special{pa 3600 1200}%
\special{fp}%
% BOX 2 0 3 0 Black White  
% 2 3600 1200 3800 1400
% 
\special{pn 8}%
\special{pa 3600 1200}%
\special{pa 3800 1200}%
\special{pa 3800 1400}%
\special{pa 3600 1400}%
\special{pa 3600 1200}%
\special{pa 3800 1200}%
\special{fp}%
% BOX 2 0 3 0 Black White  
% 2 3800 1200 4000 1400
% 
\special{pn 8}%
\special{pa 3800 1200}%
\special{pa 4000 1200}%
\special{pa 4000 1400}%
\special{pa 3800 1400}%
\special{pa 3800 1200}%
\special{pa 4000 1200}%
\special{fp}%
% BOX 2 0 1 0 Black Black  
% 2 4000 1200 4200 1400
% 
\special{pn 0}%
\special{sh 0.150}%
\special{pa 4000 1200}%
\special{pa 4200 1200}%
\special{pa 4200 1400}%
\special{pa 4000 1400}%
\special{pa 4000 1200}%
\special{ip}%
\special{pn 8}%
\special{pa 4000 1200}%
\special{pa 4200 1200}%
\special{pa 4200 1400}%
\special{pa 4000 1400}%
\special{pa 4000 1200}%
\special{pa 4200 1200}%
\special{fp}%
% BOX 2 0 1 0 Black Black  
% 2 4200 1200 4400 1400
% 
\special{pn 0}%
\special{sh 0.150}%
\special{pa 4200 1200}%
\special{pa 4400 1200}%
\special{pa 4400 1400}%
\special{pa 4200 1400}%
\special{pa 4200 1200}%
\special{ip}%
\special{pn 8}%
\special{pa 4200 1200}%
\special{pa 4400 1200}%
\special{pa 4400 1400}%
\special{pa 4200 1400}%
\special{pa 4200 1200}%
\special{pa 4400 1200}%
\special{fp}%
% BOX 2 0 1 0 Black Black  
% 2 4400 1200 4600 1400
% 
\special{pn 0}%
\special{sh 0.150}%
\special{pa 4400 1200}%
\special{pa 4600 1200}%
\special{pa 4600 1400}%
\special{pa 4400 1400}%
\special{pa 4400 1200}%
\special{ip}%
\special{pn 8}%
\special{pa 4400 1200}%
\special{pa 4600 1200}%
\special{pa 4600 1400}%
\special{pa 4400 1400}%
\special{pa 4400 1200}%
\special{pa 4600 1200}%
\special{fp}%
% BOX 2 0 1 0 Black Black  
% 2 3400 1400 3600 1600
% 
\special{pn 0}%
\special{sh 0.150}%
\special{pa 3400 1400}%
\special{pa 3600 1400}%
\special{pa 3600 1600}%
\special{pa 3400 1600}%
\special{pa 3400 1400}%
\special{ip}%
\special{pn 8}%
\special{pa 3400 1400}%
\special{pa 3600 1400}%
\special{pa 3600 1600}%
\special{pa 3400 1600}%
\special{pa 3400 1400}%
\special{pa 3600 1400}%
\special{fp}%
% BOX 2 0 1 0 Black Black  
% 2 3600 1400 3800 1600
% 
\special{pn 0}%
\special{sh 0.150}%
\special{pa 3600 1400}%
\special{pa 3800 1400}%
\special{pa 3800 1600}%
\special{pa 3600 1600}%
\special{pa 3600 1400}%
\special{ip}%
\special{pn 8}%
\special{pa 3600 1400}%
\special{pa 3800 1400}%
\special{pa 3800 1600}%
\special{pa 3600 1600}%
\special{pa 3600 1400}%
\special{pa 3800 1400}%
\special{fp}%
% BOX 2 0 3 0 Black White  
% 2 3800 1400 4000 1600
% 
\special{pn 8}%
\special{pa 3800 1400}%
\special{pa 4000 1400}%
\special{pa 4000 1600}%
\special{pa 3800 1600}%
\special{pa 3800 1400}%
\special{pa 4000 1400}%
\special{fp}%
% BOX 2 0 1 0 Black Black  
% 2 4000 1400 4200 1600
% 
\special{pn 0}%
\special{sh 0.150}%
\special{pa 4000 1400}%
\special{pa 4200 1400}%
\special{pa 4200 1600}%
\special{pa 4000 1600}%
\special{pa 4000 1400}%
\special{ip}%
\special{pn 8}%
\special{pa 4000 1400}%
\special{pa 4200 1400}%
\special{pa 4200 1600}%
\special{pa 4000 1600}%
\special{pa 4000 1400}%
\special{pa 4200 1400}%
\special{fp}%
% BOX 2 0 1 0 Black Black  
% 2 4200 1400 4400 1600
% 
\special{pn 0}%
\special{sh 0.150}%
\special{pa 4200 1400}%
\special{pa 4400 1400}%
\special{pa 4400 1600}%
\special{pa 4200 1600}%
\special{pa 4200 1400}%
\special{ip}%
\special{pn 8}%
\special{pa 4200 1400}%
\special{pa 4400 1400}%
\special{pa 4400 1600}%
\special{pa 4200 1600}%
\special{pa 4200 1400}%
\special{pa 4400 1400}%
\special{fp}%
% BOX 2 0 1 0 Black Black  
% 2 4400 1400 4600 1600
% 
\special{pn 0}%
\special{sh 0.150}%
\special{pa 4400 1400}%
\special{pa 4600 1400}%
\special{pa 4600 1600}%
\special{pa 4400 1600}%
\special{pa 4400 1400}%
\special{ip}%
\special{pn 8}%
\special{pa 4400 1400}%
\special{pa 4600 1400}%
\special{pa 4600 1600}%
\special{pa 4400 1600}%
\special{pa 4400 1400}%
\special{pa 4600 1400}%
\special{fp}%
\end{picture}}%
\]
The Schubert cell $\SC{w}$ is represented by the set of matrices of the form
\begin{align}\label{eq: Schubert cell example}
\left(\begin{array}{rrrrrr}
z_{13} & z_{12} & 1 & 0 & 0 & 0 \\
z_{23} & 1 & 0 & 0 & 0 & 0 \\
1 & 0 & 0 & 0 & 0 & 0 \\
0 & 0 & 0 & z_{46} & z_{45} & 1 \\
0 & 0 & 0 & z_{56} & 1 & 0 \\
0 & 0 &0 & 1 & 0 & 0
\end{array}\right)
\end{align}
with free genuine variables $z_{12}$, $z_{45}$, $z_{13}$, $z_{46}$, $z_{23}$, $z_{56}$ (in order), 
where the genuine variables $z_{p,q}$ satisfying $p>q$ vanish on $\SC{w}$.
One can directly verify that $\Hess(N,h)\cap \SC{w}=\SC{w}$ in this case so that $z_{12}$, $z_{45}$, $z_{13}$, $z_{46}$, $z_{23}$, $z_{56}$ form the free coordinates on the Hessenberg-Schubert cell $\Hess(N,h)\cap \SC{w}$.
The generators of $J_{w,h}$ are
$g_{4,3}$, $g_{4,2}$, $g_{5,3}$, $g_{4,1}$, $g_{5,2}$, $g_{6,3}$ (in order), and the Jacobian matrix on $\Hess(N,h)\cap \SC{w}$ is 
$$
%WinTpicVersion4.32a
{\unitlength 0.1in%
\begin{picture}(55.8500,16.5800)(12.0000,-22.0200)%
% STR 2 0 3 0 Black White  
% 4 1200 1087 1200 1170 2 0 0 0
% $g_{4,2}$
\put(12.0000,-11.7000){\makebox(0,0)[lb]{$g_{4,2}$}}%
% STR 2 0 3 0 Black White  
% 4 1200 839 1200 922 2 0 0 0
% $g_{4,3}$
\put(12.0000,-9.2200){\makebox(0,0)[lb]{$g_{4,3}$}}%
% STR 2 0 3 0 Black White  
% 4 1200 1583 1200 1665 2 0 0 0
% $g_{4,1}$
\put(12.0000,-16.6500){\makebox(0,0)[lb]{$g_{4,1}$}}%
% STR 2 0 3 0 Black White  
% 4 1200 1830 1200 1913 2 0 0 0
% $g_{5,2}$
\put(12.0000,-19.1300){\makebox(0,0)[lb]{$g_{5,2}$}}%
% STR 2 0 3 0 Black White  
% 4 1200 1335 1200 1417 2 0 0 0
% $g_{5,3}$
\put(12.0000,-14.1700){\makebox(0,0)[lb]{$g_{5,3}$}}%
% STR 2 0 3 0 Black White  
% 4 1200 2078 1200 2161 2 0 0 0
% $g_{6,3}$
\put(12.0000,-21.6100){\makebox(0,0)[lb]{$g_{6,3}$}}%
% STR 2 0 3 0 Black White  
% 4 4174 591 4174 674 2 0 0 0
% $z_{5,1}$
\put(41.7400,-6.7400){\makebox(0,0)[lb]{$z_{5,1}$}}%
% STR 2 0 3 0 Black White  
% 4 5884 1830 5884 1913 2 0 0 0
% $1$
\put(58.8400,-19.1300){\makebox(0,0)[lb]{$1$}}%
% STR 2 0 3 0 Black White  
% 4 4504 591 4504 674 2 0 0 0
% $z_{4,3}$
\put(45.0400,-6.7400){\makebox(0,0)[lb]{$z_{4,3}$}}%
% STR 2 0 3 0 Black White  
% 4 3183 591 3183 674 2 0 0 0
% $z_{4,1}$
\put(31.8300,-6.7400){\makebox(0,0)[lb]{$z_{4,1}$}}%
% STR 2 0 3 0 Black White  
% 4 2522 591 2522 674 2 0 0 0
% $z_{4,6}$
\put(25.2200,-6.7400){\makebox(0,0)[lb]{$z_{4,6}$}}%
% STR 2 0 3 0 Black White  
% 4 4835 591 4835 674 2 0 0 0
% $z_{5,2}$
\put(48.3500,-6.7400){\makebox(0,0)[lb]{$z_{5,2}$}}%
% STR 2 0 3 0 Black White  
% 4 3513 591 3513 674 2 0 0 0
% $z_{5,6}$
\put(35.1300,-6.7400){\makebox(0,0)[lb]{$z_{5,6}$}}%
% STR 2 0 3 0 Black White  
% 4 5165 591 5165 674 2 0 0 0
% $z_{6,1}$
\put(51.6500,-6.7400){\makebox(0,0)[lb]{$z_{6,1}$}}%
% STR 2 0 3 0 Black White  
% 4 1530 591 1530 674 2 0 0 0
% $z_{1,2}$
\put(15.3000,-6.7400){\makebox(0,0)[lb]{$z_{1,2}$}}%
% STR 2 0 3 0 Black White  
% 4 5826 591 5826 674 2 0 0 0
% $z_{6,2}$
\put(58.2600,-6.7400){\makebox(0,0)[lb]{$z_{6,2}$}}%
% STR 2 0 3 0 Black White  
% 4 1861 591 1861 674 2 0 0 0
% $z_{4,5}$
\put(18.6100,-6.7400){\makebox(0,0)[lb]{$z_{4,5}$}}%
% STR 2 0 3 0 Black White  
% 4 3843 591 3843 674 2 0 0 0
% $z_{4,2}$
\put(38.4300,-6.7400){\makebox(0,0)[lb]{$z_{4,2}$}}%
% STR 2 0 3 0 Black White  
% 4 2852 591 2852 674 2 0 0 0
% $z_{2,3}$
\put(28.5200,-6.7400){\makebox(0,0)[lb]{$z_{2,3}$}}%
% STR 2 0 3 0 Black White  
% 4 2191 591 2191 674 2 0 0 0
% $z_{1,3}$
\put(21.9100,-6.7400){\makebox(0,0)[lb]{$z_{1,3}$}}%
% STR 2 0 3 0 Black White  
% 4 4835 1335 4835 1417 2 0 0 0
% $-1$
\put(48.3500,-14.1700){\makebox(0,0)[lb]{$-1$}}%
% LINE 2 0 3 0 Black White  
% 2 1200 980 6785 980
% 
\special{pn 8}%
\special{pa 1200 980}%
\special{pa 6785 980}%
\special{fp}%
% LINE 2 0 3 0 Black White  
% 2 1200 1475 6785 1475
% 
\special{pn 8}%
\special{pa 1200 1475}%
\special{pa 6785 1475}%
\special{fp}%
% LINE 2 2 3 0 Black White  
% 2 1200 1227 6785 1227
% 
\special{pn 8}%
\special{pa 1200 1227}%
\special{pa 6785 1227}%
\special{dt 0.045}%
% LINE 2 2 3 0 Black White  
% 2 1200 1971 6785 1971
% 
\special{pn 8}%
\special{pa 1200 1971}%
\special{pa 6785 1971}%
\special{dt 0.045}%
% LINE 2 2 3 0 Black White  
% 2 1200 1723 6785 1723
% 
\special{pn 8}%
\special{pa 1200 1723}%
\special{pa 6785 1723}%
\special{dt 0.045}%
% LINE 2 2 3 0 Black White  
% 2 1803 550 1803 2202
% 
\special{pn 8}%
\special{pa 1803 550}%
\special{pa 1803 2202}%
\special{dt 0.045}%
% LINE 2 0 3 0 Black White  
% 2 2133 550 2133 2202
% 
\special{pn 8}%
\special{pa 2133 550}%
\special{pa 2133 2202}%
\special{fp}%
% LINE 2 2 3 0 Black White  
% 2 2464 550 2464 2202
% 
\special{pn 8}%
\special{pa 2464 550}%
\special{pa 2464 2202}%
\special{dt 0.045}%
% LINE 2 0 3 0 Black White  
% 2 2794 550 2794 2202
% 
\special{pn 8}%
\special{pa 2794 550}%
\special{pa 2794 2202}%
\special{fp}%
% LINE 2 2 3 0 Black White  
% 2 3125 550 3125 2202
% 
\special{pn 8}%
\special{pa 3125 550}%
\special{pa 3125 2202}%
\special{dt 0.045}%
% LINE 2 0 3 0 Black White  
% 2 1200 732 6785 732
% 
\special{pn 8}%
\special{pa 1200 732}%
\special{pa 6785 732}%
\special{fp}%
% LINE 2 2 3 0 Black White  
% 2 3455 550 3455 2202
% 
\special{pn 8}%
\special{pa 3455 550}%
\special{pa 3455 2202}%
\special{dt 0.045}%
% LINE 2 0 3 0 Black White  
% 2 3786 550 3786 2202
% 
\special{pn 8}%
\special{pa 3786 550}%
\special{pa 3786 2202}%
\special{fp}%
% LINE 2 2 3 0 Black White  
% 2 4116 550 4116 2202
% 
\special{pn 8}%
\special{pa 4116 550}%
\special{pa 4116 2202}%
\special{dt 0.045}%
% LINE 2 0 3 0 Black White  
% 2 4447 550 4447 2202
% 
\special{pn 8}%
\special{pa 4447 550}%
\special{pa 4447 2202}%
\special{fp}%
% LINE 2 2 3 0 Black White  
% 2 4777 550 4777 2202
% 
\special{pn 8}%
\special{pa 4777 550}%
\special{pa 4777 2202}%
\special{dt 0.045}%
% LINE 2 2 3 0 Black White  
% 2 5107 550 5107 2202
% 
\special{pn 8}%
\special{pa 5107 550}%
\special{pa 5107 2202}%
\special{dt 0.045}%
% LINE 2 0 3 0 Black White  
% 2 5438 550 5438 2202
% 
\special{pn 8}%
\special{pa 5438 550}%
\special{pa 5438 2202}%
\special{fp}%
% LINE 2 2 3 0 Black White  
% 2 5768 550 5768 2202
% 
\special{pn 8}%
\special{pa 5768 550}%
\special{pa 5768 2202}%
\special{dt 0.045}%
% LINE 2 0 3 0 Black White  
% 2 6099 550 6099 2202
% 
\special{pn 8}%
\special{pa 6099 550}%
\special{pa 6099 2202}%
\special{fp}%
% LINE 2 0 3 0 Black White  
% 2 1473 550 1473 2202
% 
\special{pn 8}%
\special{pa 1473 550}%
\special{pa 1473 2202}%
\special{fp}%
% STR 2 0 3 0 Black White  
% 4 4504 1087 4504 1170 2 0 0 0
% $-1$
\put(45.0400,-11.7000){\makebox(0,0)[lb]{$-1$}}%
% STR 2 0 3 0 Black White  
% 4 4893 1079 4893 1161 2 0 0 0
% $1$
\put(48.9300,-11.6100){\makebox(0,0)[lb]{$1$}}%
% STR 2 0 3 0 Black White  
% 4 5223 1327 5223 1409 2 0 0 0
% $1$
\put(52.2300,-14.0900){\makebox(0,0)[lb]{$1$}}%
% STR 2 0 3 0 Black White  
% 4 5496 1830 5496 1913 2 0 0 0
% $-1$
\put(54.9600,-19.1300){\makebox(0,0)[lb]{$-1$}}%
% STR 2 0 3 0 Black White  
% 4 6338 1822 6338 1905 2 0 0 0
% $z_{5,6}$
\put(63.3800,-19.0500){\makebox(0,0)[lb]{$z_{5,6}$}}%
% STR 2 0 3 0 Black White  
% 4 6322 591 6322 674 2 0 0 0
% $z_{6,3}$
\put(63.2200,-6.7400){\makebox(0,0)[lb]{$z_{6,3}$}}%
% STR 2 0 3 0 Black White  
% 4 5496 591 5496 674 2 0 0 0
% $z_{5,3}$
\put(54.9600,-6.7400){\makebox(0,0)[lb]{$z_{5,3}$}}%
% STR 2 0 3 0 Black White  
% 4 6120 2077 6120 2160 2 0 0 0
% $z_{1,2}-z_{2,3}$
\put(61.2000,-21.6000){\makebox(0,0)[lb]{$z_{1,2}-z_{2,3}$}}%
% STR 2 0 3 0 Black White  
% 4 3843 839 3843 922 2 0 0 0
% $-1$
\put(38.4300,-9.2200){\makebox(0,0)[lb]{$-1$}}%
% STR 2 0 3 0 Black White  
% 4 4232 831 4232 913 2 0 0 0
% $1$
\put(42.3200,-9.1300){\makebox(0,0)[lb]{$1$}}%
% STR 2 0 3 0 Black White  
% 4 4579 814 4579 897 2 0 0 0
% $*$
\put(45.7900,-8.9700){\makebox(0,0)[lb]{$*$}}%
% STR 2 0 3 0 Black White  
% 4 5553 1574 5553 1657 2 0 0 0
% $1$
\put(55.5300,-16.5700){\makebox(0,0)[lb]{$1$}}%
% STR 2 0 3 0 Black White  
% 4 5826 2078 5826 2161 2 0 0 0
% $-1$
\put(58.2600,-21.6100){\makebox(0,0)[lb]{$-1$}}%
% STR 2 0 3 0 Black White  
% 4 5826 2078 5826 2161 2 0 0 0
% $-1$
\put(58.2600,-21.6100){\makebox(0,0)[lb]{$-1$}}%
% STR 2 0 3 0 Black White  
% 4 4909 814 4909 897 2 0 0 0
% $*$
\put(49.0900,-8.9700){\makebox(0,0)[lb]{$*$}}%
% STR 2 0 3 0 Black White  
% 4 5240 814 5240 897 2 0 0 0
% $*$
\put(52.4000,-8.9700){\makebox(0,0)[lb]{$*$}}%
% STR 2 0 3 0 Black White  
% 4 5570 814 5570 897 2 0 0 0
% $*$
\put(55.7000,-8.9700){\makebox(0,0)[lb]{$*$}}%
% STR 2 0 3 0 Black White  
% 4 5900 814 5900 897 2 0 0 0
% $*$
\put(59.0000,-8.9700){\makebox(0,0)[lb]{$*$}}%
% STR 2 0 3 0 Black White  
% 4 5900 1310 5900 1393 2 0 0 0
% $*$
\put(59.0000,-13.9300){\makebox(0,0)[lb]{$*$}}%
% STR 2 0 3 0 Black White  
% 4 5570 1310 5570 1393 2 0 0 0
% $*$
\put(55.7000,-13.9300){\makebox(0,0)[lb]{$*$}}%
% STR 2 0 3 0 Black White  
% 4 5570 1054 5570 1137 2 0 0 0
% $*$
\put(55.7000,-11.3700){\makebox(0,0)[lb]{$*$}}%
% STR 2 0 3 0 Black White  
% 4 5900 1054 5900 1137 2 0 0 0
% $*$
\put(59.0000,-11.3700){\makebox(0,0)[lb]{$*$}}%
% STR 2 0 3 0 Black White  
% 4 6380 1054 6380 1137 2 0 0 0
% $*$
\put(63.8000,-11.3700){\makebox(0,0)[lb]{$*$}}%
% STR 2 0 3 0 Black White  
% 4 6380 814 6380 897 2 0 0 0
% $*$
\put(63.8000,-8.9700){\makebox(0,0)[lb]{$*$}}%
% STR 2 0 3 0 Black White  
% 4 6380 1310 6380 1393 2 0 0 0
% $*$
\put(63.8000,-13.9300){\makebox(0,0)[lb]{$*$}}%
% STR 2 0 3 0 Black White  
% 4 6239 1574 6239 1657 2 0 0 0
% $-z_{4,5}$
\put(62.3900,-16.5700){\makebox(0,0)[lb]{$-z_{4,5}$}}%
\end{picture}}%
$$
where the symbols $*$ represent some polynomials on the genuine variables appearing in \eqref{eq: Schubert cell example}, and we omit zeros on the empty spots.
Evaluating this matrix on the permutation flag $w_{\bullet}$ (i.e., taking $z_{ij}=0$ for all the genuine variables), we see that $\rank(\Jac|_{w_\bullet})= 5$ at $w_{\bullet}$.
In fact, 
the lower right submatrix
$$
\left(\begin{array}{cc|c}
 1 & 0 & 0 \\
 -1 & 1 & 0 \\
 0 & -1 & 0
\end{array}\right)
$$ 
of $\Jac|_{w_\bullet}$ arising from the generators $g_{4,1}, g_{5,2}, g_{6,3}$ does not have its proper full rank, where it can also be explained by Theorem~\ref{thm:fixed_point_singular} since $(4,1), (5,2), (6,3)\in \hc{w}$ form a lower diagonal full string.
Thus, $w_{\bullet}$ is a singular point of $\Hess(N,h)$.
However, on the Hessenberg-Schubert cell $\Hess(N,h)\cap \SC{w}$, there is a dense open subset where the the Jacobian matrix has its proper full rank. More precisely, any point where the determinant of the lower right submatrix
$$
\left(\begin{array}{cc|c}
 1 & 0 & -z_{45} \\
 -1 & 1 & z_{56} \\
 0 & -1 & z_{12} - z_{23}
\end{array}\right)
$$ 
is nonzero will have its proper full rank.
That is, the following Zariski open subset of $\Hess(N,h)\cap \SC{w}$ consists of nonsingular points of $\Hess(N,h)$ :

$$ \left \{ \left(\begin{array}{cccccc}
z_{13} & z_{12} & 1 & 0 & 0 & 0 \\
z_{23} & 1 & 0 & 0 & 0 & 0 \\
1 & 0 & 0 & 0 & 0 & 0 \\
0 & 0 & 0 & z_{46} & z_{45} & 1 \\
0 & 0 & 0 & z_{56} & 1 & 0 \\
0 & 0 &0 & 1 & 0 & 0
\end{array}\right) \middle | \ z_{12}-z_{23} - z_{45}+z_{56} \neq 0 \right \} .$$
Hence, in this cell, the permutation flag $w_{\bullet}$ is a singular point of $\Hess(N,h)$, but generic points of the cell $\Hess(N,h)\cap \SC{w}$ are nonsingular points of $\Hess(N,h)$.
\end{example}

\bigskip

\section{Normality of regular nilpotent Hessenberg varieties}\label{sec: normality}

The goal of this section is to prove Theorem~\ref{thm:normal}, which explicitly classifies when regular nilpotent Hessenberg varieties $\Hess(N,h)$ are normal algebraic variety based on the shape of their corresponding Hessenberg functions $h$. As they are local complete intersections (\cite[Corollary~3.17]{ADGH18}), they are Cohen-Macaulay (\cite[Proposition~18.13]{Eis95}). Hence, $\Hess(N,h)$ is normal if and only if its singular locus is codimension 2 or greater \cite[Chapter~II,~Proposition~8.23]{Har77}. For $\Hess(N,h)$ that are not normal, the proof of Theorem~\ref{thm:normal} also describes the the codimension 1 Hessenberg-Schubert cells which consists of singular points of $\Hess(N,h)$. 
As in the previous sections, we always assume that a Hessenberg function $h\colon[n]\rightarrow [n]$ always satisfies that
\begin{align}\label{eq: basis assumption} 
 h(i) \ge i+1 \qquad (1\le i<n).
\end{align}

\subsection{Codimension 1 Hessenberg-Schubert cells}

Let $w_0$ denote the longest element in $\Sn$.  
Based on the work of Cho-Hong-Lee (\cite[Definition~4.3]{CHL21}), we begin with describing the set of codimension 1 Hessenberg-Schubert cells of $\Hess(N,h)$ under the assumption \eqref{eq: basis assumption}.

\begin{definition} \label{def:codim1}
Let $h(0)\coloneqq1$. For $1\le i\le n-1$, we define a permutation $p_i\in\Sn$  by taking cases as follows $($we write the one-line notation in the second line of each definition$):$ 
\begin{itemize} 
 \item[\rm(i)] if $h(i-1)>i$ and $h(i)>i+1$, then
  \begin{align*}
   p_i  
   &\coloneqq w_0 s_i \  \\
   &\ = n , n-1 , \cdots , n+2-i , n-i , n+1-i , n-i-1 , \cdots , 2 ,1   ,
  \end{align*}  
 \item[\rm(ii)] if $h(i-1)>i$ and $h(i)=i+1$, then
  \begin{align*}
   p_i 
   & \coloneqq w_0 s_{n-1}\cdots s_{i+1}s_{i} \\ 
   &\ = n , n-1 , \cdots , n+2-i , 1 , n+1-i , n-i , \cdots , 3 , 2 ,
  \end{align*}  
 \item[\rm(iii)] if $h(i-1)=i$ and $h(i)>i+1$, then
  \begin{align*}
   p_i 
  & \coloneqq w_0 s_1s_2\cdots s_i \\
  & \ = n-1 ,n-2 , \cdots , n-i , n , n-i-1 , \cdots , 2 , 1 ,
  \end{align*}  
  \item[\rm(iv)] if $h(i-1)=i$ and $h(i)=i+1$, then
  \begin{align*}
   p_i &
   \coloneqq w_{P_i} \\ &  \ = i , i-1 , \cdots , 1 , n , n-1 , \cdots , i+2 , i+1 ,
  \end{align*}  
  where $w_{P_i}$ is the longest element of the parabolic subgroup $P_i\leq \Sn$ generated by all simple transpositions in $\Sn$ excluding $s_i$, i.e., $$P_i= \langle s_1,s_2, \ldots, s_{i-1}, s_{i+1},s_{i+2},\ldots, s_{n-1}  \rangle. $$
\end{itemize}  
\end{definition}

The following result is proven essentially in Sections 3 and 4 of \cite{CHL21}.
We restate the result in our notation here for ease of reference.

\begin{proposition} \label{prop:codim1}
We have $(p_i )_{\bullet} \in\Hess(N,h)$ for $1\le i\le n-1$.
Moreover, the codimension-$1$ Hessenberg-Schubert cells of $\Hess(N,h)$ are precisely $\Hess(N,h)\cap \SC{p_i}$ for $1\le i\le n-1$.
\end{proposition}

\begin{proof}
Writing $d= \dim_{\C}\Hess(N,h)$, 
the $(2d-2)$-th Betti number of $\Hess(N,h)$ is equal to the second Betti number of $\Hess(N,h)$ by \cite[Theorem 12]{Pre16}. Since we are assuming \eqref{eq: basis assumption}, the latter number is known to be $n-1$ from \cite[Corollary~4.13]{Pre13}.
Thus, there are precisely $n-1$ Hessenberg-Schubert cells of (complex) codimension 1.
In \cite{CHL21}, the authors introduced the permutations $w^{[1]}$, $w^{[2]}$, $\ldots,$ $w^{[n-1]}\in\Sn$  which satisfy
\begin{itemize}
 \item[(1)] $(w_0w^{[i]})_{\bullet} \in \Hess(N,h)$ for $1\le i\le n-1$,
 \item[(2)] $\dim_{\C}\Hess(N,h)\cap \SC{w_0w^{[i]}} = \dim_{\C}\Hess(N,h)-1$
\end{itemize}
(see \cite[Proposition~4.5 and the proof of Proposition~3.2]{CHL21}).
By their Definition~4.3, it is clear that $p_i=w_0w^{[i]}$ for $1\le i\le n-1$. 
(We remark that \cite{CHL21} used the convention $h(0)=2$ whereas we use the convention $h(0)=1$. The difference of conventions does not cause any difference since when $i=1$ cases (i) and (iii) give the same permutation and so do cases (ii) and (iv).)
Thus, the claim follows.
\end{proof}

\vspace{10pt}

\begin{example}
Let $n=7$ and $h=(2,4,5,5,6,7,7)$.
Then the permutations of the codimension 1 Hessenberg-Schubert cells of $\Hess(N,h)$ are
\begin{align*} 
 &p_1 = 1\ 7\ 6\ 5\ 4\ 3\ 2 \quad (h(0)=1, \ h(1)=2),\\
 &p_2 = 6\ 5\ 7\ 4\ 3\ 2\ 1 \quad (h(1)=2, \ h(2)>3),\\
 &p_3 = 7\ 6\ 4\ 5\ 3\ 2\ 1 \quad (h(2)>3, \ h(3)>4),\\
 &p_4 = 7\ 6\ 5\ 1\ 4\ 3\ 2 \quad (h(3)>4, \ h(4)=5),\\
 &p_5 = 5\ 4\ 3\ 2\ 1\ 7\ 6 \quad (h(4)=5, \ h(5)=6),\\
 &p_6 = 6\ 5\ 4\ 3\ 2\ 1\ 7 \quad (h(5)=6, \ h(6)=7).
\end{align*}  
One can also verify that $(p_1)_{\bullet},\ldots, (p_7)_{\bullet} \in \Hess(N,h)$ directly by Lemma~\ref{lem: permutation flags in Hess 2}.
\end{example}

\vspace{10pt}

Now that we have established that the permutations $p_i$ correspond to codimension~1 Hessenberg-Schubert cells $\Hess(N,h)\cap \SC{p_i}$, we set about analyzing when each cell $\Hess(N,h)\cap \SC{p_i}$ consists of singular points of $\Hess(N,h)$.  We will use the following lemma repeatedly to classify singularity of these codimension 1 cells.

\begin{lemma} \label{lem:ui_not_n1 2}
For a permutation flag $w_{\bullet}\in \Hess(N,h)$, we have $(i,j)\in \hc{w}$ if and only if 
$(w^{-1}(i),w^{-1}(j))\in \hc{}$, where $\hc{}\coloneqq \hc{e}$ is the standard Hessenberg complement.
\end{lemma}

\begin{proof}
The claim follows directly from \eqref{eq:r_h}.
\end{proof}
 
\vspace{10pt}
The next two lemmas provide case-by-case analysis of when a codimension 1 Hessenberg-Schubert cell $\Hess(N,h)\cap \SC{p_i}$ 
consists of nonsingular points of $\Hess(N,h)$.

\begin{lemma}\label{lem: nonsingular codim 1 cell}
If $p_i\in\Sn$ is defined as in cases \text{\rm (i), (ii), or (iii)},
then the permutation flag $(p_i)_{\bullet} \in \Hess(N,h)$ is a nonsingular point of $\Hess(N,h)$.
\end{lemma}

\begin{proof}
By Theorem~\ref{thm:fixed_point_singular}, it suffices to show that $\hc{p_i}$ does not contain any lower diagonal full-string. We consider the cases (i), (ii), (iii) of the definition of $p_i$ separately.

\underline{Case (i) :}
By definition, we have $p_i = w_0 s_i$.
In this case, we have
\begin{align*}
 p_i^{-1}(1) = 
  \begin{cases} 
   n \quad &(i<n-1), \\
   n-1 &(i=n-1).
 \end{cases}
\end{align*}  
So we have $p_i^{-1}(1)\ge n-1$. 
ence, it follows from Lemma~\ref{lem:ui_not_n1 2} that we have $(a,1)\notin \hc{p_i}$ for all $1\le a\le n$ since $p_i^{-1}(1)\ge n-1$ implies that $(p_i^{-1}(a),p_i^{-1}(1))\notin \hc{}$.
Thus, $\hc{p_i}$ does not contain any lower diagonal full-string in this case.

\underline{Case (ii) :}
By definition, we have $p_i= w_0 s_{n-1}\cdots s_{i+1}s_{i}$.
In this case, we have
\begin{align*}
 p_i^{-1}(n) = 
  \begin{cases} 
   1 \quad &(i>1), \\
   2 &(i=1).
 \end{cases}
\end{align*}  
So we have $p_i^{-1}(n)\le 2$.
Hence, it follows from Lemma~\ref{lem:ui_not_n1 2} that we have $(n,b)\notin \hc{p_i}$ for all $1\le b\le n$ since $p_i^{-1}(n)\le 2$ implies that $(p_i^{-1}(n),p_i^{-1}(b))\notin \hc{}$.
Thus, $\hc{p_i}$ does not contain any lower diagonal full-string in this case as well.

\underline{Case (iii) :}
By definition, we have $p_i= w_0 s_1s_2\cdots s_i$.
In this case, we have
\begin{align*}
 p_i^{-1}(1) = 
  \begin{cases} 
   n \quad &(i<n-1), \\
   n-1 &(i=n-1).
 \end{cases}
\end{align*}  
Thus, the same argument as in Case (i) proves that $\hc{p_i}$ does not contain any lower diagonal full-string in this case as well.
\end{proof}

\vspace{10pt}

\begin{lemma}\label{lem: nonsingular codim 1 cell 2}
The permutation flags $(p_1)_{\bullet}$ and $(p_{n-1})_{\bullet}$ are nonsingular points of $\Hess(N,h)$.
\end{lemma}

\begin{proof}
We first show that $(p_1)_{\bullet}$ is a nonsingular point in $\Hess(N,h)$.
We take cases.
If $h(1)=2$, then (by the convention $h(0)=1$) $p_1$ is defined as case (iv), and hence
  \begin{align*}
   p_1
   = 1 , n , n-1 , \cdots , 3 , 2 
  \end{align*}  
in one-line notation.
Notice that cases (ii) and (iv) define the same permutation $p_1$.
So we may think of this $p_1$ as defined as in case (ii) so that Lemma~\ref{lem: nonsingular codim 1 cell} implies that $(p_1)_{\bullet}$ is a nonsingular point of $\Hess(N,h)$.
If $h(1)>2$, then (by $h(0)=1$) $p_1$ is defined as in case (iii) so that $(p_1)_{\bullet}$ is a nonsingular point of $\Hess(N,h)$ by Lemma~\ref{lem: nonsingular codim 1 cell}. 

We next show that $(p_{n-1})_{\bullet}$ is a nonsingular point of $\Hess(N,h)$.
We take cases again.
If $h(n-2)=n-1$, then (by $h(n-1)=n$) $p_{n-1}$ is defined as in case (iv), and hence
  \begin{align*}
   p_{n-1}
   = n-1 , n-2 , \cdots , 1 , n
  \end{align*}  
in one-line notation.
Notice that cases (iii) and (iv) define the same permutation $p_{n-1}$.
So we may think of this $p_{n-1}$ as defined as in case (iii) so that Lemma~\ref{lem: nonsingular codim 1 cell} implies that $(p_{n-1})_{\bullet}$ is a nonsingular point of $\Hess(N,h)$.
If $h(n-2)>n-1$, then (by $h(n-1)=n$) $p_{n-1}$ is defined as in case (ii) so that it is a nonsingular point of $\Hess(N,h)$ by Lemma~\ref{lem: nonsingular codim 1 cell}.
\end{proof}

\vspace{12pt}

We are ready to classify regular nilpotent Hessenberg varieties which are normal algebraic varieties. 
Recall that a regular nilpotent Hessenberg variety $\Hess(N,h)$ is always irreducible (see Section~\ref{sec: background}) and that we are assuming the condition \eqref{eq: basis assumption}.

\begin{theorem} \label{thm:normal}
A regular nilpotent Hessenberg variety $\Hess(N,h)$ is normal if and only if we have $h(i-1)> i$ or $h(i)> i+1$ for all $1 < i < n-1$.
\end{theorem}

\begin{proof}
To begin with, we note that $\Hess(N,h)$ a local complete intersection (\cite[Corollary~3.17]{ADGH18}).
Thus, $\Hess(N,h)$ is normal if and only if the codimension of its singular locus is greater than or equal to $2$ (\cite[Chapter~II,~Proposition~8.23]{Har77}).

Assume that if $h(i-1)> i$ or $h(i)> i+1$ for all $1 < i < n-1$.
We first show that $(p_1)_{\bullet},\ldots,(p_{n-1})_{\bullet}$ are nonsingular points of $\Hess(N,h)$.
Since we know that $p_1$ and $p_{n-1}$ are nonsingular points by Lemma~\ref{lem: nonsingular codim 1 cell 2}, 
it suffices to consider $p_i$ for $1< i< n-1$.
Then we have $h(i-1)>i$ or $h(i)>i+1$ by the assumption. This means that $p_i$ is defined as in case (i), (ii) or (iii). Hence, by Lemma~\ref{lem: nonsingular codim 1 cell}, it follows that $(p_i)_{\bullet}$ is a nonsingular point.
Hence, we proved that $(p_1)_{\bullet},\ldots,(p_{n-1})_{\bullet}$ are nonsingular points of $\Hess(N,h)$ as claimed above.

Now, it follows from Proposition~\ref{prop: sufficient condition on cell}~(1) that every point of the codimension 1 cell $\Hess(N,h)\cap \SC{p_i}$ is a nonsingular point of $\Hess(N,h)$ for $1\le i\le n-1$. Therefore, by Proposition~\ref{prop:codim1}, the codimension of the singular locus of $\Hess(N,h)$ is greater than or equal to $2$ which implies that $\Hess(N,h)$ is normal.

Next suppose that there exists $1 < i < n-1$ such that $h(i-1)=i$ and $h(i)=i+1$, and we show that $\Hess(N,h)$ is not normal.
In this case, the permutation $p_i$ is defined as in case (iv), that is, $p_i=w_{P_i}$. For our purpose, we prove that the Hessenberg complement $\hc{p_i}$ contains a lower diagonal full-string of height $n-2$:
\begin{align}\label{eq: proof normal 10}
 \{(n-1,1),(n,2)\} \subseteq \hc{p_i}.
\end{align}  
We prove this as follows.
Since $1 < i <n-1$, we see directly from the one-line notation of $p_i=w_{P_i}$ that $(p_i)^{-1}(n-1)= i+2$ and $(p_i)^{-1}(1)= i$. 
Hence the condition $h(i)=i+1<i+2$ and Lemma~\ref{lem:ui_not_n1 2} imply that $(n-1,1)\in \hc{p_i}$.
Similarly, we have $(p_i)^{-1}(n)= i+1$ and $(p_i)^{-1}(2)= i-1$, and hence the condition $h(i-1)=i<i+1$ implies that $(n,2)\in \hc{p_i}$. Thus we obtain \eqref{eq: proof normal 10}.
Since $\{(n-1,1),(n,2)\}$ is a lower diagonal full-string of height $n-2$, we may apply Proposition~\ref{prop: sufficient condition on cell} to see that every point of the Hessenberg-Schubert cell $\SC{p_i} \cap \Hess(N,h)$ is a singular point of $\Hess(N,h)$.
By Proposition~\ref{prop:codim1}, this cell $\SC{p_i}\cap \Hess(N,h)$ is codimension 1 in $\Hess(N,h)$.
Therefore $\Hess(N,h)$ is not normal as its singular locus is codimension $1$. 
\end{proof}

\vspace{12 pt}

We remark that Theorem~\ref{thm:normal} can be restated equivalently as  $\Hess(N,h)$ is not normal if and only if there exists some $1 <i < n-1$ with $h(i-1) = i$ and $h(i) = i+1$.
In particular, one recovers Kostant's result; Peterson variety (i.e.\ $\Hess(N,h)$ with $h(i)=i+1$ for $1\le i<n$) is not normal except for $n\le 3$ (\cite{Kos96}). 
Colloquially, 
Theorem~\ref{thm:normal} says that $\Hess(N,h)$ is not normal if there is some index $1< i < n-1 $ where the Hessenberg function looks like that of the Peterson variety.

\begin{example}
When $n=5$, the regular nilpotent Hessenberg varieties (with the assumption \eqref{eq: basis assumption}) which are not normal
are those corresponding to the following Hessenberg functions:
\begin{itemize}
    \item $h=(\textcolor{red}{\textbf{2}},\textcolor{red}{\textbf{3}},\textcolor{red}{\textbf{4}},5,5)$; $p_2=21543$ and $p_3=32154 $,
    \item $h=(3,\textcolor{red}{\textbf{3}},\textcolor{red}{\textbf{4}},5,5)$; $p_3=32154 $,
    \item $h=(\textcolor{red}{\textbf{2}},\textcolor{red}{\textbf{3}},5,5,5)$;  $p_2=21543 $,
\end{itemize}
where we highlight in red the pairs of indexes that offend the normality condition of Theorem~\ref{thm:normal}.
After each Hessenberg function, we list type (iv) permutations $p_i$ from Definition~\ref{def:codim1} corresponding to a singular cell of codimension 1 in $\Hess(N,h)$.  
\end{example}

We end this section with a proposition showing that every singular Hessenberg-Schubert cell in the Peterson variety corresponds to a permutation satisfying condition (2) of Proposition~\ref{prop: sufficient condition on cell}.  

\begin{proposition} \label{proposition:Peterson Strings} 
Suppose that $\Hess(N,h)$ is a Peterson variety $($i.e.\ $h(i)=i+1$ for $1\le i<n$$)$, and let $w_\bullet \in \Hess(N,h)$ be a permutation flag.
If $H_w^c$ contains a lower diagonal full-string, then it contains a lower diagonal full-string of height $n-1$ or $n-2$.
\end{proposition}

\begin{proof}
We know that the permutation flags in the Peterson variety correspond to the longest words in the parabolic subgroups of $\mathfrak{S}_n$ (e.g.\ \cite[Lemma~9]{IY12}). 
There is no lower-diagonal full string in $\hc{w_0}$ so the claim is trivial for the longest permutation $w_0$ (the case of codim 0). 
Since we have $h(i)=i+1$ for $1\le i<n$, the permutations $p_i=w_{P_i}$ are all defined as in case (iv).
By Lemma~\ref{lem: nonsingular codim 1 cell 2}, the permutation flags for $p_1$ and $p_{n-1}$ are nonsingular points of $\Hess(N,h)$, and hence $\hc{p_1}$ and $\hc{p_{n-1}}$ do not contain lower diagonal full-strings by Theorem~\ref{thm:fixed_point_singular}.
For $1<i<n-1$, the argument of the proof of Theorem~\ref{thm:normal} shows that the Hessenberg complement $\hc{p_i}$ for $1<i<n-1$ contains a lower diagonal full-string of height $n-2$.

So far we have proven the claim for each permutation whose Hessenberg-Schubert cell is of codimension $0$ and $1$.
The remaining permutations to check are those below some of the permutations $p_1,\ldots,p_{n-1}$ in Bruhat order. 

Suppose that  $w_\bullet \in \Hess(N,h)$ and $\Hess(N,h)\cap \SC{w}$ has codimension 2.
Then $w \leq p_i$ and $w \leq p_j$ for some $1 \leq i < j \leq  n -1$ and $w$ is the longest word in the parabolic subgroup 
 generated by $s_1,s_2,\ldots, s_{i-1}, s_{i+1}, s_{i+2},\ldots, s_{j-1},s_{j+1},\ldots, s_{n-1}$
and its one-line notation has the form
\begin{align*}
 w = i, i-1,\ldots, 1, j, j-1,\ldots, i+1, n , n-1, \ldots, j+1.
\end{align*}  
In other words, its permutation matrix is block diagonal and each block has ones along the reverse diagonal \cite[Section~2.1]{AHKZ21}.
Then $w^{-1}(n) = j+1 $ and $w^{-1}(1) =i $.  
Since $i < j$, we have $h(i)<j+1$.
Hence it follows that $(n,1)\in \hc{w}$ by Lemma~\ref{lem:ui_not_n1 2}.
This is a lower diagonal full string of height $n-1$.

If $w$ lies below more than two of the permutations $p_1,\ldots,p_{n-1}$ in Bruhat order, then the argument is essentially the same. 
Let $1 \leq i < j \leq  n -1$ be such $i$ is the smallest index with $w \leq p_i$ and let $j$ be the largest index such that $w \leq p_j$.
Then $w^{-1}(n) = j+1$ and $w^{-1}(1) =i $.  
Since $i < j$, we have $h(i)<j+1$, and hence it follows that $(n,1)\in \hc{w}$ by Lemma~\ref{lem:ui_not_n1 2}.
This is a lower diagonal full string of height $n-1$. 
\end{proof}

%\bigskip

\section{Future Directions and Open Problems}
We end this manuscript with a list of interesting open problems.
It is known that the singular locus of a Peterson variety is a union of Peterson-Schubert varieties. More precisely, when $h(i)=i+1$ for $1\le i<n$, $\Hess(N,h)$ is the Peterson variety, and \cite[Theorem 4]{IY12} implies that its singular locus is the union of Peterson-Schubert varieties $\overline{\Hess(N,h)\cap \SC{p_i}}$ for $1<i<n-1$.  
Recently, Escobar-Precup-Shareshian proved that the singular locus of nilpotent codimension one Hessenberg varieties are in fact nilpotent Hessenberg varieties of different Hessenberg functions (\cite{EPS22}); they are of course Hessenberg-Schubert varieties for the latter Hessenberg functions. These results motivate the following question.
 
\begin{openproblem}
When is the singular locus of a regular nilpotent Hessenberg variety isomorphic to a union of Hessenberg-Schubert varieties?
\end{openproblem}

We also suspect that the results in this paper may generalize naturally to wider classes of Hessenberg varieties. 

\begin{openproblem}
 Generalize Theorem~$\ref{thm:fixed_point_singular}$ and Theorem~$\ref{thm:normal}$ to other Lie types or to $($possibly non-nilpotent$)$ regular Hessenberg varieties.
\end{openproblem}

%\bigskip

\section{Acknowledgments} 
We thank Megumi Harada, Martha Precup, Jenna Rajchgot, and Nicholas Seguin for stimulating conversations and helpful comments.
In particular, we thank Jenna Rajchgot for suggesting that we apply our singularity criteria to classify normal regular nilpotent Hessenberg varieties.
We are also grateful to Patrick Brosnan, Megumi Harada, John Shareshian, Michelle Wachs for organizing the workshop \textit{Hessenberg Varieties in Combinatorics, Geometry and Representation Theory} in 2018 which made our collaboration possible, as well as Banff International Research Station for their hospitality.
This research is supported in part by Osaka City University Advanced Mathematical Institute (MEXT Joint Usage/Research Center on Mathematics and Theoretical Physics): the topology and combinatorics of Hessenberg varieties.
The first author is supported in part by JSPS Grant-in-Aid for Early-Career Scientists: 18K13413.

%\bigskip


\begin{thebibliography} {26}

\bibitem[ADGH18]{ADGH18}
Hiraku Abe, Lauren DeDieu, Federico Galetto, and Megumi Harada.
\newblock Geometry of {H}essenberg varieties with applications to
  {N}ewton-{O}kounkov bodies.
\newblock {\em Selecta Math. $($N.S.$)$}, 24(3):2129--2163, 2018.

\bibitem[AFZ20]{AFZ20}
Hiraku Abe, Naoki Fujita, and Haozhi Zeng.
\newblock Geometry of regular {H}essenberg varieties.
\newblock {\em Transform. Groups.} 25(2):305--333, 2020

\bibitem[AHHM17]{AHHM17}
Hiraku Abe, Megumi Harada, Tatsuya Horiguchi, and Mikiya Masuda.
\newblock The cohomology rings of regular nilpotent {H}essenberg
 varieties in {L}ie type {A}.
\newblock {\em Int. Math. Res. Not.} 2019(17) (2019), 5316--5388.


\bibitem[AHKZ21]{AHKZ21}
Hiraku Abe, Tatsuya Horiguchi, Hideya, Kuwata, and Haoshi Zeng.
\newblock  Geometry of {P}eterson {S}chubert calculus in type {A} and left-right diagrams.
\newblock  	arXiv:2104.02914. 

\bibitem[AT10]{AT10}
Dave Anderson and Julianna Tymoczko.
\newblock Schubert polynomials and classes of {H}essenberg varieties.
\newblock {\em J. Algebra}, 323(10):2605--2623, 2010.

\bibitem[Bal17]{Bal17}
Ana Balibanu.
\newblock The Peterson variety and the wonderful compactification.
\newblock {\em Represent. Theory}, 21, 132--150, 2017.

\bibitem[BL00]{BL00}
Sara Billey and V. Lakshmibai, Singular loci of Schubert varieties, Progr. Math.  \textbf{182.} Birkh\"{a}user, Boston, MA, 2000.

\bibitem[CHL21]{CHL21} 
Soojin Cho, Jaehyun Hong, and Eunjeong Lee.
\newblock   Permutation module decomposition of the second cohomology of a regular semisimple {H}essenberg variety.
\newblock   	arXiv:2107.00863.
 
\bibitem[DMPS92]{DMPS92}
Filippo De~Mari, Claudio Procesi, and Mark~A. Shayman.
\newblock Hessenberg varieties.
\newblock {\em Trans. Amer. Math. Soc.}, 332(2):529--534, 1992.

\bibitem[DMS88]{DMS88}
Filippo De~Mari and Mark~A. Shayman.
\newblock Generalized {E}ulerian numbers and the topology of the {H}essenberg
  variety of a matrix.
\newblock {\em Acta Appl. Math.}, 12(3):213--235, 1988.

\bibitem[Eis95]{Eis95}
David Eisenbud.
\newblock Commutative algebra with a view toward algebraic geometry. 
\newblock Graduate Texts in Mathematics, No. 150. Springer-Verlag, New York, 1995.

\bibitem[EPS22]{EPS22}
Laura Escobar, Martha Precup, John Shareshian.
\newblock  Hessenberg varieties of codimension one in the flag variety.
\newblock  arXiv:2208.06299.

\bibitem[HT17]{HT17}
Megumi Harada and Julianna Tymoczko.
\newblock Poset pinball, {GKM}-compatible subspaces, and {H}essenberg
             varieties,
\newblock {\em J. Math. Soc. Japan} 69(3), 2017.

\bibitem[Har77]{Har77}
Robin Hartshorne,
\newblock Algebraic geometry.
Graduate Texts in Mathematics, No. 52. Springer-Verlag, New York-Heidelberg, 1977.

\bibitem[I15]{I15}
 Erik Insko.
\newblock Schubert calculus and the homology of the Peterson variety.
\newblock { \em Electron. J. Combin.}, 22(2), \#P2.26, 2015.

 
\bibitem[IP19]{IP19}
 Erik Insko and  Martha Precup.
\newblock The singular locus of semisimple {H}essenberg varieties.
\newblock { \em J. Algebra}, 521, 65-96, 2019.

 
\bibitem[IT16]{IT16}
Erik Insko and Julianna Tymoczko.
\newblock Intersection theory of the {P}eterson variety and certain
  singularities of {S}chubert varieties.
\newblock {\em Geom. Dedicata}, 180:95--116, 2016.

\bibitem[ITW20]{ITW20}
Erik Insko, Julianna Tymoczko, and Alexander Woo.
\newblock A formula for the cohomology and {$K$}-class of a regular {H}essenberg variety.
\newblock {\em J. Pure Appl. Algebra }, 224: 5, 2020.

\bibitem[IY12]{IY12}
   E. Insko and A. Yong,  Patch ideals and Peterson varieties, 
   Transform. Groups \textbf{17} (2012), no. 4, 1011--1036. 

\bibitem[Kos96]{Kos96}
Bertram Kostant.
\newblock Flag manifold quantum cohomology, the {T}oda lattice, and the
  representation with highest weight {$\rho$}.
\newblock {\em Selecta Math. $($N.S.$)$}, 2(1):43--91, 1996.

\bibitem[Pet97]{Pet97}
 Dale Peterson.
 \newblock Quantum Cohomology of $G/P$, Lecture course.
 \newblock  Massachusetts Institute of Technology, Spring Term, 1997. 

\bibitem[Pre13]{Pre13}
 Martha Precup.
 \newblock Affine pavings of {H}essenberg varieties for semisimple
              groups
 \newblock {\em Selecta Math.},  19(4): 903-922, 2013. 

\bibitem[Pre16]{Pre16}
Martha Precup.
\newblock The {B}etti numbers of regular {H}essenberg varieties are
  palindromic.
\newblock {\em Transform. Groups}, 23(2):191--499, 2018.

\bibitem[Rie03]{Rie03}
Konstanze Rietsch.
\newblock Totally positive {T}oeplitz matrices and quantum cohomology of
  partial flag varieties.
\newblock {\em J. Amer. Math. Soc.}, 16(2):363--392, 2003.

\bibitem[ST06]{ST06}
Eric Sommers and Julianna Tymoczko.
\newblock Exponents for {$B$}-stable ideals.
\newblock {\em Trans. Amer. Math. Soc.}, 358(8):3493--3509, 2006.

\bibitem[Tym04]{Tym04}
	Julianna Tymoczko.
	\newblock Linear conditions imposed on flag varieties.
	\newblock  {\em American Journal of Mathematics}, 128(6): 1587--1604, 2006. .
	
\bibitem[Tym06]{Tym06}
	Julianna Tymoczko.
\newblock  Paving Hessenbergs by affines.
\newblock {\em Selecta Math. $($N.S.$)$}, 13:353--367, 2007.

\end{thebibliography}
\end{document}